\documentclass{article}

\usepackage[utf8]{inputenc}
\usepackage{geometry, mathtools, amsmath, amssymb,amsfonts, amsthm, cases, thmtools, tabularx, appendix, xcolor, bbm, titling, bm}
\usepackage{soul}
\usepackage{stmaryrd}
\usepackage{hyperref}
\usepackage{bbding}
\usepackage[maxbibnames=9]{biblatex}
\usepackage{tikz,lipsum,lmodern}
\usepackage{tkz-euclide}
\usepackage{float}
\usepackage[most]{tcolorbox}
\usepackage{amsthm,latexsym,amssymb,amsmath, verbatim, empheq, enumerate,bm}
\numberwithin{equation}{section}
\newcommand{\TV}{\mathrm{TV}}
\usepackage{makebox}
\renewbibmacro{in:}{%
  \ifentrytype{article}
    {}
    {\bibstring{in}%
     \printunit{\intitlepunct}}}
\DeclareFieldFormat{pages}{#1}
\renewcommand{\epsilon}{\varepsilon}

\newtheorem{theorem}{Theorem}[section]
\newtheorem{definition}[theorem]{Definition}

\newtheorem{proposition}[theorem]{Proposition}
\newtheorem{lemma}[theorem]{Lemma}
\newtheorem{corollary}[theorem]{Corollary}
\newtheorem{remark}[theorem]{Remark}
\usepackage[T1]{fontenc}
\usepackage{lmodern}

\newcommand{\EZ}[1]{\textcolor{magenta}{EZ : #1}}

\newcommand{\lr}[1]{\left( #1\right)}
\usepackage[normalem]{ulem}
\newcommand{\eq}[1]{\begin{equation}
\begin{split}
#1
\end{split}
\end{equation}}

\newcommand{\ep}{\varepsilon}

\usepackage{bbding}

\DeclareMathOperator*{\esssup}{ess\,sup}
\title{Microscopic derivation of a one-dimensional lubrication model with roughness}
\date{\today}
\begin{document}
\author{A. Lefebvre-Lepot\thanks{CNRS, Univ. Paris Saclay, Univ. Paris Cité, ENS Paris Saclay, SSA, INSERM, Centre Borelli, F-91190, Gif-sur-Yvette}, M.A. Mehmood\thanks{Imperial College London, London, United Kingdom; muhammed.mehmood21@imperial.ac.uk}, C. Perrin\thanks{Aix Marseille Univ, CNRS, I2M, Marseille, France; charlotte.perrin@cnrs.fr}, E. Zatorska\thanks{Mathematics Institute, University of Warwick, Zeeman Building, Coventry CV4 7AL, United Kingdom}}

\maketitle
\begin{abstract}
We derive a hydrodynamic model for the motion of inertial particles with a spherical hard core, interacting through lubrication forces and pairwise repulsive forces. The repulsion arises from the assumption that each particle is surrounded by a thin rough layer of reduced permeability. We prove that, as the number of particles tends to infinity (and their size tends to 0), the microscopic dynamics converges to a macroscopic hydrodynamic model in which congestion effects are encoded directly into the macroscopic interaction forces, depending on a local critical density transported by the flow.  In particular, we extend the work of Lefebvre-Lepot and Maury \cite{lefebvre2008micro} where non-inertial particles, submitted to only a lubrication force were considered, and present the convergence proof when inertial effects and roughness are taken into account. 
\end{abstract}

\section{Introduction}
The dynamics of dense suspensions of rigid particles immersed in a viscous fluid have attracted sustained attention across applied mathematics, physics, and engineering. At high particle concentrations, the behaviour of such systems is governed by strong short-range interactions that arise when neighbouring particles approach one another. The leading-order hydrodynamic effect in this regime is the {\emph{lubrication}} force, which is known to become singular as the interparticle distance tends to zero. These singular interactions prevent direct contact between smooth particles and strongly influence the macroscopic behaviour of the suspension.

A  mathematical description of this phenomenon was established in \cite{hillairet2015} by Hillairet and Kelaï at the microscopic level of the grains, and a first one-dimensional hydrodynamic limit was obtained by Lefebvre-Lepot and Maury \cite{lefebvre2008micro} for a system of non-inertial smooth particles interacting solely through lubrication forces.
Their macroscopic model, derived in the limit of large particle number (small size), captures the emergence of congestion through a singularity in the density-dependent  viscosity. This derivation and postulated ``hard-congestion'' limit, corresponding to the vanishing-viscosity regime, have since inspired numerous rigorous studies of the macroscopic  model of lubrication \cite{HCL,ChaMePeZa,perrin2018one}. 


However,  real particles are rarely perfectly smooth \cite{SmartLeighton1989,VinogradovaYakubov2006}. Experimental and theoretical studies have shown that surface roughness, polymer coatings, or thin low-permeability layers may allow physical contact to occur even when lubrication forces alone would forbid it, see for instance \cite{GV2012, Mongruel13} and the references therein. 
Such microscopic irregularities generate effective repulsive forces at small distances, often modelled through short-range potentials or through minimal admissible gaps \cite{Rognon2011,Gillissen2020,wachs2023modeling}. Incorporating these effects at the microscopic scale, and understanding their impact on the hydrodynamic limit is the main motivation for this work.

We derive a macroscopic hydrodynamic model starting from a one-dimensional microscopic system of inertial rigid spheres of radius 
$\ep$, which interact through both lubrication forces and pairwise repulsive forces. The repulsive forces model the presence of a thin rough or low-permeability layer surrounding each particle, which introduces a short-range repulsive mechanism penalizing interparticle distances $d < d^\star$ where $d^\star$ denotes a pair-dependent threshold distance defined in the next section. 
In the limit as the number of particles tends to infinity and $\ep \to 0$, this microscopic constraint gives rise to a macroscopic critical volume fraction (or, in this one-dimensional context, critical density) $\rho^\star$, transported by the flow, at which congestion effects induced by roughness become significant.


Our contribution is twofold.
First, we provide a unified microscopic description that couples lubrication forces with a repulsive potential, allowing for limited overlap of the effective rough surfaces. Second, and most significantly, we prove that in the hydrodynamic limit the empirical particle density, the velocity field, and the critical density converge (in a suitable weak sense) to a continuum system of conservation laws with nonlinear singular viscosity and a ``soft-congestion'' pressure. To the best of our knowledge, this is the first rigorous hydrodynamic limit that simultaneously incorporates singular lubrication forces, repulsive rough-layer interactions, and inertial effects. 

The resulting limit system includes three coupled equations: a continuity equation for the density  $\rho$,  a momentum equation  for the velocity $u$, and a transport equation for the critical density $\rho^{\star}$:
 \begin{equation}
\label{limit-NSE-intro}
    \left\{
    \begin{aligned}
        &\partial_{t} \rho + \partial_{x}(\rho u ) = 0, \\[1ex]
       &\partial_t(\rho u) + \partial_x (\rho u^2) - \partial_{x} \left( \frac{\mu}{1- \rho} \partial_{x} u \right) 
       + \partial_{x} \left( \frac{\rho}{\rho^{\star}} \right)^{\gamma} =  \rho f,\\[1ex] &\partial_{t} \rho^{\star} + u \partial_{x} \rho^{\star} = 0.
    \end{aligned}
    \right.
\end{equation}

The first two terms in the momentum equation come from the inertia of the particles at the microscopic level,
 the third term represents the singular viscosity induced by lubrication, and the fourth term is a soft-congestion pressure penalising densities exceeding $\rho^\star$, in the spirit of \cite{Lions1999}, and induced by the repulsive forces at the microscopic level; $f$ is the external force.
 Such penalisation (/congestion) pressures are indeed commonly used in a wide range of applications, including traffic and crowd dynamics \cite{degond2018transport, HCLMehmood},
 the motion of floating objects or flows in constricted geometries \cite{godlewski2018}, tumour growth and biological tissues 
 \cite{perthame2014, DS24, DPSV, ES25}, and are also widely implemented at the microscopic level of contact mechanics~\cite{johnson1987}.

 This macroscopic model and the associated hydrodynamic limit extend the framework of Lefebvre-Lepot and Maury by incorporating both repulsive interactions and inertial effects. 
 The presence of inertia notably complicates the analysis, particularly the compactness arguments necessary to obtain the identification of the limiting nonlinear convective and pressure terms.

The remainder of the paper is organised as follows. In Section \ref{Sec:micro} we introduce the microscopic model and describe the lubrication and repulsive forces in detail. Section \ref{sec:main-results} states the main convergence theorem. Section \ref{sec:global-NSE} is devoted to the analysis of the discrete system, including global-in-time existence and uniform bounds. In Section \ref{sec:limit} we derive the PDE representation and establish compactness properties sufficient for passing to the limit.  Our results are complemented by a series of numerical simulations carried out in Section \ref{sec:numerics}. 
Lastly, in Appendix~\ref{sec:extension} we discuss an extension of our main result to the case of partially congested initial data, and in Appendix~\ref{sec:AppB}, we detail the numerical solver used in section~\ref{sec:numerics}.

\section{The microscopic dynamics}\label{Sec:micro}

We consider a distribution of $N+1$ three-dimensional  particles $P_0, ..., P_N$ suspended in a viscous fluid, moving along an axis.
The particles have solid spherical cores, all with a radius $\varepsilon$, and are characterized by their positions $(q_i)_{i=0,\dots,N}$ along the axis and their velocities $(u_i)_{i=0,\dots,N}$, see Figure~\ref{fig:micro_distrib}. 
We suppose that the solid cores do not overlap, so that the distance between them is non-negative, that is, $d_i := q_{i} - q_{i-1} -2\varepsilon \geq 0$, for $i = 1, \dots, N$.
\begin{figure}[ht]
	\centering
	\includegraphics[scale= 1.2]{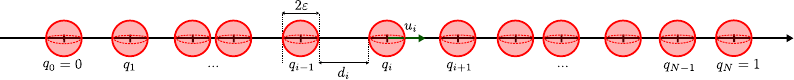}
	\caption{Discrete particles moving along the horizontal axis.}
	\label{fig:micro_distrib}
\end{figure}
We assume that the system lies in the interval $I=[0,1]$; the positions and velocities of the two extremal particles are fixed and equal to $(q_0,u_0) = (0,0)$, $(q_N, u_N) = (1,0)$, respectively. 

The positions and velocities of the other particles depend on the hydrodynamic forces, that is the interactions with the surrounding fluid.
In this system, there are a number of different physical effects that contribute to the total hydrodynamic forces exerted on the particles such as drag, lift (in the case of two dimensions) or Basset forces (which essentially arise due to temporal delay in boundary layer development as the relative velocity changes with time). In this work, we focus on the modelling of the lubrication forces which are due to short-range hydrodynamic interactions between close particles. The importance of lubrication increases with the particle volume fraction as the average distance between particles decreases and eventually becomes the dominant term in hydrodynamic interactions for dense suspensions \cite{wachs2023modeling}.

\subsection{Description of forces}
As established in \cite{cox1974motion}, when the distance between two neighbouring spheres $i$  and $j$  goes to zero ($j=i-1$ or $j=i+1$ in our configuration), the leading term in the asymptotic expansion of the lubrication force exerted on particle $i$ by the fluid in the narrow gap between the two spheres is:
\[
F_{j \to i} = -  2 \pi \nu \frac{\epsilon_{i}^2 \epsilon_j^2}{(\epsilon_{i} + \epsilon_j)^2} \frac{u_{i} - u_{j}}{d_{ij}},
\]
where $\nu$ is the viscosity of the interstitial fluid, $d_{ij}$ is the distance between the two particles, and $(\epsilon_i,u_i)$ and $(\epsilon_{j},u_j)$ are the radii and velocities of particles $i$ and $j$ respectively. Therefore, in the configuration described above, the asymptotic expansions of the two lubrication forces exerted on particle $i$ by its neighbours are
\begin{equation} \label{Fi-intro}
    F_{i-1 \to i} = -  \mu(2\epsilon)^{2} \frac{u_{i} - u_{i-1}}{d_{i}} \quad \text{and}\quad F_{i+1 \to i} = -  \mu(2\epsilon)^{2} \frac{u_{i} - u_{i+1}}{d_{i+1}}
\end{equation} where $\mu := 2 \pi \nu / 4$ is a constant introduced for convenience.

If one considers two particles at distance $d$, the lubrication force scales as $-\phi(d) \dot d$, where the primitive of the function $d\to\phi(d)=1/d$ blows up as $d\to 0$. By exploiting this property, applying Newton’s second law, and invoking the Cauchy-Lipschitz theorem, one can show that such forces prevent the contact between particles  in finite time. It is worth noting that the lubrication force places the system at the threshold of contact: any weaker force -- such as $\phi(d)=1/d^{1-\eta}$ with $\eta > 0$ -- would allow particles to come into contact.
This non-contact property is specific to smooth particles and was also obtained for the full coupled Stokes/rigid particles system~\cite{hesla2004,hillairet2007}. 
In contrast, rough (non-smooth) particles can experience contact. Experimental measurements indicate that a rough sphere behaves equivalently to a smooth sphere that is slightly smaller than the rough one \cite{SmartLeighton1989, VinogradovaYakubov2006}. Similarly, from a mathematical standpoint, it was established in~\cite{GV2012} that the lubrication force exerted on a sphere by a corrugated wall behaves as if it were exerted by a shifted smooth plane.
As a consequence, while the lubrication prevents contact between the smooth surfaces, it allows  for contact between the rough surfaces. To model this behaviour, we consider that $\epsilon$ is the radius of the equivalent smooth spheres and that each particle $P_{i}$ has a roughness radius $r_{i}$ (see Figure~\ref{fig:rough_sphere}).
\begin{figure}[ht]
	\centering
	\includegraphics[scale= 1.9]{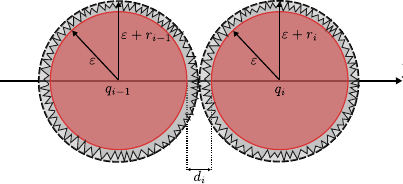}
	\caption{Two rough solid particles }
	\label{fig:rough_sphere}
\end{figure}

For each pair of neighbouring particles, $i-1$ and $i$, we introduce a critical distance $d_i^\star$ for which the rough surfaces touch:
\begin{equation}
    d_{i}^{\star} := r_{i-1} + r_{i},
\end{equation}
so that the distance between the particles should remain greater than this value,  if the rough areas are impermeable. We denote by $\mathbf{d}^{\star} = (d_{1}^{\star}, d_{2}^{\star}, ..., d_{N}^{\star})$ the set of critical distances. The non-overlapping property is relaxed by considering a collection of regularized repulsive contact forces $\mathbf{G}^{\epsilon} = (G_{0\to1}, ~... ~, G_{(N-1)\to N})$, where $G_{(i-1)\to i}$ represents the force exerted by particle $i-1$ onto particle $i$. Conversely, particle $i-1$ undergoes the force $-G_{(i-1)\to i}$ from particle $i$, where
\begin{equation} \label{Gi-intro}
G_{(i-1) \to i}= (2\epsilon)^{2}\left( \frac{d_{i}^{\star} + 2\epsilon}{d_{i} + 2\epsilon} \right)^{\gamma},
\end{equation}
for some $\gamma \ge 1$. The factor of $(2\epsilon)^{2}$ reflects the modelling assumption that the interaction force is proportional to the size of the interaction surface seen by the neighbouring particles. Note that, when $\gamma$ is given, this force is non stiff when $d_i$ goes to $d_i^\star$, so that it allows overlapping of the roughness areas ($d_i$ smaller than $d_i^\star$). This can be understood as the possibility for the rough surface to deform when they get in contact. When $\gamma$ increases, the repulsive force is more and more stiff and finally behaves as a contact force between rigid rough surfaces, imposing $d_i>d_i^\star$. Note that such a relaxed repulsive model can also be considered for repulsive colloidal particles or for active entities that would repulse its neighbours.

Finally, we consider an external force $F^{ext}_i(t)$ exerted on each particle, which is  derived from a force density $f_{3d}(t,x,y,z)$:
\[
F^{ext}_i(t)=\int_{P_i^\epsilon(t)} f_{3d}(t,x,y,z)\,dx\,dy\,dz = \int_{q_i-\epsilon}^{q_i+\epsilon} \left(\int_{D_i^\epsilon(t,x)} f_{3d}(t,x,y,z)\,dy\,dz \right)\, dx,
\]
where $P_i^\epsilon$ is the 3-dimensional domain covered by particle $i$ at time $t$ and $D_i^\epsilon(t,x)$ is a slice of this domain at abscissa $x$. The point $x$ being given, the inner integral scales as $\epsilon^2$ when $\epsilon$ goes to zero, so that  we write the following model for the external force:
\[
F^{ext}_i(t) = (2\epsilon)^2 \int_{q_i-\epsilon}^{q_i+\epsilon} f(t,x) dx,
\]
 where $f$ is given and will be assumed to be Lipschitz in space and time.
If we now define a mean linear density at time $t$ on particle $i$ by:
\begin{equation}\label{eqfi-nse}
\bar f_i^\epsilon(t) = \frac{1}{2\epsilon} \int_{q_i-\epsilon}^{q_i+\epsilon} f(t,x) \,dx
\end{equation}
we obtain
\[
F^{ext}_i(t) = (2\epsilon)^3 \bar f_i^\epsilon(t).
\]

\subsection{The microscopic model}
We now give a full mathematical description of the microscopic problem for positive times. Consider $N+1$ particles on $I = [0,1]$, each of them of radius $\epsilon$ (later, $\ep$ will depend on $N$), where initially the positions, velocities and critical distances are denoted by $ (\mathbf{q}^\epsilon_{0} , \mathbf{u}_0^\epsilon, \mathbf{d}^{\star,\epsilon }_0)$, where
\begin{equation}
    \begin{aligned}
        &\mathbf{q}_{0}^\epsilon  = (q_{0,1}^\epsilon , ..., q_{0,N-1}^\epsilon ), \\[1ex]
        &\mathbf{u}_{0}^\epsilon = (u_{0,1}^\epsilon , ..., u_{0,N-1}^\epsilon ), \\[1ex]
        &\mathbf{d}^{\star, \epsilon}_0 = (d_{0,1}^{\star,\epsilon},...,d_{0,N-1}^{\star,\epsilon}), 
    \end{aligned}
\end{equation}
and for positive times the positions, velocities and critical distances are expressed as
\begin{equation}
    \begin{aligned}
        &\mathbf{q}^\epsilon(t)  = (q_{1}^\epsilon(t) , ..., q_{N-1}^\epsilon(t) ), \\[1ex]
        &\mathbf{u}^\epsilon(t) = (u_{1}^\epsilon(t) , ..., u_{N-1}^\epsilon(t) ), \\[1ex]
        &\mathbf{d}^{\star, \epsilon}(t) = (d_{1}^{\star, \epsilon}(t),...,d_{N-1}^{\star, \epsilon}(t)) = (d_{0,1}^{\star, \epsilon},...,d_{0,N-1}^{\star, \epsilon}).
    \end{aligned}
\end{equation}
Notice that each $d^{\star, \epsilon}_{i}$ is constant in time, since it depends only on the roughness profile of the pair of particles $P_{i-1}$ and $P_i$. We assume the initial positions $\mathbf{q}_{0}^\epsilon$ are such that no two particles are in contact, i.e. that 
\begin{equation}\label{eq:distrib_init}
    d_{0,i}^\epsilon:= q_{0,i}^\epsilon - q_{0,i-1}^\epsilon  - 2\epsilon > 0
\end{equation}
for each $i$, and that the first and last particles are fixed at $x=0$ and $x=1$ respectively. Taking into account the interaction forces described above, the equation of motion at time $t$ for particle $P_{i}$, for each $i=1,..., N-1$, is given by
\begin{equation} \label{balance-orig-intro}
  m \ddot{q}_i^\epsilon(t) =  F_{(i+1) \to i}(t) + F_{(i-1) \to i}(t) + G_{(i+1) \to i}(t) + G_{(i-1)\to i}(t) + F_{i}^{ext}(t),
\end{equation}  where $m := (2\epsilon)^{3}$ is the volume of each particle.
For the sake of simplicity we will adopt the notation $G_i^\epsilon(t)=G_i^\epsilon[\mathbf{q}^\epsilon(t),\mathbf{d}^{\star,\epsilon}(t)]$, where
\begin{equation} \label{Gi-def}
    G_{i}^\epsilon(t) := (2\epsilon)^{-2}G_{(i-1) \to i}(t) = \left( \frac{d_{i}^{\star, \epsilon} + 2\epsilon}{d_{i}^\epsilon(t) + 2\epsilon} \right)^{\gamma}.
\end{equation} 
Then using the symmetry of the lubrication and interaction forces and dividing \eqref{balance-orig-intro} by $(2\epsilon)^2$, we arrive at the simplified balance of forces for particle $P_i$:
\begin{equation} \label{balance-intro}
  2\epsilon \ddot{q}_i^\epsilon(t) =  \mu\left( \frac{u_{i+1}^\epsilon(t) - u_{i}^\epsilon(t)}{d_{i+1}^\epsilon(t)} 
 - \frac{u_{i}^\epsilon(t) - u_{i-1}^\epsilon(t)}{d_{i}^\epsilon(t)} \right) + \left( G_{i}^\epsilon(t) - G_{i+1}^\epsilon(t)\right) +  2\epsilon \bar{f}_i^\epsilon(t). 
\end{equation}
It will also be useful to express the balance of forces in matrix form. We can define $\mathbb{A}(\mathbf{q}^\epsilon(t) )$ to be the $(N-1) \times (N-1)$ tri-diagonal stiffness matrix \begin{equation} \label{A-matrix}
    \mathbb{A}(\mathbf{q}^\epsilon(t) ) = \mu\begin{pmatrix}
\frac{1}{d_1^\epsilon(t)} + \frac{1}{d_2^\epsilon(t)} & -\frac{1}{d_2^\epsilon(t)}  & 0 & \cdots & 0  \\
-\frac{1}{d_2^\epsilon(t)} & \frac{1}{d_2^\epsilon(t)} + \frac{1}{d_3^\epsilon(t)} & -\frac{1}{d_3^\epsilon(t)}  & \cdots & 0  \\
0 & \ddots & \ddots & \ddots & \vdots \\
\vdots & \vdots & - \frac{1}{d_{N-2}^\epsilon(t)} & \frac{1}{d_{N-2}^\epsilon(t)} + \frac{1}{d_{N-1}^\epsilon(t)} & -\frac{1}{d_{N-2}^\epsilon(t)} \\
0 & 0 & \cdots & -\frac{1}{d_{N-1}^\epsilon(t)} & \frac{1}{d_{N-1}^\epsilon(t)} + \frac{1}{d_N^\epsilon(t)} \\
\end{pmatrix}.
\end{equation} 
We also define the vectors \begin{align} \label{b-vector}
    &\mathbf{b}(\mathbf{q}^\epsilon(t) , \mathbf{d}^{\star,\epsilon }) =(b_{i}^\epsilon(t))_{i=1}^{N-1}:= (G_{1}^\epsilon(t)  - G_{2}^\epsilon(t) , G_{2}^\epsilon(t) -G_{3}^\epsilon(t) , ..., G_{N-1}^\epsilon(t)  - G_{N}^\epsilon(t) ),\\[1ex]
    &\mathbf{f}(t,\mathbf{q}^\epsilon(t))  = 2\epsilon( \bar{f}_1^\epsilon(t), ..., \bar{f}_{N-1}^\epsilon(t)). \label{eq:bff}
\end{align}
Then the equation of motion for $P_1, ..., P_{N-1}$ \eqref{balance-intro}
 can be rewritten as
 \begin{equation}\label{balance-matrix-intro}
     2\ep \ddot{\mathbf{q}}^\epsilon(t)=-\mathbb{A}(\mathbf{q}^\epsilon(t) )\dot{\mathbf{q}}^\epsilon(t)  +    \mathbf{b}(\mathbf{q}^\epsilon(t) , \mathbf{d}^{\star,\epsilon })+\mathbf{f}(t,\mathbf{q}^\epsilon(t) ).
 \end{equation}
The initial distances $d_{0,i}^\epsilon$ being positive \eqref{eq:distrib_init}, we have by continuity of the distances that the matrix \eqref{A-matrix} is well-defined at least on a short time interval. In summary, a solution to the microscopic model is a mapping $t \mapsto(\mathbf{q}^\epsilon(t), \mathbf{u}^\epsilon(t), \mathbf{d}^{\star,\epsilon})$ that solves the ODE
 \begin{equation}
\label{ODE_NS}
    \left\{
    \begin{aligned}
       & \dot{\mathbf{q}}^\epsilon  = \mathbf{u}^\epsilon , \\[1ex]
       &2\epsilon \dot{\mathbf{u}}^\epsilon  = -\mathbb{A}(\mathbf{q}^\epsilon )\mathbf{u}^\epsilon  +    \mathbf{b}(\mathbf{q}^\epsilon , \mathbf{d}^{\star,\epsilon })+\mathbf{f}(t,\mathbf{q}^\epsilon ), \\[1ex]
        &\dot{\mathbf{d}}^{\star,\epsilon } = 0,
    \end{aligned}
    \right.
\end{equation} 
with some initial data
\begin{equation}\label{initial_NS}
    \left\{
    \begin{aligned}
    &\mathbf{q}^\epsilon(0)=\mathbf{q}_{0} = (q_{0,i}^\epsilon)_{i=1}^{N-1}, \\
    &\mathbf{u}^\epsilon(0) = \mathbf{u}_{0} = (u_{0,i}^\epsilon)_{i=1}^{N-1},\\
    &\mathbf{d}^{\star,\epsilon}(0)=\mathbf{d}^{\star}_0 = (d_{0,i}^{\star,\epsilon})_{i=1}^{N-1}.
    \end{aligned}
    \right.
 \end{equation}


 Finally, let us also note that if two particles are in contact initially (i.e. $d_{0,i}^\epsilon = 0$ for some $i=1,\ldots,N$), then one needs to re-interpret the balance of forces \eqref{balance-intro}. For now, we will not assume that any two particles can be in contact initially. This case will be discussed in the Appendix.

\section{The main convergence result} \label{sec:main-results}
Our main result states that for any well prepared initial data, we can construct weak solutions to the macroscopic system~\eqref{limit-NSE-intro} as the limit of a microscopic approximation satisfying \eqref{ODE_NS}.

\subsection{The macroscopic model} 
We are interested in the asymptotic limit $N\to\infty$ of system \eqref{ODE_NS} with initial data \eqref{initial_NS} satisfying~\eqref{eq:distrib_init}. We impose the scaling $\ep(N) \propto N^{-1}$, so that the size particles $\epsilon=\epsilon(N) \to 0$.  In this limit, we expect that a solution $(\mathbf{q}^\epsilon, \mathbf{u}^\epsilon, \mathbf{d}^{\star,\epsilon})_\epsilon$ to \eqref{ODE_NS} converges, in some sense, to a weak solution $(\rho, u, \rho^{\star})$ of the macroscopic system \eqref{limit-NSE-intro} considered on $(0,T)\times I$ and supplemented with Dirichlet boundary conditions for the velocity.

\begin{definition}  \label{def:nse-soln}
Let  $f \in W^{1,\infty}([0,T] \times I)$, $\rho_{0},\rho^\star_0 \in [L^{\infty}(I)]^2$ with $\rho_0 \in [0,1)$,  $\rho^\star_0 \in [0,1]$ and $u_0 \in H^1_0(I)$. We say that the triple $(\rho,u,\rho^\star)$ is a weak solution to problem \eqref{limit-NSE-intro}  with zero Dirichlet boundary conditions for $u$ and initial conditions
    \begin{equation}
\label{initial-semi}
    \left\{
    \begin{aligned}
    &\rho(0,\cdot) = \rho_{0},\\
    &u(0,\cdot) = u_0,\\
       &\rho^\star(0,\cdot) = \rho^\star_{0},
    \end{aligned}
    \right.
\end{equation}
if:
\begin{itemize}
\item $\rho \in [0,1),$ $\rho^{\star} \in [0,1]$ and the triple $(\rho, u, \rho^{\star})$ belongs to the following regularity class
\begin{align}
     (\rho, u, \rho^{\star})\in  C([0,T]; L^{\infty}(I)) \times L^{2}(0,T; H^{1}_{0}(I)) \times C([0,T];L^{\infty}(I)),
    \end{align} with \begin{equation}
        \rho u \in C_{weak}([0,T]; L^2(I));
    \end{equation}
    \item the continuity equation 
    
\begin{equation} \label{WFcontinuity}
   \int_0^T \!\!\!\int_I \rho \partial_{t}\phi(t,x)\,dx\,dt + \int_{I} \rho_{0}(x)\phi(0,x)\,dx + \int_0^T \!\!\!\int_I \rho u \partial_{x}\phi(t,x)\,dx\,dt = 0, 
\end{equation}
   holds for all $\phi \in C^1_c([0,T)\times \bar I)$;
   \item  the momentum equation
 \begin{equation} \label{WFmom}
 \begin{aligned}
    &  \int_0^T \!\!\!\int_I \rho  u \partial_t \phi(t,x)\,dx\,dt+\int_I  \rho_0 u_0(x)\phi(0,x)\,dx  + \int_0^T \!\!\!\int_I \rho u^2 \partial_x \phi (t,x)\,dxds  \\
     &\qquad-  \int_0^T \!\!\!\int_{I} \frac{\mu}{1-\rho} \partial_{x}u ~\partial_{x}\phi (t,x)\,dx\,dt + \int_0^T \!\!\!\int_{I} \left(\frac{\rho}{\rho^{\star}}  \right)^{\gamma} \partial_{x}\phi (t,x)\,dx\,dt = -\int_0^T \!\!\!\int_{I}\rho f \phi (t,x)\,dx\,dt, \end{aligned}
    \end{equation}
holds for all $\phi \in C^1_c([0,T)\times I)$;
   \item the transport equation 
   \begin{equation} \label{WFtransport}
   \int_0^T \!\!\!\int_I \rho^{\star} \partial_{t}\phi(t,x)\,dx\,dt + \int_{I} \rho^{\star}_{0}(x)\phi(0,x)\,dx + \int_0^T \!\!\!\int_I \rho^{\star} \partial_{x}(u \phi)(t,x)\,dx\,dt = 0, 
\end{equation}
holds for all $ \phi \in C^1_c([0,T)\times \bar I)$.
\end{itemize}
\end{definition}

Before stating the theorem more precisely, we first define the auxiliary approximate functions connecting the discrete quantities $(\mathbf{q}^\epsilon, \mathbf{u}^{\epsilon}, \mathbf{d}^{\star,\epsilon})_{\epsilon>0}$ to their macroscopic limits $(\rho, u, \rho^{\star})$. 

\subsection{\texorpdfstring{Approximation of initial data $(\rho_0, u_0, \rho^{\star}_0)$}{Approximation of initial data (rho0, u0, rho*0)}}
First, given the initial data of the microscopic system $(\mathbf{q}_0^\epsilon, \mathbf{u}_0^\epsilon, \mathbf{d}_0^{*,\epsilon})$, the approximate initial density $\rho_0^\epsilon(x)=\rho_0^\epsilon[\mathbf{q}_0^\epsilon]$ and critical density $\rho_0^{\star, \epsilon}(x) = \rho_0^{\star, \epsilon}[\mathbf{q}_0^\epsilon,\mathbf{d}^{\star,\epsilon}_0]$ are defined as (see Figure \ref{fig:rho-ss})
 \begin{align}
                &\rho^\epsilon_0(x) := \sum_{i=1}^{N}\frac{2\epsilon}{d_{0,i}^\epsilon+2\epsilon} \mathbf{1}_{[q_{0,i-1}^\epsilon, q_{0, i}^\epsilon)}(x),  \label{ID-rho0-NSE} \\[1ex] 
                &\rho_0^{\star,\epsilon}(x) := \sum_{i=1}^{N}\frac{2\epsilon}{d_{0,i}^{\star,\epsilon}+2\epsilon}\mathbf{1}_{[q_{0,i-1}^\epsilon, q_{0, i}^\epsilon)}(x).  \label{ID-rho0st-NSE}
\end{align}
 The macroscopic representation of the initial velocities $\mathbf{u}_0^\epsilon$ is denoted by $u_0^\epsilon = u_0^\epsilon[\mathbf{q}^\epsilon_0,\mathbf{u}^\epsilon_0]$, and defined as (see Figure \ref{fig:u-ss})
\begin{equation} \label{ID-u0-NSE}
            u^\epsilon_0(x) :=  \sum_{i=1}^{N} \left[u_{0,i-1}^\epsilon + \left( \frac{u_{0,i}^\epsilon - u_{0,i-1}^\epsilon}{q_{0,i}^\epsilon - q_{0,i-1}^\epsilon} \right)(x-q_{0,i-1}^\epsilon) \right] \mathbf{1}_{[q_{0,i-1}^\epsilon, q_{0,i}^\epsilon)}(x).
\end{equation}

\subsection{\texorpdfstring{Approximation of solutions $(\rho, u, \rho^{\star})$}{Approximation of solutions (rho, u, rho*) to the microscopic model for t≥0}}

We assume that $f$ is Lipschitz in both variables and $d_{0,i}^\epsilon>0$~\eqref{eq:distrib_init}, and seek a solution to the ODE \eqref{ODE_NS}. Note that such a solution $t \mapsto \mathbf{q}^\epsilon(t)$ is necessarily continuous and so the distances $d_i(t)$ must be positive on some short time interval, meaning that the matrix \eqref{A-matrix} is well-defined at least for short time. Therefore we can safely apply the Cauchy-Lipschitz theorem to \eqref{ODE_NS} to obtain the existence of $T^\star_\epsilon > 0$ such that we have a unique solution $ t \mapsto (\mathbf{q}^\epsilon(t), \mathbf{u}^\epsilon(t), \mathbf{d}^{\star, \epsilon}(t))_{\epsilon > 0}$ on $(0,T^\star_\ep)$, where
 \begin{equation}
     \begin{aligned}
         &\mathbf{q}^\epsilon(t) := (q_{1}^\epsilon(t), ..., q_{N-1}^\epsilon(t)), \\[1ex]
         &\mathbf{u}^{\epsilon}(t) := (u_1^\epsilon(t), ..., u_{N-1}^\epsilon(t)),
     \end{aligned}
 \end{equation} 
 and $\mathbf{d}^{\star,\epsilon}(t) = \mathbf{d}^{\star,\epsilon}_0$, since $\dot{\mathbf{d}}^{\star,\epsilon}(t) = 0$ from \eqref{ODE_NS}. 
 Then we can extend the definitions of $\rho_0^\epsilon$ and $\rho_0^{\star,\epsilon}$ to positive times. 
 The density $\rho^\epsilon = \rho^{\epsilon}[\mathbf{q}^{\epsilon}]$ is defined as
  \begin{align} \label{rho-NSE}
    \rho^{\epsilon}(t,x) := \sum_{i=1}^{N} \rho_{i}^{\epsilon}(t) \mathbf{1}_{[q_{i-1}^{\epsilon}(t), q_{i}^{\epsilon}(t))}(x), \hspace{30pt} \rho_{i}^{\epsilon}(t) := 1 - \frac{d_{i}^{\epsilon}(t)}{q_{i}^{\epsilon}(t)-q_{i-1}^{\epsilon}(t)} = \frac{2\epsilon}{d_{i}^{\epsilon}(t) + 2\epsilon},
\end{align} 
where the distances $t \mapsto d_i^{\epsilon}(t)$ are defined as 
\begin{equation} \label{def_di}
    d_i^{\epsilon}(t) := q_{i}^{\epsilon}(t) - q_{i-1}^{\epsilon}(t)-2\epsilon,\quad \text{for each }i=1,...,N-1,
\end{equation}
and similarly for the critical density $\rho^{\star,\epsilon} = \rho^{\star, \epsilon}[\mathbf{q}^{\epsilon}, \mathbf{d}^{\star,\epsilon}_0]$,
\begin{align} \label{rhost-NSE}
    \rho^{\star, \epsilon}(t,x) := \sum_{i=1}^{N} \rho_{i}^{\star, \epsilon}(t) \mathbf{1}_{[q_{i-1}^{\epsilon}(t), q_{i}^{\epsilon}(t))}(x), \hspace{30pt} \rho_{i}^{\star, \epsilon}:=  \frac{2\epsilon}{d_{0,i}^{\star,\epsilon} + 2\epsilon}.
\end{align}
\begin{figure}[ht]
    \centering
    \includegraphics[width=1\linewidth]{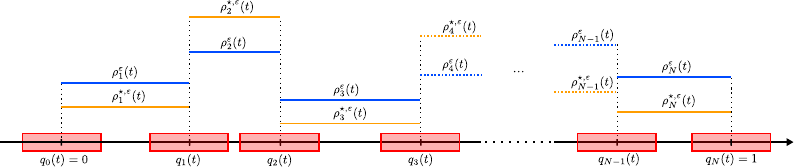}
    \caption{$\rho^\epsilon$ and $\rho^{\star,\epsilon}$ at time $t \ge0$}
    \label{fig:rho-ss}
\end{figure}
 Finally, we associate with $\mathbf{u}^{\epsilon}[\mathbf{q}^\epsilon, \mathbf{d}^{\star, \epsilon}, \mathbf{f}^\epsilon] = (u_1^\epsilon, ..., u_{N-1}^\epsilon)$   the piecewise linear function $u^\epsilon=u^\epsilon[\mathbf{q}^{\epsilon},\mathbf{u}^{\epsilon}]: [0,T_\epsilon^\star] \times I \to \mathbb{R}$  such that
 \begin{equation}  \label{u-NSE}
     u^\epsilon(t,x) = \sum_{i=1}^{N} \left[u_{i-1}^\epsilon(t) + \left( \frac{u_{i}^\epsilon(t) - u_{i-1}^\epsilon(t)}{q_{i}^\epsilon(t) - q_{i-1}^\epsilon(t)} \right)(x-q_{i-1}^\epsilon(t)) \right] \mathbf{1}_{[q_{i-1}^\epsilon(t), q_{i}^\epsilon(t))}(x).
 \end{equation}

\begin{figure}[ht]
    \centering
    \includegraphics[width=1\linewidth]{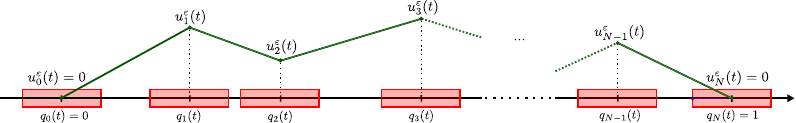}
    \caption{$u^\epsilon$  at time $t\ge0$}
    \label{fig:u-ss}
\end{figure} 

As a first step, we will prove in Section~\ref{sec:extension-NSE} that this local solution can be extended to a global one, namely that $d_i^\epsilon >0$ on any interval $(0,T)$, $T > 0$.

\subsection{The main theorem}
We are now ready to state our main result.
\begin{theorem} \label{thm:nse}
    Let $T>0$ and let $f \in W^{1,\infty}([0,T] \times I)$, $(\rho_0, \rho_0^\star) \in [W^{1,\infty}(I)]^2$ satisfying $\rho_0 \in [\delta, \overline{\rho}]$ and $\rho^\star_0 \in [\delta, 1]$ for some $0 < \delta < \overline{\rho} < 1$ be given. Moreover, suppose we have $u_0 \in H^1_0(I)$. \\
    
    1. There exists a sequence $(\mathbf{q}^{\epsilon}_{0},\mathbf{u}_0^\epsilon, \mathbf{d}^{\star, \epsilon}_0)_{\epsilon >0}$ satisfying the following properties:
    \begin{enumerate}
        \item[(i)] There exist positive constants $c_0, C_0$ independent of $\epsilon$ such that  for each $i=1,...,N$,
\eq{ \label{nse-dist-assumption}
        0 < c_0 \epsilon \le & \,\,d_{0,i}^{\epsilon} \le C_0\epsilon, \hspace{5pt} \hspace{5pt} 0 \, \leq \, d_{0,i}^{\star, \epsilon} \le C_0\epsilon,}
        and, for  each $i=1,...,N-1$,
        \eq{\label{nse-di-diff-assumption}
                    &|d_{0,i+1}^\epsilon - d_{0,i}^\epsilon| \le C_0\epsilon^2, \\[1ex]
                    &|d_{0,i+1}^{\star} - d_{0,i}^{\star}|\le C_0 \epsilon^2.
    }

        \item[(ii)] We have 
           \begin{equation} \label{nse-conv-assumption}
            \rho^{\epsilon}_0 \to \rho_0 \text{ and } \rho^{\star, \epsilon}_0 \to \rho^{\star}_0 \text{ in } L^{\infty}(I), \end{equation} where $\rho^\epsilon_0, \rho^{\star,\epsilon}_0$ are the approximate initial densities and critical densities respectively, defined by \eqref{ID-rho0-NSE} and \eqref{ID-rho0st-NSE} respectively.
        \item[(iii)] We have \begin{equation} \label{u0-NSE}
            u^\epsilon_0 \to u_0 \text{ in } C(I),
        \end{equation} where $u^\epsilon_0$ is defined by \eqref{ID-u0-NSE}
        and for each $i= 1,... ,N-1$,
        \begin{equation} \label{nse-u0-condition}
        |u_{0,i}| \le \|u_0\|_{L^\infty}. 
        \end{equation}
\end{enumerate}

 2.  There exists a unique solution $t \mapsto (\mathbf{q}^{\epsilon}(t), \mathbf{u}^\epsilon(t), \mathbf{d}^{\star, \epsilon}(t))$ to the ODE \eqref{ODE_NS}-\eqref{initial_NS} on $(0,T)$ with initial data $(\mathbf{q}_0^\epsilon, \mathbf{u}_0^\epsilon,\mathbf{d}_0^{\star,\epsilon})_\epsilon$ satisfying \eqref{nse-dist-assumption}-\eqref{nse-u0-condition} and $\mathbf{f}$ given by \eqref{eqfi-nse}, \eqref{eq:bff}.
     The macroscopic representation of this solution, $(\rho^\epsilon, u^\epsilon,\rho^{\star,\epsilon})$  given by \eqref{rho-NSE}, \eqref{rhost-NSE} and \eqref{u-NSE}, satisfies
     \begin{equation}
         \begin{aligned}
             &\|\rho^\epsilon\|_{BV((0,T) \times I)} + \|\partial_t\rho^\epsilon\|_{L^2(0,T; H^{-1}(I))} \le C(T), \\[1ex]
             &\|u^\epsilon\|_{L^2(0,T; H^1_0(I))} \le C(T), \\[1ex]
             &\|\rho^{\star,\epsilon}\|_{L^\infty(0,T; BV(I))} + \|\partial_t\rho^{\star,\epsilon}\|_{L^2(0,T; H^{-1}(I))} \le C(T).
         \end{aligned}
     \end{equation}

     3. The system \eqref{limit-NSE-intro} with zero Dirichlet boundary conditions and initial conditions $(\rho_0, u_0,
  \rho_0^\star)$ admits a weak solution $(\rho, u, \rho^\star)$ in the sense of Definition \ref{def:nse-soln}, such that  for $p \in [1,\infty)$,  up to a subsequence,
  \begin{equation}
      \begin{aligned}
                &\rho^\epsilon \to \rho \text{ strongly in }  C([0,T]; L^p(I)), \\[1ex]
          &u^{\epsilon} \rightharpoonup u ~ \text{ weakly in } L^{2}(0,T; H_0^1(I)), \\[1ex]
          &\rho^{\star,\epsilon} \to \rho^{\star} \text{ strongly in }   C ([0,T]; L^{p}(I)),
      \end{aligned}
  \end{equation}  
  where
  \begin{equation*}
      \rho, \rho^\star \in L^\infty_t BV_x \cap BV_{t,x},\quad   0 < \rho < 1, ~ 0 < \rho^\star \le 1.
  \end{equation*}
\end{theorem}
\begin{remark}
Throughout the paper, the limit of sequences indexed by  $\epsilon $ should be understood as the limit when $N\to\infty$ where the associated particle radius $\epsilon(N)\to0$. For notational simplicity, we will always write this parameter simply as $\epsilon$ even though it is implicitly tied to $N$.
\end{remark}

\begin{remark}
    As we will see from the proof of Theorem \ref{prop:global-nse}, the density $\rho^{\epsilon}(t,\cdot) < 1$ for all $t>0$ whenever $\rho^{\epsilon}(0,\cdot) < 1$ (i.e. no particles will come into contact in finite time if no two particles are in contact initially). This is because the lubrication force prevents collisions from forming - see Proposition \ref{prop:global-nse} for a quantification of this statement.
\end{remark}


\begin{remark}[Notation]
It is important to note that each of the microscopic quantities we have defined so far ($q_i, u_i, d_i$ and $d_i^\star$) depend on $\epsilon$, but we will drop this dependence from now, for the sake of brevity. 

The superscript $\epsilon$ will only be used for the vector quantities $(\mathbf{q}^\epsilon, \mathbf{u}^\epsilon, \mathbf{d}^{\star,\epsilon})$ and those that depend on space. For instance, we will write $\rho^\epsilon$ since it varies spatially, but not $q_i^\epsilon(t)$; instead, we write $q_i(t)$. Furthermore, we will use the notation $X_t Y_x$ to denote the Bochner space $X(0,T; Y(I))$ for suitable Banach spaces $X$ and $Y$.
\end{remark}

The proof of Theorem \ref{thm:nse} is conducted in Sections \ref{sec:global-NSE} and \ref{sec:limit}. We first analyze the microscopic ODE system \eqref{ODE_NS} providing the results for its local and global solvability in Section \ref{sec:global-NSE}. Following that, in Section, \ref{sec:limit} we turn to the macroscopic functions $\rho^\epsilon, u^\epsilon, \rho^{\star,\epsilon}$ (originally introduced in \eqref{rho-NSE},  \eqref{u-NSE} and \eqref{rhost-NSE} respectively) and obtain a PDE representation of the microscopic system. In Subsection \ref{sec:uniform-NSE} we derive some uniform in $\epsilon$ estimates for the macroscopic quantities. Much of the difficulty lies in obtaining sufficient bounds on the macroscopic velocity $u^\epsilon$ and its gradient.
 The uniform estimates 
 allow us eventually to take the limit $\epsilon \to 0$ in Subsection \ref{sec:limit-NSE}, leading to system \eqref{limit-NSE-intro}. 

\section{Analysis of the microscopic system}
 \label{sec:global-NSE}



\subsection{Construction of initial data and existence of a local-in-time solution} \label{sec: construction-NSE}

The following proposition shows that from well prepared macroscopic initial data $(\rho_0, u_0,\rho_0^\star)$ we can construct discrete initial data satisfying \eqref{nse-dist-assumption}-\eqref{nse-u0-condition}.
%
\begin{proposition}\label{prop:ID_micro}
Let $(\rho_0, u_0,\rho_0^\star)\in W^{1,\infty}(I)\times H^1_0(I)\times W^{1,\infty}(I)$ satisfying $\rho_0 \in [\delta, \overline{\rho}]$ and $\rho^\star_0 \in [\delta, 1]$ for some $0 < \delta < \overline{\rho} < 1$ be given. We denote $M_0 = \int_0^1 \rho_0(x) \ dx$ the total mass.
Then there exist  $(\mathbf{q}_0^\epsilon, \mathbf{u}_0^\epsilon, \mathbf{d}^{\star,\epsilon}_0)_\epsilon$ (or equivalently $(\mathbf{d}_0^\epsilon, \mathbf{u}_0^\epsilon, \mathbf{d}^{\star,\epsilon}_0)_\epsilon$)
a number of particles $N  \propto \ep^{-1}$ depending on $\rho_0$,
and some constants $c_0, C_0 > 0$ depending on $\delta, \bar \rho$, but independent of $\ep$, such that   the following estimates hold
\begin{align} 
0 < c_0\ep \leq d_{0,i} \leq C_0\ep&, \qquad 0 \leq d_{0,i}^\star \leq C_0\ep, \label{eq:boundd0}\\
|d_{0,i+1} - d_{0,i}| &\le \ C_0 \epsilon^2, \label{di-diff-t=0}\\
|d_{0,i+1}^{\star} - d_{0,i}^{\star}| &\le C_0 \epsilon^2,   \label{dist-diff-t=0}\\
|u_{0,i}| &\le \|u_0\|_{L^\infty}\label{eq:boundu0i}
\end{align} 
for all $i =1, \dots, N-1$.
Moreover, defining
\begin{align}
&\rho^\epsilon(0,x) := \sum_{i=1}^{N}\frac{2\epsilon}{d_{0,i}+2\epsilon} \mathbf{1}_{[q_{0,i-1}, q_{0, i})}(x),  \label{ID-rho0-NSE-bis} \\[1ex] 
&\rho^{\star,\epsilon}(0,x) := \sum_{i=1}^{N}\frac{2\epsilon}{d_{0,i}^{\star}+2\epsilon}\mathbf{1}_{[q_{0,i-1}, q_{0, i})}(x), \label{ID-rho0st-NSE-bis}\\[1ex]
&u^\ep(0,x):= \sum_{i=1}^{N} \left[u_{0,i-1}^\ep + \frac{u_{0,i}^\ep -u_{0,i-1}^\ep}{q_{0,i} -q_{0,i-1}} (x-q_{0,i-1})\right] \mathbf{1}_{[q_{0,i-1},q_{0,i})}(x),\label{ID-u-NSE-bis}
\end{align}
we have the following convergences:
\begin{align} 
&\rho^\epsilon(0,\cdot) \to \rho_0 \quad \text{and} \quad  \rho^{\star,\epsilon}(0,\cdot) \to \rho_0^\star \quad \text{in } L^\infty(I), \label{conv-rho0}\\
&u^\epsilon(0, \cdot) \to u_0 \text{ in } C(I).\label{u0-NSE-bis}
\end{align} 
\end{proposition}
\begin{proof}
To begin, for fixed $N\gg 1$, we take $\ep=\ep(N)=\frac{M_0}{2N}$ and set $q_{0,0} = 0$. We then find $q_{0,1}$ such that
\[
\int_{0}^{q_{0,1}} \rho_0(x) dx = 2\ep.
\]
The existence of $q_{0,1} \in (0,1)$ is guaranteed by the positivity of $\rho_0$ and the fact that the total mass is larger than $2\ep$ for sufficiently large $N$.
Similarly, we construct recursively $q_{0,i}$ for $i=2, \dots, N$, such that
\begin{equation}\label{q-choice}
\int_{q_{0,i-1}}^{q_{0,i}} \rho_0(x) dx = 2\ep.
\end{equation}
Note that our construction ensures that $q_{0,N}=1$, and that
$\rho_{0,i}$, $i= 1, \dots,N$, is equal to the mean value of $\rho_0$ on $(q_{0,i-1} , q_{0,i})$. Indeed, from the definition of $\rho_{0,i}$ \eqref{rho-NSE} we have
\begin{equation}\label{rhoirho}
   \rho_{0,i}   =\frac{2\epsilon}{d_{0,i} +2\epsilon}
   = \frac{2\epsilon}{q_{0,i}-q_{0,i-1}}  = \frac{1}{q_{0,i}-q_{0,i-1}} \int_{q_{0,i-1} }^{q_{0,i}} \rho_0(x) \,dx.
\end{equation}

Let us also note that using the lower bound $\rho_0 \ge \delta$ the equality \eqref{q-choice} leads to the bound 
 \begin{equation} \label{qi-diff}
     q_{0,i}  - q_{0,i-1}  \le \frac{2\epsilon}{\delta}.
 \end{equation}
or in other words, using that $d_{0,i} = q_{0,i} - q_{0,i-1}-2\epsilon$
\begin{equation} \label{q-bound}
    d_{0,i}  \le 2\epsilon \left( \frac{1}{\delta} - 1 \right) =: C_0\epsilon.
 \end{equation} 
  Similarly we can use the upper bound $\rho_0 \le \overline{\rho}$ to get 
 \begin{equation} \label{di-LB}
     d_{0,i}  \ge 2\epsilon \left( \frac{1}{\overline{\rho}} - 1 \right) =: c_0\epsilon > 0.
 \end{equation}
Having fixed positions $q_{0,i}$ and $\rho_0^\star$ being given, we choose $d^{\star}_{0,i}$ for $i = 1,..., N$ so that 
$\rho^\star_{0,i}$ is equal to the mean value of $\rho_0^\star$ on $(q_{0,i-1},q_{0,i})$:
\begin{equation} \label{q-st-choice}
    \frac{1}{q_{0,i}  - q_{0,i-1} } \int_{q_{0,i-1}}^{q_{0,i}} \rho_0^{\star}(x)\,dx 
    = \frac{2\epsilon}{d_{0,i}^{\star}+2\epsilon}.
\end{equation} 
%
As above we can use the upper/lower bounds $\delta <\rho_0^{\star} \leq 1$ to show that this implies \begin{equation} \label{q-st-bound}
    0 \leq  d_{0,i}^{\star}   \le 2\epsilon \left( \frac{1}{\delta} - 1 \right)=C_0\epsilon.
\end{equation} 
%
 Next,  fixing $x \in [q_{0,i-1} , q_{0,i} )$  for 
 $i\in\{1,\ldots,N\}$, we observe that for $\rho^\epsilon(0, \cdot)$ defined in \eqref{ID-rho0-NSE-bis} thanks to \eqref{rhoirho}, we have
\begin{equation} \label{r-eps-bound}
    \begin{aligned}
         |\rho^{\epsilon}(0,x) - \rho_0(x)| 
        &= \left|\frac{1}{q_{0,i} -q_{0,i-1} } \int_{q_{0,i-1} }^{q_{0,i} } \rho_0(y)~dy - \rho_0 (x) \right| \\[1ex] 
        & \leq \frac{1}{q_{0,i}  - q_{0,i-1} } \int_{q_{0,i-1} }^{q_{0,i} } \left|\rho_0 (y)-\rho_0 (x)\right|~dy \\[1ex] 
        & \le L_0  |q_{0,i}  - q_{0,i-1} | 
        \le L_0 \frac{2\epsilon}{\delta},
    \end{aligned} 
\end{equation}
where $L_0$ is the Lipschitz constant of $\rho_0$ and we have used \eqref{qi-diff}. 
Therefore, $\rho^\ep(0, \cdot)$ converges to $\rho_0$ in $L^\infty(I)$.

Similarly, using $\rho_{0,i} \geq \delta$, we have for $i=1,..,N-1$ that
\begin{align*}
|d_{0,i+1} - d_{0,i} |
& = \dfrac{2\ep}{\rho_{0,i} \rho_{0,i+1}} |\rho_{0,i} - \rho_{0,i+1}| \\
& \leq \dfrac{2\ep}{\delta^2} \left| \frac{1}{q_{0,i+1}-q_{0,i}} \int_{q_{0,i}}^{q_{0,i+1}} \rho_0(x)\,dx - \rho_0(q_{0,i}) 
	- \frac{1}{q_{0,i}-q_{0,i-1}}\int_{q_{0,i-1}}^{q_{0,i}} \rho_0(x)\,dx + \rho_0(q_{0,i})\right| \\[1ex] 
&\le \dfrac{2\ep}{\delta^2} \Bigg[ \frac{1}{q_{0,i+1}-q_{0,i}} \int_{q_{0,i}}^{q_{0,i+1}} \left| \rho_0(x) - \rho_0(q_{0,i})\right|\,dx      \\[1ex]     
& \hspace{2cm} 	+  \frac{1}{q_{0,i}-q_{0,i-1}} \int_{q_{0,i-1}}^{q_{0,i}} \left|\rho_0(x) - \rho_0(q_{0,i})\right|\,dx \Bigg]   \\[1ex] 
&\le \dfrac{2\ep}{\delta^2} \left[\frac{L_0}{q_{0,i+1}-q_{0,i}} \int_{q_{0,i}}^{q_{0,i+1}} \left| x-q_{0,i}\right|\,dx  
	+  \frac{L_0}{q_{0,i}-q_{0,i-1}} \int_{q_{0,i-1}}^{q_{0,i}} \left|x-q_{0,i}\right|\,dx\right] \\[1ex]
&\le \dfrac{2\ep}{\delta^2} \left[ L_0 \left( |q_{0,i+1} - q_{0,i}| + |q_{0,i} - q_{0,i-1}| \right) \right] \\
& \le  8\dfrac{L_0}{\delta^3}\epsilon^2 \leq C_0\epsilon^2,
\end{align*}
by using \eqref{qi-diff}. 

Repeating similar estimates for $\rho^{\star,\epsilon}(0,x)$ (defined by \eqref{ID-rho0st-NSE-bis}) gives us that for $x \in (q_{0,i-1} , q_{0,i} )$, 
\begin{equation*}
     |\rho^{\star,\epsilon}(0,x) - \rho^{\star}(0,x)| \le L_0^\star \frac{2\epsilon}{\delta},
\end{equation*}  
and 
\begin{equation*}
|d_{0,i+1}^\star - d_{0,i}^\star| \leq \dfrac{8L^\star_0}{\delta^3} \ep^2 \leq C_0\epsilon^2
\end{equation*}
where $L_0^{\star}$ is the Lipschitz constant of $\rho_0^{\star}$. 
Since these bounds are independent of $i$ we can conclude that \eqref{conv-rho0} holds. 
%

To define the approximation of initial velocity, we simply set 
\begin{equation}
\mathbf{u}_0^\epsilon = (u_{0,1}, \dots, u_{0,N-1}) \quad \text{with} \quad u_{0,i} := u_0(q_{0,i}), \text{ for each } i=1,...,N-1.
\end{equation} 
Since $u_0\in L^\infty(I)$, we have for free \eqref{eq:boundu0i} and a uniform control in $L^\infty$ on $u^\ep(0, \cdot)$ (defined by \eqref{ID-u-NSE-bis}). 
Moreover, using the regularity $u_0 \in H^1_0$ and the Cauchy-Schwarz inequality, we have for $x \in (q_{0,i-1} , q_{0,i} )$:
\begin{equation} \label{u0-conv-proof}
\begin{aligned}
| u^\ep(0,x) - u_0(x)| 
&\le  | u^\ep(0,x) - u_0(q_{0,i-1} )|+ |u_0(q_{0,i-1} ) - u_0(x)| \\[1ex] 
&\le \left| u_0(q_{0,i} ) - u_0 (q_{0,i-1} ) \right| \frac{x-q_{0,i-1} }{q_{0,i}  - q_{0,i-1} } + \|u_0\|_{H^1} |q_{0,i-1}  -x|^{1/2} \\[1ex] 
&\le  2 \|u_0\|_{H^1} |q_{0,i}  -q_{0,i-1}|^{1/2} \\[1ex] 
&\le C\epsilon^{1/2}.
\end{aligned}
\end{equation}
 As a consequence, we can show that the sequence $(u^\ep(0, \cdot))_\ep$ is uniformly-equi-continuous. Indeed, since $u^0$ is uniformly-continuous, for any $\alpha$ there exists $\eta > 0$ such that if $x,y\in (q_{0,i-1} , q_{0,i} )$ satisfy $|x-y| \leq \eta$ then $|u_0(x) - u_0(y)| \leq \alpha$ and 
\begin{align*}
|u^\ep(0, x) - u^\ep(0,y)|
& \leq |u^\ep(0,x) - u_0(x)| + |u^\ep(0,y) - u_0(y)| + |u_0(x) - u_0(y)| \\
& \leq C \ep^{1/2} + \alpha \leq 2\alpha,
\end{align*}
for sufficiently small $\ep$. Since moreover, by definition $\|u^\ep(0,\cdot)\|_{L^\infty} \leq \|u_0\|_{L^\infty}$, we can apply Arzel\`a–Ascoli's theorem and obtain the uniform convergence \eqref{u0-NSE-bis}. 
\end{proof}
%




{\it Local-in-time existence of the discrete solution.}~~
We have now shown that we can construct initial data $(\mathbf{q}^{\epsilon}_{0},\mathbf{u}_0^\epsilon, \mathbf{d}^{\star,\epsilon}_0)_{\epsilon >0}$ satisfying \eqref{nse-dist-assumption}-\eqref{u0-NSE}. 

The right-hand side of the ODE \eqref{ODE_NS} is Lipschitz in $U := (\mathbf{q},\mathbf{u}, \mathbf{d}^\star)$. 
Indeed, the external force $f$ is assumed to be Lipschitz in $x$, implying that $\mathbf{f}$ is Lipschitz in $\mathbf{q}$; $\mathbf{b}$ is Lipschitz in $\mathbf{q}$ and $\mathbf{d}^\star$ in view of~\eqref{Gi-def},\eqref{b-vector} for non-negative $d_i^\star, d_i$, and finally $\mathbb{A}$ defined by~\eqref{A-matrix} is Lipschitz in $\mathbf{q}$ for positive $d_i$'s. In particular, our construction ensures that $d_{0,i} > c_0 \ep,\  d_{0,i}^\star \geq 0$.
Therefore, we can apply the Cauchy-Lipschitz theorem to the ODE \eqref{ODE_NS} complemented by the initial datum $(\mathbf{q}^{\epsilon}_{0},\mathbf{u}_0^\epsilon, \mathbf{d}^{\star,\epsilon}_0)$  to obtain, for fixed $\ep> 0$, a solution $(\mathbf{q}^\epsilon(t), \mathbf{u}^\epsilon(t), \mathbf{d}^{\star, \epsilon}(t))$, with  $\mathbf{d}^{\star, \epsilon}(t) = \mathbf{d}^{\star, \epsilon}_0$, on some  time interval $(0,T^\star_\epsilon)$ with $T^\star_\epsilon>0$ small enough.  
Moreover, on that time interval, we ensure that $d_i(t) > 0$.

\subsection{Extension to a global solution } 
\label{sec:extension-NSE}
In this subsection we aim to show that the solution to the ODE \eqref{ODE_NS} is in fact global-in-time, i.e. we can take $T^\star_\epsilon = T$ for any $T>0$.
To do so, we will show that, at $\ep$ fixed, the distances $d_i$ remain ``far'' from $0$.
To begin, let us assert an $L^\infty$ bound in time on the interaction potential $G_i = G_i[\mathbf{q}^\ep, \mathbf{d}^{\star, \epsilon}]$ defined by \eqref{Gi-def}.
\begin{lemma} 
	Let $\epsilon>0$ be given and consider $(\mathbf{q}^\epsilon, \mathbf{u}^\epsilon,\mathbf{d}^{\star,\epsilon})$ the local solution to the ODE \eqref{ODE_NS} constructed previously on $[0,T_\epsilon^\star]$. We have, for each $i=1,...,N$
	\begin{equation}
	\|G_i\|_{L^{\infty}(0,T^\star_\epsilon)}  \le C,
	\end{equation} 
	for some $C > 0$ independent of $T^\star_\epsilon$ and $\ep$.
    \label{lemma:boundGi}
\end{lemma}

\begin{proof}
    Using the bounds~\eqref{eq:boundd0}, we have for each $i=1,...,N$ that
\begin{equation} 
\begin{aligned}\label{Geps-bound-NSE}
   G_i= \left( \frac{d_i^\star + 2\epsilon}{d_i(t)+2\epsilon} \right)^\gamma\le \left( \frac{d_i^\star + 2\epsilon}{2\epsilon} \right)^\gamma \le \left( \frac{C_0+2}{2} \right)^\gamma.
   \end{aligned}
\end{equation} 
\end{proof}
Next, we carry out a discrete energy estimate.
\begin{lemma}[Discrete energy estimate] For any $\epsilon >0$ and $t < T^\star_\epsilon$, 
    \begin{equation} \begin{aligned}\label{energy-NSE}
   \epsilon\sum_{i=1}^{N-1}   (u_i )^2 (t) +&   \frac{1}{\gamma-1} \sum_{i=1}^N   (d_i(t)+2\epsilon) G_i+\frac{\mu}{4}\sum_{i=1}^N\int_0^t \frac{|u_{i} (\tau) - u_{i-1} (\tau)|^2}{d_i(\tau)}~d\tau \\[1ex] 
   &\le  \epsilon\sum_{i=1}^{N-1}   (u_i )^2 (0) + \frac{1}{\gamma-1}\sum_{i=1}^N  (d_i(0)+2\epsilon) G_i+ \frac{1}{\mu}\|f\|_{L^2((0,t); L^1(I))}^{2} D_N(0) \\[1ex]
   & =: \mathcal{E}_0 + \frac{1}{\mu}\|f\|_{L^2((0,t); L^1(I))}^{2} D_N(0).\end{aligned}
    \end{equation}
where we have defined $D_N(t) := \sum_{i=1}^{N}d_i(t)$ which is independent of $t$ and $\epsilon$.
\end{lemma}
\begin{proof}
Fix $t \in (0,T^\star_\epsilon)$. Multiplying the balance of forces~\eqref{balance-intro} by $u_i $ and summing over $i=1,...,N-1$ gives 
	\begin{align}
 \epsilon \sum_{i=1}^{N-1} \frac{d}{dt} (u_i(t))^2  
 & =  \mu \sum_{i=1}^{N-1} \left( \frac{u_{i+1}(t) - u_{i}(t)}{d_{i+1}(t)}
- \frac{u_{i}(t) - u_{i-1}(t)}{d_{i}(t)} \right) u_i(t) \notag\\
& \qquad  + \sum_{i=1}^{N-1} \left( G_{i}(t) - G_{i+1}(t)\right)u_i(t) + 2\epsilon \sum_{i=1}^{N-1} \bar f_i(t)  u_i(t).	
\end{align}
Now since the first and final particles are  fixed at $x=0,1$ respectively, their velocities are $0$, i.e. $u_0(t)  = u_N(t)  = 0$ \footnote{There is an ambiguity with the notation $u_0$ here which should be clarified. Here, $u_0$ represents the velocity of particle $0$ at time $t$, but previously (particularly in Section \ref{sec: construction-NSE}) $u_0$ was the initial data for the velocity.} for all $t\ge0$, summation by parts tells us that 
\begin{align*}
  & \sum_{i=1}^{N-1} \left( \frac{u_{i+1}(t) - u_{i}(t)}{d_{i+1}(t)}
  - \frac{u_{i}(t) - u_{i-1}(t)}{d_{i}(t)} \right) u_i(t)  \\
  &=  u_N   \frac{u_{N}(t) - u_{N-1}(t)}{d_{N}(t)} - u_1(t)   \frac{u_{1}(t) - u_{0}(t)}{d_{1}(t)} - \sum_{i=1}^{N-1}  \frac{u_{i+1}(t) - u_{i}(t)}{d_{i+1}(t)} (u_{i+1} (t) - u_i (t)) \\[1ex] 
  &= - \frac{|u_1(t)-u_0(t)|^2}{d_1(t)} - \sum_{i=1}^{N-1} \frac{|u_{i+1}(t)  - u_{i}(t) |^2}{d_{i+1}(t)} \\[1ex] 
  &= -\sum_{i=1}^N\frac{|u_{i}(t)  - u_{i-1}(t) |^2}{d_i(t)}.
\end{align*} 
In a similar way, using again $u_0(t)  = u_N(t)  = 0$, we have for the interaction term 
\begin{align*}
& \sum_{i=1}^{N-1} \left( G_{i}(t) - G_{i+1}(t)\right)u_i(t) \\
& = \sum_{i=1}^{N-1} G_{i}(t)(u_{i}(t) - u_{i-1}(t)) - G_N(t)u_{N-1}(t) \\[1ex] 
& = \sum_{i=1}^{N} G_i(t)(u_{i}(t)-u_{i-1}(t)).
\end{align*} 
Now, since $u_i(t) - u_{i-1}(t) = \frac{d}{dt}(d_i(t)+2\epsilon)$, 
\begin{align*} 
     G_i(t)(u_{i}(t)-u_{i-1}(t))
     & =  \left( \frac{d_i^\star + 2\epsilon}{d_i(t)+2\epsilon} \right)^\gamma  \frac{d}{dt}(d_i(t)+2\epsilon)  \\[1ex]
     & = \frac{d}{dt} \left( \frac{(d_i^\star + 2\epsilon)^\gamma }{(1-\gamma)(d_i(t)+2\epsilon)^{\gamma-1}} \right) \\
     & =  \frac{d}{dt} \left( \frac{d_i(t) + 2\epsilon}{1-\gamma} G_i(t) \right). 
\end{align*} 
Therefore, 
\begin{equation} \begin{aligned}\label{Gi-dot}
     \sum_{i=1}^{N-1} \left( G_{i}(t) - G_{i+1}(t)\right)u_i(t)
      = -\sum_{i=1}^N  \frac{d}{dt} \left( \frac{d_i + 2\epsilon}{\gamma-1} G_i (t)\right).
     \end{aligned}
\end{equation}

For the remaining term of the estimate, we use the following ``discrete Poincaré''  inequality (using the boundary condition $u_0= 0$): 
\begin{equation} \label{ui-triangle}
\max_i |u_i(t)| \leq \sum_{k=1}^{N} |u_k(t) - u_{k-1}(t) |,  
\end{equation}
and deduce that 
\begin{equation}
\begin{aligned}
      \left|2\epsilon \sum_{i=1}^{N-1} \bar{f}_i(t) u_i (t) \right|
      &\leq 2\epsilon \sum_{i=1}^{N-1} |\bar{f}_i(t) ||u_i (t)|\\
      & \leq  2\ep \max_i |u_i(t)| \sum_{i=1}^{N-1} |\bar{f}_i(t)| \\
      &= \max_i |u_i(t)| \sum_{i=1}^{N-1}\left|\int_{q_{i}(t)-\epsilon}^{q_i(t)+\epsilon} f(t,x)\,dx\right|\\
      &\leq \max_i |u_i(t)| \sum_{i=1}^{N-1}\int_{q_{i}(t)-\epsilon}^{q_i(t)+\epsilon} |f(t,x)|\,dx \\
      &=  \max_i |u_i(t)| \|f(t, \cdot)\|_{L^1_x} \\[1ex] 
       & \leq \|f(t, \cdot)\|_{L^1_x}\sum_{i=1}^{N} |u_i(t) - u_{i-1}(t)|.
\end{aligned}
\end{equation}

We now multiply and divide by $\sqrt{2d_i(t)/\mu}$ to get
\begin{align*} \label{fu-energy-est} 
     2\epsilon \sum_{i=1}^{N-1} \bar{f}_i(t) u_i (t) 
     &\le \|f(t,\cdot)\|_{L^1(I)}\sum_{i=1}^{N} \dfrac{\sqrt{\mu}}{\sqrt{2}}\frac{|u_i(t) - u_{i-1}(t)|}{\sqrt{d_i(t)}} \dfrac{\sqrt{2}}{\sqrt{\mu}}\sqrt{d_i(t)} \\[1ex] 
     &\le \frac{\mu}{4} \sum_{i=1}^{N} \frac{|u_i(t) - u_{i-1}(t)|^2}{d_i(t)} + \frac{1}{\mu} \|f(t, \cdot)\|_{L^1_x}^2 D_N(0). 
\end{align*}
 All in all, we have for any $\epsilon > 0$ and $t \in (0,T^\star_\epsilon):$ 
 \begin{equation} \label{energy-ddt}
      \epsilon\sum_{i=1}^{N-1} \frac{d}{dt} (u_i )^2 (t) +  \sum_{i=1}^N \frac{d}{dt}\left(  \frac{d_i(t) + 2\epsilon}{\gamma-1} G_i(t) \right) +  \frac{\mu}{4}\sum_{i=1}^N\frac{|u_{i} (t) - u_{i-1} (t)|^2}{d_i(t)}   \le \frac{1}{\mu}   \|f(t,\cdot)\|_{L^{1}_x}^{2} D_N(0).
\end{equation} 
Integrating in time, we arrive at \eqref{energy-NSE} \begin{equation} \label{integrated-est-NSE}
    \begin{aligned}
         \epsilon\sum_{i=1}^{N-1}   (u_i )^2 (t) +&   \sum_{i=1}^N   \frac{d_i(t) + 2\epsilon}{\gamma-1} G_i(t) +\frac{\mu}{4}\sum_{i=1}^N\int_0^t \frac{|u_{i} (\tau) - u_{i-1} (\tau)|^2}{d_i(\tau)}~d\tau \\[1ex]
          &\le    \epsilon\sum_{i=1}^{N-1}   (u_i )^2 (0) + \sum_{i=1}^N   \frac{d_i(0) + 2\epsilon}{\gamma-1} G_i(0) + \frac{1}{\mu}\|f\|_{L^2_tL^{1}_{x}}^{2} D_N(0).
    \end{aligned}
\end{equation}\end{proof}

\begin{corollary}[Bounds on $u^\epsilon$] \label{u-bounds-NSE}
    Using the notations of the previous lemma, we have for each $i=1,...,N-1$ the following bound on the components of the velocity vector $\mathbf{u}^\ep$:
    \begin{equation} \label{ui-L1-NSE}
        \|u_i\|_{L^1(0,T^\star_\ep)} \le  \mathcal{E}_0 + \left(\|f\|_{L^2(0,T^\star_\ep;L^1(I))}^2 + t\right) \dfrac{D_N(0)}{\mu},
    \end{equation}
    and on the macroscopic velocity:
    \begin{equation} \label{dxu-L2-NSE}
      \|\partial_x u^\epsilon\|_{L^{2}((0,T^\star_\ep) \times I)}^2 \le \frac{4}{\mu} \left(\mathcal{E}_0 + \dfrac{1}{\mu}\|f\|_{L^2(0,T^\star_\ep;L^1(I))}^2 D_N(0)\right).
    \end{equation}
\end{corollary}
\begin{proof}
Using again a discrete Poincaré inequality, we have thanks to \eqref{energy-NSE}, for each $i = 1, \dots , N-1$
\begin{equation} \begin{aligned}
   \int_0^t |u_i(\tau)|~d\tau &\le  \sum_{k=1}^N \int_0^t |u_k(\tau) - u_{k-1}(\tau) | ~d\tau \\[1ex]
    &\le \frac{\mu}{4} \sum_{i=k}^{N} \int_0^t \frac{|u_k(\tau) - u_{k-1}(\tau)|^2}{d_i(\tau)}~d\tau + \frac{1}{\mu}  D_N(0) t\\
    & \le \mathcal{E}_0 + \dfrac{1}{\mu}\|f\|_{L^2(0,t;L^1(I))}^2 D_N(0) + \frac{t}{\mu} D_N(0). 
\end{aligned}
\end{equation}
For $\partial_x u^\epsilon$, we can use the definition of $u^\epsilon$ from \eqref{u-NSE} to get for $t>0$, \begin{equation} \begin{aligned}
    \|\partial_x u^\epsilon(t,\cdot)\|_{L^{2}(I)}^2 &= \sum_{i=1}^{N-1} \int_{q_{i-1}(t)}^{q_i(t)}  \left|\frac{u_{i}  (t) - u_{i-1}(t)}{q_{i}(t)-q_{i-1}(t)}\right|^2\,dx = \sum_{i=1}^{N-1} \frac{|u_{i}  (t) - u_{i-1}(t)|^2}{q_{i}(t)-q_{i-1}(t)}. 
    \end{aligned}
\end{equation} 
Therefore, using $ q_i(t) - q_{i-1}(t) \geq d_i(t)$, and integrating the above expression in time, we get from \eqref{energy-NSE}
\begin{equation} 
      \|\partial_x u^\epsilon\|_{L^{2}((0,T^\star_\ep) \times I)}^2 \le \sum_{i=1}^{N-1} \int_0^{T^\star_\ep} \frac{|u_{i}  (t) - u_{i-1}(t)|^2}{d_{i}(t)}~dt \le \frac{4}{\mu} \left(\mathcal{E}_0 + \dfrac{1}{\mu}\|f\|_{L^2(0,T^\star_\ep;L^1(I))}^2 D_N(0)\right).
\end{equation} 
\end{proof}

 \begin{proposition}[Upper and lower bound on the distances $d_i$] \label{prop:global-nse}
 For any $t \in (0,T^{\star}_\epsilon)$ and any index $i \in \{1,...,N\}$, 
    \begin{equation} \label{di(t)-NSE}
0< c_1(t)\epsilon\le d_{i}(t) \le C_2(t)\epsilon,
    \end{equation} 
    with $c_1(t), C_2(t) \in (0,+\infty)$ for any finite $t>0$, and independent of $\epsilon$.
    As a consequence, the solution $t \mapsto \mathbf{U}(t) = (\mathbf{q}(t), \mathbf{u}(t), \mathbf{d}^\star)$ to \eqref{ODE_NS} obtained via the Cauchy-Lipschitz theorem exists in fact on any finite time interval $[0,T]$, $T > 0$. 
\end{proposition}

\begin{proof}
Let $k$ and $\tau$ be given and prove the lower bound for $d_k(\tau)$. On the one hand, if $d_k(\tau)>d_k(0)$, we use the hypothesis $d_k(0)\geq a\epsilon$ \eqref{nse-dist-assumption} on the initial distances to conclude. On the other hand, if $d_k(\tau)\leq d_k(0)$, using the fact that $\sum_1^N d_i(t)$ is constant in time, one can find an index $m$ such that $d_m(\tau)\geq d_m(0)$. If $m<k$ (the converse case is similar), we sum up the balance of forces \eqref{balance-intro} from particle $i=m$ to particle $i=k-1$ and get
\begin{equation}
        2\epsilon \sum_{j=m}^{k-1}\dot{u}_{j}(t) = \frac{\dot{d}_{k}}{d_{k}}(t) - \frac{\dot{d}_{m}}{d_{m}}(t) - (G_{k} (t) - G_{m} (t)) + 2\epsilon \sum_{j=m}^{k-1} \bar f_{j}(t),
\end{equation} 
and integrate in time from time $0$ to time $\tau$, to obtain
\begin{equation} \label{bof-integrated}
     \ln \left( \frac{d_{k}(\tau)}{d_{k}(0)} \right)   =  \ln \left( \frac{d_{m}(\tau)}{d_{m}(0)} \right) + 2\epsilon \sum_{j=m}^{k-1} (u_{j} (\tau) - u_{j} (0))  + \int_{0}^{\tau} \lr{G_{k} (s) - G_{m} (s)}~ds - \sum_{j=m}^{k-1} \int_{0}^{\tau}
     \int_{q_{j-1}}^{q_{j}}f(s,x)\,dxds.
\end{equation} 
In other words,
\begin{equation} \label{bof-int}
        \frac{d_{k}(\tau)}{d_{k}(0)} = \frac{d_{m}(\tau)}{d_{m}(0)} \exp \left( 2\epsilon \sum_{j=m}^{k-1} (u_{j} (\tau) + u_{j} (0))  + \int_{0}^{\tau} \lr{G_{k} (s) - G_{m} (s)}~ds - \sum_{j=m}^{k-1} \int_{0}^{\tau} \int_{q_{j-1}}^{q_{j}}f(s,x)\,dxds \right).
\end{equation} 
To estimate the velocity term, we can use Young's inequality, exploit the energy estimate \eqref{energy-NSE} as well as the boundedness of $\mathbf{u}^\epsilon(0)$ (cf~\eqref{eq:boundu0i}):
\begin{equation}
    \begin{aligned}
        \left|2\epsilon \sum_{i=0}^{k-1} (u_{i} (\tau) - u_{i} (0))\right|
        & \le 2\epsilon \sum_{i=1}^{N-1}|u_i(t)| + 2\epsilon \sum_{i=1}^{N-1} |u_i(0)| \\[1ex] 
        &\le 2\epsilon\sum_{i=1}^{N-1}|u_i(\tau)|^2 +  2\|u_0\|_{L^\infty} \\[1ex] 
        &\le 2(\mathcal{E}_0 + \frac{1}{\mu}D_N(0)\|f\|_{L^2_t L^1_x}^2 + \|u_0\|_{L^\infty}).
    \end{aligned} \label{dk-explicit}
\end{equation}
Returning to \eqref{bof-int}, using the previous estimate, together with the bound on $G_i$ given by Lemma \ref{lemma:boundGi} and the fact that $d_{m}(\tau) \ge d_{m}(0)$ we have
\begin{equation} 
\label{di-LB-NSE}
    d_k(\tau) \geq d_k(0)\exp\left( - \left[ 2(\mathcal{E}_0 + \|u_0\|_{L^\infty}+ \frac{1}{\mu}D_N(0)\|f\|_{L^2(0, +\infty; L^1(I))}^2 )+ 2C\tau  + \|f\|_{L^1((0, +\infty) \times I)} \right] \right). 
\end{equation}
Finally, using the hypothesis $d_k(0)\geq c_0\epsilon$ \eqref{nse-dist-assumption} we obtain
\begin{equation}
    d_k(\tau) \geq c_0 \ep \exp(-C(\tau)).
\end{equation}
with $C(\tau)> 0$ independent of $\ep$, which concludes the proof of the lower bound in \eqref{di(t)-NSE}. To obtain the upper bound, one can repeat the same argument. The index $k$ and time $\tau$ being given, if $d_k(\tau) < d_k(0)$, we conclude using
the hypothesis $d_k(0) \leq C_0\epsilon$ (2.9). On the other hand, if $d_k(\tau) \geq d_k(0)$, one can find an index $m$ such that $d_m(\tau) \leq d_m(0)$ and obtain the required estimation.\\
Now, we write the ODE~\eqref{ODE_NS} as $\dfrac{d}{dt} \mathbf{U}(t) = F(t,\mathbf{U}(t))$. Thanks to the previous lower and upper bounds on the $d_i$'s, we can control the right-hand side by a linear function of $\|\mathbf{U}\|$:
\[
\|F(t,\mathbf{U}(\cdot)) \|_{\mathbb{R}^{3(N-1)}} \leq \|f \|_{L^\infty((0,t) \times I)} + \beta_\ep \|\mathbf{U}\|_{\mathbb{R}^{3(N-1)}} 
\]
with $\beta_\ep > 0$. 
So, by a Gronwall inequality on $\|\mathbf{U}\|_{\mathbb{R}^{3(N-1)}} $, we can conclude to the global-in-time existence of the solution.
\end{proof}

\section{Toward the macroscopic system} \label{sec:limit}
\subsection{Obtaining a PDE representation} \label{sec:obtainPDE}

In this subsection we introduce continuous representations for the non-linear terms in our system, and derive the PDE satisfied for all $\epsilon >0$. For convenience, we introduce the notation $P_i(t) := [q_{i}(t) - \ep, q_i(t)+\ep]$ for $i=1,...,N-1$ and $P_0(t) := [0, \epsilon]$, $P_N(t) := [1-\epsilon, 1]$. For each $t \in (0,T)$, we denote
\begin{equation}
    w_{i}(t) := \frac{u_{i} (t) - u_{i-1} (t)}{d_{i}(t)},
    \label{def:w-ss}
\end{equation}  
so that, for all $x\in (q_{i-1}(t), q_{i}(t))$,
\[
w_{i}(t) = \frac{1}{1- \rho_{i}^\ep (t)} \partial_{x}u^{\epsilon} (t,x) ,
\]
where we used the definition of $\rho^\epsilon_i$ from \eqref{rho-NSE}. Then the function $w^{\epsilon} \in C([0,T] \times I)$ is defined as 
\begin{align}
w^\epsilon(t,x)
: =& \sum_{i=1}^{N-1} \left[ \left(w_i(t)  + \dfrac{x-(q_i(t)  - \epsilon)}{2\epsilon} (w_{i+1}(t)  - w_i(t) ) \right) \mathbf{1}_{P_i(t)} (x)
+ w_{i+1}(t)  \mathbf{1}_{(q_i(t) +\epsilon, q_{i+1}(t) -\epsilon)}(x)\right] \notag \\[1ex]
 &+\left[ \frac{w_1(t) }{\epsilon}x \right] \mathbf{1}_{P_0(t)}(x) + \left[w_N(t)  - \frac{w_N(t) }{\epsilon}(x-(1 - \epsilon))\right] \mathbf{1}_{P_N(t)}(x),  \label{w-defn-nse}
\end{align}
which is depicted in Figure \ref{fig:wg-ss} below.
The interaction force $G_i^\ep(t)$, defined through \eqref{Gi-def}, can be rewritten as
\begin{equation}
    G_{i}^{\epsilon}(t)   = \left( \frac{d_{i}^{\star}+2\epsilon}{d_{i}(t)+2\epsilon} \right)^{\gamma}= \left( \frac{\rho_{i}(t)}{\rho^{\star}_{i}}\right)^{\gamma},
\end{equation}
where the definitions  \eqref{rho-NSE} and \eqref{rhost-NSE} has been used again.
The associated continuous representation $G^{\epsilon}\in C([0,T] \times I)$ is then defined as \begin{equation} \begin{aligned}  \label{G-defn-nse}
G^\epsilon(t,x)
: =& \sum_{i=1}^{N-1} \left[ \left(G_i(t)  + \dfrac{x-(q_i(t)  - \epsilon)}{2\epsilon} (G_{i+1}(t)  - G_i(t) ) \right) \mathbf{1}_{P_i(t)}(x)  + G_{i+1}(t)  \mathbf{1}_{(q_i(t) +\epsilon, q_{i+1}(t) -\epsilon)}(x) \right] \\[1ex]
 &+\left[ \frac{G_1(t) }{\epsilon}x \right] \mathbf{1}_{P_0(t)}(x) + \left[G_N(t)  - \frac{G_N(t) }{\epsilon}(x-(1 - \epsilon)\right] \mathbf{1}_{P_N(t)}(x),
 \end{aligned}
 \end{equation}
 and depicted in Figure \ref{fig:wg-ss} as well.
  \begin{figure}[ht]
    \centering
    \includegraphics[width=1\linewidth]{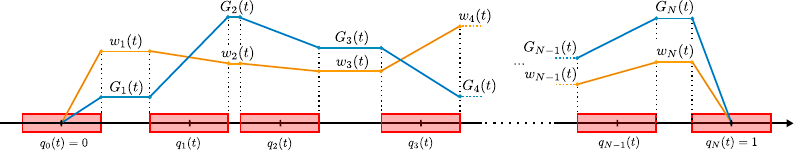}
    \caption{$w^\epsilon$ and $G^\epsilon$ at time $t$}
    \label{fig:wg-ss}
\end{figure}\\
We also introduce an alternative representation of the density, namely the volume fraction $\chi^\epsilon$:
\begin{equation}\label{def-vol}
   \chi^\epsilon(t,x) := \sum_{i=0}^{N}\mathbf{1}_{P_i(t)}(x).
\end{equation}

  \begin{figure}[ht]
    \centering
    \includegraphics[width=1\linewidth]{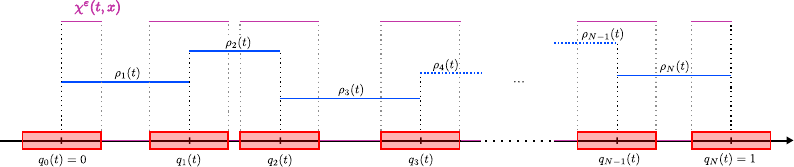}
    \caption{$\rho^\epsilon$ (blue) and $\chi^\epsilon$ (pink) at time $t$}
    \label{fig:rhochi-ss}
\end{figure}

To obtain a suitable PDE representation, we finally introduce a new velocity $v^\epsilon$ defined as (see Figure \ref{fig:veps-NSE} below)
\begin{equation} \label{def-veps}
    v^{\epsilon}(t,x) = \sum_{i=1}^{N-1} u_{i}(t)\mathbf{1}_{P_i(t)}(x) + \sum_{i=1}^{N} \left[ \frac{u_{i}(t) - u_{i-1}(t)}{d_{i}(t)}  \right](x-q_{i-1}(t)-\epsilon)\mathbf{1}_{(q_{i-1}(t)+\epsilon, q_{i}(t)-\epsilon)}(x).
\end{equation} 
 \begin{figure}[ht]
	\centering
	\includegraphics[scale= 1.2]{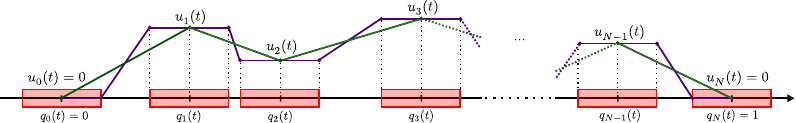}
	\caption{Depiction of $v^\epsilon$ (purple) and $u^\epsilon$ (green) at time $t$. Note that $v^\epsilon$ is constant on each particle.}
	\label{fig:veps-NSE}
\end{figure}
This representation $v^\ep$ of the velocity allows us to formulate an approximate momentum equation for any $\epsilon > 0$ on a common ``grid'', since we define $\partial_x w^\epsilon$ and $\partial_x G^\epsilon$ to be constant on each particle. On the other hand, $u^\epsilon$ is the natural approximate velocity to obtain the transport and conservation equations. In Section \ref{sec:uniform-NSE}, we will show that $v^\epsilon$ and $u^\epsilon$ converge to the same limit, meaning that both are valid macroscopic representations of the limit velocity field.

We can now derive the approximate PDE system that we will eventually use to pass to the limit $\epsilon \to 0$ and obtain the macroscopic model. 
\begin{proposition}[Approximate PDE system] \label{prop:NSE-eqn-eps-copy}
The following system is satisfied in $\mathcal{D}'([0,T) \times I)$ for any $\epsilon > 0$:
\begin{subnumcases} {\label{NSE-eps}}
     \partial_{t} \rho^\epsilon + \partial_{x}(\rho^\epsilon u^\epsilon) = 0,  \label{NSE-eps-cons} \\[1ex]
        \partial_{t} (\chi^\epsilon v^\epsilon) + \partial_{x}\left( \chi^\epsilon (v^\epsilon)^{2} \right) -  \partial_{x} w^\epsilon + \partial_{x} G^\epsilon = f^\epsilon, \label{NSE-eps-mom} \\[1ex]  \partial_{t} \rho^{\star,\epsilon} + u^\epsilon \partial_{x} \rho^{\star,\epsilon} = 0 \label{NSE-eps-transport},
\end{subnumcases}
where $f^\epsilon(t,x) = \sum_{i=1}^{N-1} \bar{f}_i(t) \mathbf{1}_{P_i(t)}(x)$.
\end{proposition}
\begin{proof}
  For the continuity equation \eqref{NSE-eps-cons}, recall that $\rho^\epsilon(t,x) = \sum_{i=1}^{N} \rho_i(t) \mathbf{1}_{[q_{i-1}(t), q_i(t))}(x)$. Then using $\rho_i(t) = 2\epsilon/(d_i(t)+2\epsilon)$, a straightforward computation shows that, for any $t>0$ and any $x\in [q_{i-1}(t), q_i(t))$,
    \begin{equation*}
        \rho_i'(t) = - \frac{2\epsilon}{(d_i(t)+2\epsilon)^2} (u_{i}(t) - u_{i-1}(t)) =- \rho_i(t) \partial_x u^\ep(t,x) \mathbf{1}_{[q_{i-1}(t), q_i(t))}(x),
    \end{equation*} where we have used $\dot{d}_i(t) = u_i(t)-u_{i-1}(t).$ Therefore in the distributional sense,
    \begin{align*}
        \partial_t \rho^\epsilon(t,\cdot) = -\rho^\epsilon(t,\cdot) \partial_x  u^\epsilon(t,\cdot) + \sum_{i=1}^N \rho_i(t) \left( u_i(t) \delta_{\{x=q_{i}(t)\}} - u_{i-1}(t) \delta_{\{x=q_{i-1}(t)\}} \right).
    \end{align*}
    Finally, notice that 
    \begin{align*}
        \partial_x(\rho^\epsilon u^\epsilon)(t,\cdot) &= \rho^\epsilon (t,\cdot)\partial_x u^\epsilon(t,\cdot) + u^\epsilon (t,\cdot)\left( \sum_{i=1}^N \rho_i(t) \left[ \delta_{\{x=q_{i-1}(t)\}} - \delta_{\{x=q_{i}(t)\}} \right] \right) \\[1ex] 
        &= \rho^\epsilon(t,\cdot) \partial_x u^\epsilon(t,\cdot) + \sum_{i=1}^N \rho_i(t) \left[ u_{i-1}(t)\delta_{\{x=q_{i-1}(t)\}} - u_i(t)\delta_{\{x=q_{i}(t)\}} \right],
    \end{align*}
    and so \eqref{NSE-eps-cons} holds. 
    
    We now move on to the transport equation \eqref{NSE-eps-transport}. For $\rho^{\star, \epsilon}$, we have that for any $t > 0$, \begin{equation} \label{dt-rhostar}
        \partial_t \rho^{\star, \epsilon}(t,\cdot) 
        = \sum_{i=1}^{N} \rho_i^\star[  u_i(t)  \delta_{\{x=q_{i}(t)\}} - u_{i-1}(t) \delta_{\{x=q_{i-1}(t)\}}],
    \end{equation} 
    while
    \begin{equation} \label{dx-rhostar}
        (\partial_x \rho^{\star, \epsilon})(t,\cdot) 
        = \sum_{i=1}^{N} \rho_i^\star [ \delta_{\{x=q_{i-1}(t)\}} 
        {\color{blue}-}\delta_{\{x=q_{i}(t)\}} ].
    \end{equation} Multiplying by $u^\epsilon$, we get
    \begin{equation}\label{udx-rhostar}
      ( u^\epsilon \partial_x \rho^{\star, \epsilon})(t,\cdot) = \sum_{i=1}^{N} \rho_i^\star [ u_{i-1}(t)\delta_{\{x=q_{i-1}(t)\}} -   u_i(t)\delta_{\{x=q_{i}(t)\}} ],
    \end{equation}
    and therefore, adding \eqref{dt-rhostar} and \eqref{dx-rhostar} leads to the transport equation~\eqref{NSE-eps-transport}. 
    
    At last, we drive  the momentum equation \eqref{NSE-eps-mom}. Summing up the balance of forces \eqref{balance-intro} for each particle and dividing by $2\epsilon$ gives us, for all $t>0$ and $x\in[0,1]$,
\begin{equation}
    \sum_{i=1}^{N-1} \dot{u}_i (t) \mathbf{1}_{ P_i(t)}(x) =  \mu\sum_{i=1}^{N-1}\partial_{x}w^{\epsilon} (t,x)\mathbf{1}_{ P_i(t)}(x) - \sum_{i=1}^{N-1}\partial_{x}G^{\epsilon}(t,x)\mathbf{1}_{ P_i(t)}(x) +  \sum_{i=1}^{N-1}\bar{f}_i(t)\mathbf{1}_{P_i(t)}(x).
    \label{eq:momentum_disc}
\end{equation}
    Now using the definition of $v^\epsilon$ from \eqref{def-veps}, we have in the distributional sense, for all $t>0$,
    \begin{equation}
        \partial_t (\chi^\epsilon v^\epsilon)(t,\cdot) = 
 \sum_{i=1}^{N-1} \left[ \dot{u}_i  (t)\mathbf{1}_{P_i(t)} - (u_i(t) )^2 \delta_{\{x=q_{i}(t)-\epsilon\}} + (u_i (t))^2 \delta_{\{x=q_{i}(t)+\epsilon\}} \right],
    \end{equation} 
while 
    \begin{equation}
        \partial_x (\chi^\epsilon (v^\epsilon)^2)(t,\cdot)=   \sum_{i=1}^{N-1} \left[ (u_i(t) )^2 \delta_{\{x=q_{i}(t)-\epsilon\}} - (u_i(t))^2 \delta_{\{x=q_{i}(t)+\epsilon\}} \right].
    \end{equation} 
Therefore, we have, for all $t>0$ and $x\in[0,1]$,
    \begin{equation}
         \partial_t (\chi^\epsilon v^\epsilon)(t,x) + \partial_x (\chi^\epsilon (v^\epsilon)^2)(t,x) =  \sum_{i=1}^{N-1} \dot{u}_i (t) \mathbf{1}_{P_i(t)}(x).
    \end{equation}
Since $w^\epsilon$ and $G^\epsilon$ are piecewise linear, continuous and constant in between particles, we have 
\begin{equation*}
        \partial_x w^\epsilon (t,\cdot)= \sum_{i=1}^{N-1}\partial_{x}w^{\epsilon}(t,\cdot) \mathbf{1}_{ P_i(t)} \text{ and } \partial_x G^\epsilon (t,\cdot)= \sum_{i=1}^{N-1}\partial_{x}G^{\epsilon} (t,\cdot)\mathbf{1}_{ P_i(t)}
    \end{equation*} almost everywhere in $(0,T) \times I$, from which we conclude using \eqref{eq:momentum_disc}.
\end{proof}

\subsection{Uniform bounds on the macroscopic variables}
\label{sec:uniform-NSE}
In this subsection we derive uniform in $\ep$ estimates for the functions $\rho^\epsilon, u^\epsilon,\rho^{\star,\epsilon}$ as well as the non-linear functions $w^\epsilon$ and $G^\epsilon$ introduced above. 

Let us first summarize what bounds we currently have on the macroscopic functions as a consequence of our estimates so far. 
\begin{corollary}
Let us assume the hypotheses of Theorem~\ref{thm:nse}. With the notations of the previous subsection, we have the following bounds:
\begin{align}
    0 < \dfrac{2}{2+ C_2(t)} \leq \rho^\ep(t,\cdot) \leq \dfrac{2}{c_1(t) + 2} < 1,   \label{rho-bound-nse}  \\[1ex]   
      0<\frac{2}{C_0+2}  \le ~\rho^{\star, \epsilon}(t,\cdot) \le 1, \label{rhost-bound-nse} \\
      \|G^\epsilon\|_{L^\infty_{t,x}} \le \left( \frac{C_0+2}{2} \right)^\gamma, \label{Geps-Linfty-NSE}\\[1ex]
      \|u^\epsilon\|_{L^2_{t}H^1_x} \le  \frac{4}{\mu} \left(\mathcal{E}_0 + \dfrac{1}{\mu}\|f\|_{L^2_tL^1_x}^2 D_N(0)\right) , \label{ueps-H1-NSE}
\end{align}
where $C_0$ and $c_1(t), C_2(t) >0$ have been previously defined in Propositions~\ref{prop:ID_micro} and~\ref{prop:global-nse}.
\end{corollary}
In order to pass to the limit in the non-linear terms of \eqref{NSE-eps-cons}-\eqref{NSE-eps-transport}, we will need further bounds on $\partial_x G^\epsilon$ and the densities $\rho^\epsilon, \rho^{\star,\epsilon}$.
We now derive a bound for $\partial_x G^\epsilon$ which is correlated to a control of the increments $d_{i+1} - d_i$. 
\begin{lemma}  \label{lem:NSE-dxG} We have
 \begin{equation} \label{di-diff-t}
        \int_0^t |d_{i+1}(\tau) - d_i(\tau)|~d\tau \le C(t)\epsilon^2, 
    \end{equation} 
    where $C(t) \in (0,+\infty)$ for all $t> 0$, and therefore 
    \begin{equation}
    \|\partial_{x}G^{\epsilon}\|_{L^{1}(0,T;L^\infty(I))} \le C(T),
    \end{equation}
    for some positive constant $C(T)$ independent of $\ep$.
\end{lemma}
\begin{proof}
 We first estimate $\partial_x G^\epsilon$ on each interval $(q_i(t) - \epsilon, q_i(t)+\epsilon)$ using the definition $G_i(t) = \left( \frac{d_i^\star + 2\epsilon}{d_i(t)+2\epsilon} \right)^\gamma$.

        \begin{align}  \notag
|\partial_{x}G^{\epsilon}\mathbf{1}_{(q_{i}(t)  - \epsilon, q_{i}(t) + \epsilon)}| 
            &= \left| \frac{G_{i+1}  - G_{i} }{2\epsilon} \right| 
            = \left| \frac{\left( \frac{d_{i+1}^{\star}+2\epsilon}{d_{i+1}+2\epsilon} \right)^{\gamma} - \left( \frac{d_{i}^{\star}+2\epsilon}{d_{i}+2\epsilon} \right)^{\gamma}}{2\epsilon} \right| \\[1ex] 
            &= \left| \frac{  ( d_{i+1}^{\star}+2\epsilon)^{\gamma} (d_{i}+2\epsilon)^{\gamma} - (d_{i}^{\star} + 2\epsilon)^{\gamma}(d_{i+1}+2\epsilon)^{\gamma} }  {2\epsilon (d_{i+1}+2\epsilon)^{\gamma} (d_{i}+2\epsilon)^{\gamma} } \right|.  \label{dxG-1}
        \end{align}
    Suppose firstly that $( d_{i+1}^{\star}+2\epsilon) (d_{i}+2\epsilon) > (d_{i}^{\star} + 2\epsilon)(d_{i+1}+2\epsilon)$. Using the inequality $|a^\gamma - b^\gamma| \le \gamma |a-b|a^{\gamma-1}$ for $a \ge b >0$ and $\gamma \ge 1$, we get
    \begin{equation*}
        \begin{aligned}
           |\partial_{x}G^{\epsilon}\mathbf{1}_{(q_{i}(t) - \epsilon, q_{i}(t) + \epsilon)}| &\le \frac{| \gamma ((d_{i} + 2\epsilon)(d_{i+1}^{\star} - d_{i} ^{\star}) - (d_{i}^{\star} + 2\epsilon)(d_{i+1}-d_{i})   ) (d_{i+1}^{\star} + 2\epsilon)^{\gamma -1}(d_{i}+2\epsilon)^{\gamma -1}|}{2\epsilon (d_{i+1}+2\epsilon)^{\gamma} (d_{i}+2\epsilon)^{\gamma} } \\[1ex] 
           & \\
           &= \frac{|\gamma ((d_{i} + 2\epsilon)(d_{i+1}^{\star} - d_{i} ^{\star}) - (d_{i}^{\star} + 2\epsilon)(d_{i+1}-d_{i}))|}{2\epsilon (d_{i+1}+2\epsilon) (d_{i}+2\epsilon) } \left( \frac{d_{i+1}^{\star}+2\epsilon}{d_{i+1}+2\epsilon} \right)^{\gamma -1} \\
           &= |H_{i}| \left( \frac{d_{i+1}^{\star}+2\epsilon}{d_{i+1}+2\epsilon} \right)^{\gamma -1}\leq |H_{i}| \left( \frac{C_0+2}{2} \right)^{\gamma-1},
        \end{aligned}            
        \end{equation*} 
       where we used  the upper bound on $G_i$ from \eqref{Geps-bound-NSE}.
       We may express $H_{i}$ as 
       \begin{equation} \begin{aligned}
        H_{i} &= \frac{\gamma}{(d_{i} + 2\epsilon)(d_{i+1} + 2\epsilon)} \left[ \left( \frac{d_{i}+2\epsilon}{2\epsilon}  \right)(d_{i+1}^{\star} - d_{i}^{\star}) -  \left(\frac{d_{i}^{\star}+2\epsilon}{2\epsilon}\right) (d_{i+1} - d_{i}) \right] \\[1ex] &= \frac{\gamma}{(d_{i} + 2\epsilon)(d_{i+1} + 2\epsilon)} \left[ \frac{1}{\rho_{i}}(d_{i+1}^{\star} - d_{i}^{\star}) -  \frac{1}{\rho_{i}^{\star}} (d_{i+1} - d_{i}) \right].
        \end{aligned}
    \end{equation}  
    Thus, using the bounds on $d_i$ \eqref{di(t)-NSE} and $d_i^\star$ \eqref{di(t)-NSE},
    we have 
 \begin{equation} \label{dxG-nse-1}
|\partial_{x}G^{\epsilon}\mathbf{1}_{(q_{i}(t)-\epsilon, q_{i}(t)+\epsilon)}| \le \frac{\gamma (C_0+2)^\gamma}{2^{\gamma}(d_{i}(t) + 2\epsilon)(d_{i+1}(t) + 2\epsilon)} \left[ |d_{i+1}^{\star} - d_{i}^{\star}| +   |d_{i+1}(t) - d_{i}(t)| \right].
    \end{equation} 
 The same bound can be obtained in the case where $(d_{i+1}^{\star}+2\epsilon) (d_{i}+2\epsilon) \le (d_{i}^{\star} + 2\epsilon)(d_{i+1}+2\epsilon)$. In that case, we can apply the inequality $|a^\gamma - b^\gamma| \le \gamma |a-b|b^{\gamma-1}$ to \eqref{dxG-1} and proceed as above.
Using the positivity of the $d_i$'s and estimate \eqref{dist-diff-t=0} on the increments $d_{i+1}^{\star} - d_{i}^{\star}$, we get
 \begin{equation} \label{dxG-nse-2}
|\partial_{x}G^{\epsilon}\mathbf{1}_{(q_{i}(t)-\epsilon, q_{i}(t)+\epsilon)}|
\le \frac{\gamma (C_0+2)^\gamma}{2^{\gamma+2}\epsilon^2} \left[  C_0 \epsilon^2 +   |d_{i+1}(t) - d_{i}(t)| \right].
    \end{equation}
Thus, it remains to obtain a suitable bound on $|d_{i+1}(t) - d_i(t)|$. 
We now aim to prove that the bound \eqref{di-diff-t=0}  propagates as \eqref{di-diff-t} for positive times.
To do so, we return to the balance of forces \eqref{balance-intro} for particle $i$. Integrating in time and taking the modulus gives us 
\begin{equation} \label{ln-eq1-ss}
  \left|  \ln \left( \frac{d_{i+1}(t)}{d_{i}(t)} \right) \right| = \left| \ln \left( \frac{d_{i+1}(0)}{d_{i}(0)} \right) - \int_{0}^{t} (G_{i+1} - G_{i})~d\tau + \int_{0}^{t}\int_{(q_{i}-\epsilon, q_{i}+\epsilon)} f\,dx d\tau - 2\epsilon(u_i(t) - u_i(0)) \right|
\end{equation}
Fix $t \in (0,T)$. Using $1 - 1/x \le \ln(x) \le x - 1$ for $x > 0$, we have
\begin{equation} \label{ln-bounds-t=0}
    \left| \ln \left( \frac{d_{i+1}(0)}{d_{i}(0)} \right) \right| \le
    \left\{
    \begin{aligned}
        &\frac{|d_{i+1}(0) - d_{i}(0)|}{d_{i}(0)}, \text{ if } d_{i+1}(0) \ge d_{i}(0), \\[1ex] 
        &\frac{|d_{i+1}(0) - d_{i}(0)|}{d_{i+1}(0)}, \text{ if } d_{i+1}(0) < d_{i}(0).
    \end{aligned}
    \right.
\end{equation} 
Now using the lower bound on the distances at initial time \eqref{nse-dist-assumption}, we obtain  
\begin{equation} \label{ln-eq2-ss}
     \left| \ln \left( \frac{d_{i+1}(0)}{d_{i}(0)} \right) \right| \le \frac{1}{c_0\epsilon} |d_{i+1}(0) - d_{i}(0)|,
\end{equation}
Similarly,
\begin{equation} \label{ln-bounds-t>0}
    \left| \ln \left( \frac{d_{i+1}(t)}{d_{i}(t)} \right) \right| \ge
    \left\{
    \begin{aligned}
        &\frac{|d_{i+1}(t) - d_{i}(t)|}{d_{i+1}(t)}, \text{ if } d_{i+1}(t) \ge d_{i}(t), \\[1ex] 
        &\frac{|d_{i+1}(t) - d_{i}(t)|}{d_{i}(t)}, \text{ if } d_{i+1}(t) < d_{i}(t),
    \end{aligned}
    \right.
\end{equation}
and using the upper bound \eqref{di(t)-NSE} we can also show that
\begin{equation} \label{ln-eq3-ss}
    \frac{1}{C_2 \epsilon} |d_{i+1}(t) - d_{i}(t)| \le \left| \ln \left( \frac{d_{i+1}(t)}{d_{i}(t)} \right) \right|.
\end{equation}
Note that we can write $G_{i+1}(t) - G_i(t) = 2\epsilon \partial_x G^\epsilon \mathbf{1}_{P_i(t)}$, which follows from the definition of $G^\epsilon$ \eqref{G-defn-nse}.
Substituting this along with \eqref{ln-eq2-ss} and \eqref{ln-eq3-ss} into \eqref{ln-eq1-ss}, results in
\begin{equation} \label{di-diff-nse-1}
     \frac{1}{C_2 \epsilon}|d_{i+1}(t) - d_{i}(t)| \le   \frac{1}{c_0\epsilon} |d_{i+1}(0) - d_{i}(0)| +  2 \epsilon \left(  \int_{0}^{t} |\partial_{x}G^{\epsilon}| \mathbf{1}_{P_i(\tau)} d\tau + \|f\|_{L^1_tL^{\infty}_{x}} + |u_i(t) - u_i(0)| \right).
\end{equation}  Combining \eqref{dxG-nse-2} and \eqref{di-diff-nse-1} gives us
\begin{equation} \label{di-diff-2}
\begin{aligned}
      |d_{i+1}(t) - d_{i}(t)| 
      & \le \frac{C_2}{c_0}|d_{i+1}(0) - d_{i}(0)| + \frac{C_2\gamma (C_0+2)^\gamma}{2^{\gamma+1}} \left[  C_0 \epsilon^2t +   \int_0^t |d_{i+1}(s) - d_{i}(s)| \ ds \right] \\
      & \quad + 2 C_2\epsilon^{2}   \left(  \|f\|_{L^1_tL^{\infty}_{x}} + |u_{i} (t)-u_{i} (0)|\right).
\end{aligned}
\end{equation} 
Integrating in time again from $\tau = 0$ to $\tau = t$, recalling \eqref{di-diff-t=0}, we get 

\begin{equation} \label{geps-1}
\begin{aligned}
    \int_0^t |d_{i+1}(\tau) - d_{i}(\tau)|~d\tau 
    & \le C \big( 1 +  \|f\|_{L^1_tL^{\infty}_{x}} \big) \ep^2 \, t + C \ep^2 t^2 + C\epsilon^2 \int_0^t|u_i(\tau) -u_i(0)|~d\tau \\
    & \qquad + C  \int_0^t  \left(\int_0^\tau |d_{i+1}(s) - d_{i}(s)|~ds \right) ~d\tau
\end{aligned}
\end{equation}

 We can use the bound on the velocities $u_i(t)$ from \eqref{ui-L1-NSE} with the boundedness of $u_i(0)$ from \eqref{nse-u0-condition} to estimate \begin{equation*}
     \begin{aligned}
         \int_0^t|u_i(\tau) -u_i(0)|~d\tau& \le  \|u_i\|_{L^1(0,t)}+ t |u_i(0)|  \le C(t).
     \end{aligned}
 \end{equation*}
 Inserting this bound into \eqref{geps-1} and applying Gronwall's inequality gives us the desired bound~\eqref{di-diff-t}
\begin{equation} \label{geps-bound-NSE}
    \begin{aligned}
          \int_0^t |d_{i+1}(\tau) - d_{i}(\tau)|~d\tau \leq C(t) \ep^2.
    \end{aligned}
\end{equation}
Returning to \eqref{dxG-nse-2} and integrating in time, we get for any $t \in [0,T]$,
\begin{equation}
\int_0^t |\partial_{x}G^{\epsilon}\mathbf{1}_{(q_{i}(\tau)-\epsilon, q_{i}(\tau)+\epsilon)}| ~d\tau 
\leq \frac{\gamma (C_0+2)^\gamma}{2^{\gamma+2}\epsilon^2} \left[  C \epsilon^2 t +   \int_0^t|d_{i+1}(\tau) - d_{i}(\tau)| \ d\tau \right]
\leq C(t),
\end{equation} 
where $C(t) > 0$ is independent of $i$ and $\epsilon$ and remains finite on $[0,T]$. This completes the proof.
\end{proof}

\begin{lemma}[Uniform boundedness of $u^\epsilon$]
    We have \begin{equation}
   \min_i u_i(0) -  \|\partial_x G^\epsilon\|_{L^1_t L^\infty_x} - \|f\|_{L^1_t L^\infty_x} \le u_i(t) \le \max_i u_i(0) + \|\partial_x G^\epsilon\|_{L^1_t L^\infty_x} + \|f\|_{L^1_t L^\infty_x},
    \end{equation} and therefore \begin{equation} \label{u-Linf-NSE}
        \|u^\epsilon\|_{L^\infty_{t,x}} \le C(T).
    \end{equation}
\end{lemma}
\begin{proof}
    Our proof follows a maximum principle strategy. Firstly, the Cauchy-Lipschitz theorem ensures that the solution to the microscopic problem \eqref{ODE_NS} is regular for $\epsilon$ fixed, i.e. that $u_i \in C^1([0, T])$ for any $T>0$. Therefore, the function \begin{equation*}
        y^\epsilon(t) := \max_{i\in\{1,\ldots,N-1\}} u_i(t)
    \end{equation*} is continuous. As a result, we can decompose the full time interval $[0,T]$ into $K$ subintervals, where on each subinterval, there exists a particle that possesses the maximum velocity for any time in that subinterval. More precisely, we have the decomposition
    \begin{equation*}
        [0,T] = [0, t_{{1}}] \cup[t_{1}, t_{2}] \cup \dots \cup [t_{{k-1}}, t_{{k}}] \cup \dots \cup [t_{{K-1}}, T],
    \end{equation*} 
    such that for any $t$ in the subinterval $[t_{{k-1}}, t_{k}]$, we have $y^\epsilon(t) = u_{i_k}(t)$ for some $i_k \in \llbracket 1, N \rrbracket $. Note that the continuity of $y^\epsilon$ means that we can assume without loss of generality that each interval has non-zero measure. Indeed, every interval of zero measure can  be absorbed into a neighbouring interval.
    Now, fix a time $t \in [0,T]$. Assuming $t \in [t_{{k-1}}, t_{{k}}]$ for some $k\in \llbracket 1,K \rrbracket$, the balance of forces for particle $i_{k}$ reads as 
\begin{align*}
2\ep \dot{u}_{i_k}(t) 
= \mu \left(\dfrac{u_{i_{k}+1}(t) - u_{i_k}(t)}{d_{i_{k}+1}(t)} 
- \dfrac{u_{i_k}(t) - u_{i_{k}-1}(t)}{d_{i_{k}}(t)}\right) 
- \big(G_{i_{k}+1}(t) - G_{i_k}(t)\big) + 2\epsilon \langle f(t) \rangle_i.
\end{align*}
Since $y^\ep(t)=u_{i_{k}}(t)$, we have that both $u_{i_{k}}(t) - u_{i_{k}-1}(t)$ and $u_{i_{k}}(t)-u_{i_{k}+1}(t)$ are negative. Therefore, upon integrating in time between $t_{k-1}$ and $t$,
\begin{align*} 
    2\epsilon( u_{i_k}(t) - u_{i_k}(t_{k-1})) \le 2\epsilon \int_{t_{k-1}}^{t} \|\partial_x G^\epsilon(s,\cdot)\|_{L^\infty_{x}}~ds + 2\epsilon \int_{t_{k-1}}^{t} \|f(s,\cdot)\|_{L^\infty_x}~ds,
\end{align*} i.e. that
\begin{equation}\label{maxprinciple-1}
     u_{i_k}(t) \le  u_{i_k}(t_{k-1}) + \int_{t_{k-1}}^{t} \|\partial_x G^\epsilon(s,\cdot)\|_{L^\infty_{x}}~ds +  \int_{t_{k-1}}^{t} \|f(s,\cdot)\|_{L^\infty_x}~ds.
\end{equation}
Using the fact that $u_{i_{k}}(t_{k-1}) = u_{i_{k-1}}(t_{k-1})$ (due to the continuity of $y$), we can perform a similar estimate for particle $i_{k-1}$ and obtain:
\begin{align*}
    u_{i_k}(t_{k-1}) ={u_{i_{k-1}}(t_{k-1})} \leq u_{i_{k-1}}(t_{k-2}) +  \int_{t_{k-2}}^{t_{k-1}} \|\partial_x G^\epsilon(s,\cdot)\|_{L^\infty_{x}}~ds +  \int_{t_{k-2}}^{t_{k-1}} \|f(s,\cdot)\|_{L^\infty_x}~ds.
\end{align*}
Substituting this expression into \eqref{maxprinciple-1} gives us 
\begin{align*}
      u_{i_k}(t) \le  u_{i_{k-1}}(t_{k-2}) + \int_{t_{k-2}}^{t} \|\partial_x G^\epsilon(s,\cdot)\|_{L^\infty_{x}}~ds +  \int_{t_{k-2}}^{t} \|f(s,\cdot)\|_{L^\infty_x}~ds.
\end{align*}
Iterating this argument over all subintervals, we arrive at the estimate
\begin{equation*}
     u_{i_k}(t) \le  \max_i u_i(0) + \int_{0}^{t} \|\partial_x G^\epsilon(s,\cdot)\|_{L^\infty_{x}}~ds +  \int_{0}^{t} \|f(s,\cdot)\|_{L^\infty_x}~ds.
\end{equation*} Using the $L^1_t L^\infty_x$ bound on $\partial_xG^\ep$ from Lemma \ref{lem:NSE-dxG} and the regularity of $f$ allows us to conclude the upper bound. Repeating this argument for minimum instead of maximum  concludes the proof.
\end{proof}
We are now in a position to improve  the estimates from Lemma \ref{lem:NSE-dxG}.
\begin{corollary}
    We have the  uniform bound \begin{equation} \label{di-diff-uniform}
        |d_{i+1}(t) - d_i(t)| \le C(T)\epsilon^2,
    \end{equation} for any $i = 1,...,N-1$, and therefore 
    \begin{equation}\label{improved_Gx}
        \|\partial_xG^\epsilon\|_{L^\infty_{t,x}} \le C(T).
    \end{equation}
\end{corollary}
\begin{proof} 
    Simply return to \eqref{di-diff-2} and use the boundedness of the velocity $|u_i(t)|$ to obtain \eqref{di-diff-uniform}. Using \eqref{di-diff-uniform} in \eqref{dxG-nse-2} leads to the bound on $\partial_xG^\ep$.
\end{proof}
\bigskip

Our next goal is to obtain estimates for $\rho^\epsilon$ and $\rho^{\star,\epsilon}$. 

\begin{proposition} \label{prop:rho-BV-NSE} We have 
\begin{equation} \label{rho-LinfBV}
    \|\rho^{\epsilon}\|_{L^\infty(0,T; BV(I))}  + \|\rho^{\star, \epsilon}\|_{L^\infty(0,T; BV(I))} \le C(T),
\end{equation}
and
\begin{equation} \label{rho-BV}
    \|\rho^{\epsilon}\|_{BV((0,T) \times I)} + \|\rho^{\star, \epsilon}\|_{BV((0,T) \times I)} \le C(T).
\end{equation}
For the time derivatives, we have 
\begin{equation}
    \|\partial_t \rho^\ep \|_{L^2(0,T; H^{-1}(I))} + \|\partial_t \rho^{\star, \ep} \|_{L^2(0,T; H^{-1}(I))} \le C(T). \label{rho-time-derivs}
\end{equation}
\end{proposition}
\begin{proof} 
{We first recall that, for $\Omega$ open subset of $\mathbb{R}^n$ and $g\in L^1(\Omega)$, $\|g\|_{BV(\Omega)} := \|g\|_{L^1(\Omega)} + TV_{\Omega} ~g$, where 
\begin{equation*}
    TV_{\Omega} ~g = \sup_{\varphi \in \mathcal{F}_\Omega}  \int_\Omega g({\mathbf x})\, \mbox{div} \varphi({\mathbf x})~d{\mathbf x} \quad \mbox{with}\quad \mathcal{F}_\Omega = \{ \varphi \in \mathcal{D}(\Omega) : \|\varphi\|_{L^\infty(\Omega)} \le 1\}.
\end{equation*}
\\
Let us now prove \eqref{rho-LinfBV}. 
Since $\|\rho^\epsilon\|_{L^{\infty}_{t,x}} \leq 1$ and $\|\rho^{\star,\epsilon}\|_{L^{\infty}_{t,x}} \leq 1$ from \eqref{rho-bound-nse} and \eqref{rhost-bound-nse}, it is enough to show \begin{equation*}
    \TV_I~ \rho^\epsilon(t,\cdot) \le C(T) \text{ and }    \TV_I~ \rho^{\star,\epsilon}(t,\cdot) \le C(T) \quad \forall ~ t \in (0,T). 
\end{equation*}
To this end, using the fact that $\rho^\ep(t,\cdot)$ is piecewise constant, together with $ \rho_i(t) =  2\epsilon / (d_i(t) +2\epsilon)$ from \eqref{rho-NSE} and the bound $|d_{i+1}(t)-d_i(t)|\le C(T)\epsilon^2$ from \eqref{di-diff-uniform}, we have
\begin{equation}\label{BVthis}
    \begin{aligned}
        \TV_I~ \rho^\epsilon(t,\cdot) &= \sum_{i=1}^{N-1}|\rho_{i+1}(t) - \rho_i(t)| = \frac{\rho_i \rho_{i+1}}{2\epsilon} \sum_{i=1}^{N-1} |d_{i+1}(t)-d_i(t)| \\[1ex] &\le \frac{1}{2\epsilon}N\epsilon^2C(T) \le C(T).
    \end{aligned}
\end{equation}  
Similarly,
\begin{equation}
    \begin{aligned}
        \TV_I~ \rho^{\star,\epsilon}(t,\cdot) & = \frac{\rho_i^\star \rho_{i+1}^\star}{2\epsilon} \sum_{i=1}^{N-1} |d_{i+1}^\star-d_i^\star|\le \frac{1}{2\epsilon}N\epsilon^2C \le C,
    \end{aligned}
\end{equation} using the definition of $\rho_i^\star$ from \eqref{rhost-NSE} and $|d_{i+1}^\star-d_i^\star|\le C\epsilon^2$ from \eqref{dist-diff-t=0}. This proves \eqref{rho-LinfBV}. \\
\\
We now prove \eqref{rho-BV} for $\rho^\epsilon$. Again, since $\|\rho^\epsilon\|_{L^{\infty}_{t,x}} < 1$ from \eqref{rho-bound-nse}, it is enough to obtain
\begin{equation*}
    TV_{[0,T]\times I} ~\rho^\ep \le C(T),
\end{equation*} 
where
\begin{equation*}
    TV_{[0,T]\times I} ~\rho^\ep = \sup_{\varphi \in \mathcal{F}_{_{[0,T]\times I}}} \left(\int_0^T \!\!\!\int_I \rho^\epsilon(t,x) \partial_t \varphi(t,x)\,dx\,dt + \int_0^T \!\!\!\int_I \rho^\epsilon(t,x) \partial_x \varphi(t,x)\,dx\,dt\right).
\end{equation*}
\\
The bound for the spatial derivative comes from the integration in time of \eqref{BVthis}.
For the temporal derivative, let us first consider $\varphi\in\mathcal{F}_{[0,T]\times I}$ given. We can use the continuity equation $\partial_t \rho^\epsilon = - \partial_x (\rho^\epsilon u^\epsilon)$ to get
\begin{equation}\label{eq:dx-rho-u}
\begin{aligned}
    \int_0^T \!\!\!\int_I \rho^\epsilon(t,x) \partial_t \varphi(t,x)\,dx\,dt
    =&\int_0^T \langle \partial_x(\rho^\epsilon u^\epsilon)(t,\cdot), \varphi (t,\cdot)\rangle_{\mathcal{D}'(I) \times \mathcal{D}(I)} ~dt\\
     =&\int_0^T\langle u^\epsilon(t,\cdot)\partial_x\rho^\epsilon(t,\cdot), \varphi (t,\cdot)\rangle_{\mathcal{D}'(I) \times \mathcal{D}(I)}  ~dt\\
     &+\int_0^T \langle \rho^\epsilon(t,\cdot)\partial_x u^\epsilon(t,\cdot), \varphi (t,\cdot)\rangle_{\mathcal{D}'(I) \times \mathcal{D}(I)}~dt.
\end{aligned}
\end{equation}
We begin by estimating the first term in the right-hand side. From the definitions of $\rho^\epsilon$ \eqref{rho-NSE} and $u^\ep$ \eqref{u-NSE}, we have 
\begin{equation*} 
\begin{aligned}
    \partial_x \rho^\epsilon(t,\cdot) &= \sum_{i=1}^{N-1} (\rho_{i+1}(t) - \rho_{i}(t)) \delta_{\{x=q_{i}(t)\}}, \\[1ex] 
    u^\epsilon (t,\cdot) \partial_x\rho^\epsilon(t,\cdot) &= \sum_{i=1}^{N-1} (\rho_{i+1}(t) - \rho_{i}(t))  u_i(t) \delta_{\{x=q_{i}(t)\}},
\end{aligned}
\end{equation*}
so that
\begin{equation*}
\langle u^\epsilon(t,\cdot)\partial_x\rho^\epsilon(t,\cdot), \varphi (t,\cdot)\rangle_{\mathcal{D}'(I) \times \mathcal{D}(I)}  = \sum_{i=1}^{N-1} (\rho_{i+1}(t) - \rho_{i}(t))  u_i(t) \varphi(t,q_{i}(t)).
\end{equation*}
Integrating in time and  using again the definition of $\rho_i$ from \eqref{rho-NSE} and the bound on $d_{i+1} - d_i$ from \eqref{di-diff-uniform} we get
\begin{equation*}
\begin{aligned}
    \left|\int_0^T \langle u^\epsilon(t,\cdot)\partial_x\rho^\epsilon(t,\cdot), \varphi (t,\cdot)\rangle_{\mathcal{D}'(I) \times \mathcal{D}(I)} ~dt \right|
    &\leq \int_0^T\sum_{i=1}^{N-1} | (\rho_{i+1}(t) - \rho_{i}(t))  u_i(t) \varphi(t,q_{i}(t))| ~dt \\
    &\leq \int_0^T\frac{\rho_i(t) \rho_{i+1}(t)}{2\epsilon} \sum_{i=1}^{N-1} |d_{i+1}(t)-d_i(t)| |u_i(t) \varphi(t,q_i(t))|  ~dt\\
    &\leq \frac{1}{2\epsilon} \|u^\epsilon\|_{L^\infty_{t,x}} \|\varphi\|_{L^\infty_{t,x}} \sum_{i=1}^{N-1} \int_0^T|d_{i+1}(t)-d_i(t)|  ~dt\\
    &\leq C(T) \|u^\epsilon\|_{L^\infty_{t,x}} \|\varphi\|_{L^\infty_{t,x}}.
    \end{aligned}
\end{equation*}
Then, coming back to \eqref{eq:dx-rho-u} we obtain
\begin{equation*}
\begin{aligned}
     \int_0^T \!\!\!\int_I \rho^\epsilon(t,x) \partial_t \varphi(t,x)\,dx\,dt
    &\leq C(T) \|u^\epsilon\|_{L^\infty_{t,x}} \|\varphi\|_{L^\infty_{t,x}} +
    \left|\int_0^T \!\!\!\int_I |\rho^\epsilon(t,x)||\partial_x u^\epsilon(t,x)||\varphi (t,x)| \,dx~dt\right|
    \\
    &\leq  C(T) \|u^\epsilon\|_{L^\infty_{t,x}} \|\varphi\|_{L^\infty_{t,x}} + \|\rho^\epsilon\|_{L^\infty_{t,x}}\|\partial_x u^\epsilon\|_{L^2_{t,x}}  \|\varphi\|_{L^2_{t,x}}\\
    &\leq C(T),
\end{aligned}
\end{equation*}
where we used the bounds on $\rho^\ep$ \eqref{rho-bound-nse}and $u^\ep$ \eqref{ueps-H1-NSE}. This allows us to conclude that $\|\rho^\epsilon\|_{BV([0,T]\times I)} \le C(T).$
}
The same argument can be repeated with $\rho^{\star,\epsilon}$, using the transport equation $\partial_t \rho^{\star,\epsilon} = - u^\epsilon \partial_x \rho^{\star,\epsilon}$ instead of the continuity equation. Using the embedding $\mathcal{M}(I) \hookrightarrow H^{-1}(I)$, the bounds \eqref{rho-time-derivs} on the time derivatives follow. 
\end{proof}
\begin{corollary} We have 
\begin{equation} \label{dx-H-1}
       \|\partial_x\rho^\epsilon\|_{L^\infty_t H^{-1}_x} + \|\partial_x \rho^{\star,\epsilon}\|_{L^\infty_t H^{-1}_x} \le C(T).
\end{equation}
\end{corollary}
\begin{proof}
    By definition of $BV(I)$, we have $\partial_x \rho(t,\cdot), \partial_x \rho^{\star,\epsilon}(t,\cdot) \in \mathcal{M}(I)$. The result then follows from the embedding $\mathcal{M}(I) \hookrightarrow H^{-1}(I)$.
\end{proof}

We now combine the energy estimate \eqref{energy-NSE} with the bounds on the distances \eqref{di(t)-NSE} to obtain a uniform estimate for $w^\epsilon$, which is defined through \eqref{def:w-ss} and \eqref{w-defn-nse}.
\begin{proposition} \label{prop:weps-bdd-NSE}
We have 
\begin{equation}
\|w^{\epsilon}\|_{L^{2}_{t,x}}\le C,
\end{equation}
for some $C>0$ independent of $\ep$.
\end{proposition}
\begin{proof}
First note that from the lower bound on $d_i$~\eqref{di(t)-NSE}, we have for any $i=1,...,N$,
\begin{equation} 
    \epsilon( w_i (t))^2 
    = \epsilon \frac{|u_i (t) - u_{i-1} (t)|^2}{(d_i(t))^2} 
    \le  \frac{|u_{i} (t) - u_{i-1} (t)|^2}{d_i(t) c_1(T)}.
\end{equation} 
Therefore, summing up the above inequalities over $i=1,...,N$, integrating in time, and using \eqref{energy-NSE}, we get
\begin{equation} \label{eps-weps-NSE}
    \epsilon \sum_{i=1}^N\int_0^T  | w_i (t)|^2~dt \le C(T),
\end{equation}
which yields the control of the $L^2_{t,x}$ norm on the macroscopic $w^\ep$ (using the fact that  $w^\epsilon$ is affine by part -- see the definition \eqref{w-defn-nse}).
\end{proof}

Next, we estimate the alternative velocity $v^\epsilon$ (see definition \eqref{def-veps} and Figure \ref{fig:veps-NSE}).
The following result confirms that this is indeed a suitable approximation for the velocity. 
\begin{proposition}[Bounds on the velocity $v^\ep$]\label{prop:veps-bound}
We have
\begin{equation} \label{dxv-bound-NSE}
\|v^\epsilon\|_{L^\infty_{t,x}}+\|\partial_{x}v^{\epsilon}\|_{L^{2}_{t,x}} \le C(T),
\end{equation} 
as well as
\begin{equation}\label{v-bounds}
 \|v^{\epsilon} - u^{\epsilon}\|_{L^{2}_{t}L^{\infty}_{x}} 
 \le C(T) \sqrt{\epsilon},
\end{equation} 
and for the initial data
\begin{equation} \label{veps-ID-bound}
     \|v^{\epsilon}(0,\cdot) - u_0\|_{L^{\infty}(I)} 
 \le C \sqrt{\epsilon}.
\end{equation}
\end{proposition} 

\begin{proof}
By definition of $v^\epsilon$, we have $\|v^\epsilon\|_{L^\infty_{t,x}} = \|u^\epsilon\|_{L^\infty_{t,x}}$ and so the boundedness of $v^\epsilon$ follows from \eqref{u-Linf-NSE}. The control on the derivative in \eqref{dxv-bound-NSE} results from the observation that
\[
\partial_{x}v^{\epsilon}(t,x) = \sum_{i=1}^{N-1} w_{i}(t)\mathbf{1}_{(q_{i}+\epsilon, q_{i+1}-\epsilon)}.
\]
and the bound \eqref{eps-weps-NSE} on the components of $w^\ep$.\\
Next, coming back to the definitions of $u^\ep$ \eqref{u-NSE} and $v^\epsilon$ \eqref{def-veps}, we observe that on $[q_{i-1}(t), q_i(t)]$, both $u^\ep$ and $v^\ep$ take values in between $u_{i-1}(t)$ and $u_i(t)$, so that
\begin{equation} \label{u-v}
|u^\ep(t,x) - v^\ep(t,x)| \leq |u_i(t) - u_{i-1}(t)|. 
\end{equation}
Since, moreover, 
     \begin{equation} \begin{aligned}
         \max_i|u_{i}(t)  - u_{i-1}(t) |^2 &\le \sum_{i=1}^{N} |u_i (t)-u_{i-1} (t)|^2 \le \max_i d_i(t) \sum_{i=1}^N \frac{|u_i (t)-u_{i-1} (t)|^2}{d_i(t)},\label{veps-bound-NSE}
         \end{aligned}
     \end{equation} 
we can integrate in time and use the energy estimate \eqref{energy-NSE} combined with the upper bound on $d_i$ (which is independent of $i$) from \eqref{di(t)-NSE} to get 
\begin{equation*}
     \int_0^T \sup_{x \in (q_{i-1}(t), q_i(t))}|v^\epsilon(t,x) - u^\epsilon(t,x)|^2~dt  \leq   \int_0^T  \max_i|u_{i}  - u_{i-1} |^2(t)~dt  \le C(T)\epsilon,
    \end{equation*} 
    as desired. 
    For the initial data, we can substitute $t=0$ into \eqref{u-v} and repeat the argument of \eqref{u0-conv-proof}. This gives
    \begin{equation}
    \begin{aligned}
       \|v^\epsilon(0,\cdot) - u_0\|_{L^\infty(I)} &\le \|u^\epsilon(0,\cdot) - u_0\|_{L^\infty(I)} + \|u^\epsilon(0,\cdot) - v^\epsilon(0,\cdot)\|_{L^\infty((I)}  \\[1ex] &\le C\epsilon^{1/2} + \max_i |u_i(0) - u_{i-1}(0)| \\[1ex] &\le C \epsilon^{1/2}.
    \end{aligned}
    \end{equation}
\end{proof}

In order to pass to the limit in the non-linear convective term in the momentum equation \eqref{NSE-eps-mom}, we need to bound $\partial_t \chi^\epsilon$ and $\partial_t (\chi^\epsilon v^\epsilon)$ where we recall $\chi^\ep$ is defined in \eqref{def-vol}.
\begin{proposition} \label{prop:chi_t-NSE}
We have
$\|\partial_{t}\chi^{\epsilon}\|_{L^\infty_{t}H^{-1}_x} + \|\partial_{t}(\chi^{\epsilon}v^{\epsilon})\|_{L^{2}_{t}H^{-1}_{x}}\le C.$ 
\end{proposition}
\begin{proof} 
Note that in the distributional sense, we have \begin{equation*} \begin{aligned}
        &\partial_{t}\chi^{\epsilon}(t,\cdot)= \sum_{i=1}^{N-1} \left[- u_i(t)\delta_{\{x = q_i -\epsilon\}} + u_i(t) \delta_{\{x=q_i + \epsilon\}} \right], \\[1ex]
        & \partial_{x}\chi^{\epsilon}(t,\cdot)= \sum_{i=1}^{N-1} \left[ \delta_{\{x = q_i -\epsilon\}} - \delta_{\{x=q_i + \epsilon\}} \right].
\end{aligned}
\end{equation*}
Also, since $v^\epsilon(t,x) = u_i(t)$ for $x \in (q_{i}(t)-\epsilon, q_i(t)+\epsilon)$, we have
\begin{equation} \label{vchi}
    v^\epsilon \chi^\epsilon(t,\cdot) = \sum_{i=1}^{N-1} u_i(t) \mathbf{1}_{(q_{i}(t)-\epsilon, q_i(t)+\epsilon)},
\end{equation}
and therefore
    \begin{equation*}
        \partial_{x}(v^\epsilon\chi^{\epsilon})(t,\cdot)= \sum_{i=1}^{N-1} \left[  u_i(t)\delta_{\{x = q_i -\epsilon\}} - u_i(t)\delta_{\{x=q_i + \epsilon\}} \right] = -\partial_t 
        \chi^\epsilon (t,\cdot).
    \end{equation*}
    Therefore for any $\phi \in H^{1}(I)$, using $\chi^\epsilon \le 1$, we have for a.e. $t$:
    \begin{equation*} \begin{aligned}
\langle \partial_{t}\chi^{\epsilon}(t,\cdot), \phi \rangle_{\mathcal{D}'(I) \times \mathcal{D}(I)}
&\le \int_I |\chi^\epsilon v^\epsilon(t,x)| |\partial_x\phi(x)| \,dx\\[1ex] &
            \le\|v^{\epsilon}\|_{L^{2}_{x}} \|\partial_x\phi\|_{L^{2}(I)} \end{aligned}
        \end{equation*}
        thanks to the bounds on $v^\epsilon$ from Proposition \ref{prop:veps-bound}. 
        Taking the $\esssup$ in time, and using that $\|v_\ep \|_{L^\infty_t L^2_x} \leq \sqrt{|I|} \|v_\ep \|_{L^\infty_{t,x}} = \|v_\ep \|_{L^\infty_{t,x}}$, we get the $L^\infty_t H^{-1}_x$ bound claimed in the proposition.  For $\chi^\ep v^\ep$, fix $\phi \in H^1(I)$ and a time $t \in (0,T)$. We get from \eqref{NSE-eps-mom}
    \begin{equation} \begin{aligned}
        |\langle \partial_t (\chi^\epsilon v^\epsilon)(t,\cdot), \phi \rangle| 
        &\le | \langle \chi^\epsilon (v^\epsilon)^2(t,\cdot), \partial_x\phi \rangle| +  | \langle w^\epsilon(t,\cdot) , \partial_x\phi \rangle| + | \langle G^\epsilon(t,\cdot) , \partial_x\phi  \rangle| +  | \langle   \sum_{i=1}^{N-1} \bar f_{i}\mathbf{1}_{P_i(t)} ,  \phi  \rangle| \\[1ex] 
        & \leq\left( \|v^\epsilon(t,\cdot)\|_{L^{\infty}_x}\|v^\epsilon(t,\cdot)\|_{L^{2}_x} + \|w^\epsilon(t,\cdot)\|_{L^{2}_x} + \|G^\epsilon(t,\cdot)\|_{L^{2}_x} +  \| \sum_{i=1}^{N-1} \bar f_{i}\mathbf{1}_{P_i(t)}\|_{L^{2}_{x}}\right)\|\phi\|_{H^{1}(I)}
        \\
        &\le  ( \|v^\epsilon(t,\cdot)\|_{L^{\infty}_x}^2 + \|w^\epsilon(t,\cdot)\|_{L^{2}_x} + \|G^\epsilon(t,\cdot)\|_{L^{\infty}_x} +  \| f(t,\cdot) \|_{L^{\infty}_{x}})\|\phi\|_{H^{1}(I)}. 
        \end{aligned}
    \end{equation} 
   Since $w^\epsilon$ is bounded in $L^2((0,T) \times I)$ (Propositions \ref{prop:weps-bdd-NSE}), while the other terms are bounded in $L^{\infty}((0,T) \times I)$, we deduce that $\partial_t (\chi^\epsilon v^\epsilon)$ is controlled in $L^2_tH^{-1}_x$.
\end{proof}

%
%
\subsection{Limit passage} \label{sec:limit-NSE}
We now use the  uniform bounds obtained in the previous section to derive convergences and identify the limits of each of the terms appearing in the PDE formulation \eqref{NSE-eps-cons}-\eqref{NSE-eps-transport}. 
\subsubsection{Convergences for the convective term}

We start by extracting convergent subsequences for the density $\rho^\epsilon$ and critical density $\rho^{\star, \epsilon}$. 
\begin{proposition} \label{prop:rho-strong-NSE}
    There exists $\rho,\rho^\star$ such that, up to a subsequence, $\rho^\epsilon \to \rho$ strongly in $C([0,T]; L^p(I))$ and $\rho^{\star,\epsilon} \to \rho^\star$ strongly in $C([0,T]; L^p(I))$ for any $p \in [1,\infty)$.
\end{proposition}
\begin{proof}
Firstly, due to the bounds on $\rho^\epsilon$ and $\rho^{\star,\epsilon}$ (see \eqref{rho-bound-nse}-\eqref{rhost-bound-nse}), we have that there exist $\rho \in L^\infty_{t,x}$ with $0 < \rho < 1$ and $\rho^\star \in L^\infty_{t,x}$ with $0<\rho^\star \le 1$  such that, up to a subsequence,
\begin{align} \label{rho-Linf-NSE}
    \rho^\epsilon \rightharpoonup^\star \rho 
 &\text{ weakly-*  in } L^\infty((0,T)\times I), \\[1ex]
  \rho^{\star,\epsilon} \rightharpoonup^\star \rho^\star
 &\text{ weakly-*  in } L^\infty((0,T)\times I).
 \end{align}
Recall that, from Proposition \ref{prop:rho-BV-NSE}, we have the uniform $L^\infty_t BV_x$ estimates for both $\rho^\epsilon$ and $\rho^{\star,\epsilon}$ and uniform $L^2_t H^{-1}_x$  estimates for $\partial_t \rho^\epsilon$ and $\partial_t \rho^{\star,\epsilon}$. 
Therefore, using the Aubin-Lions-Simon lemma  (see  Theorem II.5.16 p.102 \cite{boyer2012}), for $p \in [1, \infty)$, we obtain:

\begin{equation}
    \begin{aligned} \label{rho-strong}
         &\rho^\epsilon \to \rho 
 \text{ strongly  in } C([0,T]; L^p(I)), \\[1ex]
  &\rho^{\star,\epsilon} \to \rho^\star
 \text{ strongly  in } C([0,T]; L^p(I)).
    \end{aligned}
\end{equation}
\end{proof} 
We can also show that $\chi^\epsilon$ shares the same limit as $\rho^\epsilon$.
\begin{proposition} \label{prop:chi-NSE}
    $\chi^{\epsilon} \rightharpoonup^{\star} \rho$  in $L^{\infty}_{t,x}$.
\end{proposition}
\begin{proof} 
     We decompose \begin{equation*}
        \chi^{\epsilon} - \rho = (\chi^{\epsilon} - \rho^{\epsilon}) + (\rho^{\epsilon} - \rho).
    \end{equation*} 
    It is easy to verify that $\int_{q_{i-1}(t)}^{q_{i}(t)} (\chi^\epsilon - \rho^\epsilon)\,dx = 0$ for any $i=1,...,N$. One can combine this with the convergence of $\rho^\epsilon$ from \eqref{rho-strong} to conclude. We refer to the proof of Lemma 3.6 in \cite{lefebvre2008micro} for the details.
\end{proof}
From the control  of the velocity $u^\ep$, there exists $u \in L^2(0,T; H^1_0(I)) \cap L^\infty\big((0,T) \times I\big) $ such that up to a subsequence,
\begin{equation}  \label{ueps-conv-NSE} \begin{aligned}
    &u^\epsilon \rightharpoonup u \text{ weakly in } &&L^2(0,T; H^1_0(I), \\[1ex] &u^\epsilon \rightharpoonup u  \text{ weakly-* in } &&L^\infty\big((0,T) \times I\big).
    \end{aligned}
\end{equation}
Thanks to a compensated compactness theorem, we are now able to pass to the limit in the nonlinear convective terms
\begin{proposition} \label{prop:chiv-conv}
     $\chi^\epsilon v^\epsilon \rightharpoonup \rho u$ and $\chi^\epsilon (v^\epsilon)^2 \rightharpoonup \rho u^2$ weakly in $L^2(0,T; L^\infty(I) )$. 
\end{proposition}
\begin{proof}
    We combine the estimate on $\partial_t \chi^\epsilon$ from Proposition \ref{prop:chi_t-NSE} with the estimate for $\partial_x v^\epsilon$ \eqref{dxv-bound-NSE} and apply a classical compensated compactness argument (e.g. Lemma 5.1. of \cite{mathlions}) to get $\chi^\epsilon v^\epsilon \to \rho u$ in the sense of distributions. Then using the boundedness of $\|\chi^\ep v^\ep\|_{L^2_t L^\infty_x}$ (from \eqref{dxv-bound-NSE} and $|\chi^\epsilon| \le 1$) and the uniqueness of limits in $\mathcal{D}'$, we obtain the specified convergence. For $\chi (v^\epsilon)^2$, the argument is the same but we use the estimate for $\partial_t (\chi^\epsilon v^\epsilon)$ from Proposition \ref{prop:chi_t-NSE} instead of the estimate for $\partial_t \chi^\epsilon $.
\end{proof}

\subsubsection{Convergences for other non-linear terms}
A key difficulty of the limit passage is identifying the limits for the non-linear terms $w^\epsilon$ and $G^\epsilon$. The former follows from the strong convergences of $\rho^\epsilon$ and $\rho^{\star,\epsilon}$ obtained in Proposition \ref{prop:rho-strong-NSE}.
\begin{lemma} \label{lemma-Geps-NSE}
    $G^{\epsilon} \to G$ in $L^{1}_{t,x}$ where $G = (\rho / \rho^{\star})^{\gamma}$ almost everywhere.
\end{lemma}
\begin{proof}
    Defining $\tilde{G}^{\epsilon} = (\rho^{\epsilon} / \rho^{\star, \epsilon})^{\gamma}$, first note this function is bounded due to the bounds on $\rho^\epsilon$ and $\rho^{\star,\epsilon}$ in \eqref{rho-bound-nse} and \eqref{rhost-bound-nse} respectively. Then it is clear from the strong convergences \eqref{rho-strong} that $\tilde{G}^{\epsilon} \to (\rho / \rho^{\star})^{\gamma}$ in $L^1_{t,x}$. We can compare this with the definition of $G^\ep$ from \eqref{G-defn-nse}. Computing the difference $\tilde{G}^{\epsilon} - G^{\epsilon}$
    for a fixed time $t$ on each particle, one can also verify that (see the proof of Proposition \ref{prop:veps-bound} for a similar computation)
    \begin{equation}
        \|G^{\epsilon}(t,\cdot) - \tilde{G}^{\epsilon}(t,\cdot)\|_{L^{\infty}(I)} \le \max_i|G_{i+1}(t)-G_{i}(t)|,
    \end{equation} and so \begin{equation}
        \|G^{\epsilon} - \tilde{G}^{\epsilon}\|_{L^{\infty}_{t,x}} \le 2\epsilon \|\partial_{x}G^{\epsilon}\|_{L^{\infty}_{t,x}} \to 0
    \end{equation} as $\epsilon \to 0$, thanks to the boundedness of $\partial_x G^\epsilon$ shown in Lemma \ref{lem:NSE-dxG}. Of course, this implies convergence in $L^1_{t,x}$. Finally, we combine the above convergences with the decomposition $G^{\epsilon} - (\rho / \rho^{\star})^{\gamma} = (G^{\epsilon} - \tilde{G}^{\epsilon}) + (\tilde{G}^{\epsilon} -  (\rho / \rho^{\star})^{\gamma})$ to obtain the claimed result.
\end{proof}
We now identify the limit of $w^\epsilon$. 
From Proposition~\ref{prop:weps-bdd-NSE}, there exists a limit  $w$ such that, up to a subsequence, 
\begin{equation} \label{weps-conv-NSE}
    w^{\epsilon} \rightharpoonup w \text{ weakly in } L^{2}_{t,x}.
\end{equation} 
and the next lemma allows us to identify this limit in terms of $u$ and $\rho$.
\begin{lemma}  $w = \partial_x u / (1-\rho)$ a.e. in $(0,T) \times I$. \label{lemma:weps-NSE}
\end{lemma}
\begin{proof}
First let us show that $\partial_x u_\ep/(1-\rho_\ep) - w_\ep$ converges strongly in $L^1$ to $0$. 
We recall that by definition $\rho^\ep$ is on constant on $(q_{i-1}, q_i)$ equal to $\rho_i$, as well as $\partial_x u^\ep$ which is equal to $\dfrac{u_i - u_{i-1}}{q_i - q_{i-1}}$. 
Hence $\dfrac{\partial_x u^\ep}{1-\rho^\ep} = \dfrac{u_i - u_{i-1}}{d_i} = w_i$, and $\dfrac{\partial_x u^\ep}{1-\rho^\ep}$ coincides with $w^\ep$ on the intervals $(q_{i-1} + \ep , q_i - \ep)$.
To demonstrate the convergence to 0, it remains then to estimate the difference between the two quantities on the rest of the interval $(q_{i-1},q_i)$, namely on $(q_{i-1}, q_{i-1} + \ep)$ and $(q_i-\ep, q_i)$. 
On the interval  $(q_{i-1}, q_{i-1} + \ep)$, we have
\begin{align*}
\left(\dfrac{\partial_x u^\ep}{1-\rho^\ep} - w^\ep\right)\mathbf{1}_{(q_{i-1}, q_{i-1} + \ep)}
& = w_i - \left(w_{i-1} + \dfrac{x - (q_{i-1} -\ep)}{2\ep} (w_i - w_{i-1})\right)\\
& = \left(1 -  \dfrac{x - (q_{i-1} -\ep)}{2\ep}\right) (w_i - w_{i-1})
\end{align*}
while on $(q_i-\ep, q_i)$:
\begin{align*}
\left(\dfrac{\partial_x u^\ep}{1-\rho^\ep} - w^\ep\right)\mathbf{1}_{(q_i-\ep, q_i)}
& = w_i - \left(w_{i} + \dfrac{x - (q_{i} -\ep)}{2\ep} (w_{i+1} - w_i)\right)\\
& = -\dfrac{x - (q_{i} -\ep)}{2\ep} (w_{i+1} - w_i).
\end{align*}
As a consequence, we have
\begin{align*}
\int_0^T \!\!\!\int_I \left|\dfrac{\partial_x u^\ep}{1-\rho^\ep} - w^\ep\right| dx dt
& \leq \int_0^T \sum_{i=1}^{N} \int_{q_{i-1}}^{q_i} \left|\dfrac{\partial_x u^\ep}{1-\rho^\ep} - w^\ep\right| dx dt \\
& \leq \int_0^T \sum_{i=1}^{N} \int_{q_{i-1}}^{q_{i-1} + \ep } \left|\dfrac{\partial_x u^\ep}{1-\rho^\ep} - w^\ep\right| dx dt 
    +  \int_0^T \sum_{i=1}^{N} \int_{q_{i}-\ep}^{q_{i}} \left|\dfrac{\partial_x u^\ep}{1-\rho^\ep} - w^\ep\right| dx dt \\
& \leq 2 \int_0^T \ep \sum_{i} |w_{i+1} - w_i| \\
& \leq 4 \int_0^T \ep \sum_{i} |w_i|.
\end{align*}
Now, since the $d_i$'s remain of order $\ep$ (cf. proposition~\ref{prop:global-nse}), we have
\[
\ep |w_i| \leq C \sqrt{\ep} \dfrac{u_i - u_{i-1}}{\sqrt{d_i}},
\]
and therefore we get by the discrete energy estimate~\eqref{energy-NSE} and a Cauchy-Schwarz inequality
\[
\int_0^T \ep \sum_{i} |w_i| 
\leq C \sqrt{\ep} (\mathcal{E}_0 + \|f\|_{L^2_t L^1_x}D_N(0))^{1/2}.
\]
Hence $\dfrac{\partial_x u^\ep}{1-\rho^\ep}-w_\ep$ converges to $0$ in $L^1_{t,x}$. 
Moreover, $\partial_x u^\ep$ converges weakly in $L^2_{t,x}$ to $\partial_x u$, $\rho^\ep$ converges strongly (in $L^2_{t,x}$) to $\rho <1$ and $w_\ep$ converges weakly in $L^2$ to $w$, so that $\dfrac{\partial_x u^\ep}{1-\rho^\ep}-w_\ep$ converges to $\dfrac{\partial_x u}{1-\rho}-w$ in the sense of distribution. By uniqueness of the limit , we conclude that $w = \dfrac{\partial_x u}{1-\rho}$ a.e.

\end{proof}
 Lastly, we show the convergence of the forcing term.
 \begin{proposition} \label{prop:feps-nse}
 We have the following convergence:
 $\sum_{i=1}^{N-1} \bar f_{i} \mathbf{1}_{P_i(t)} \rightharpoonup^{\star} \rho f
 \text{ in } L^{\infty}((0,T) \times I)$.
 \end{proposition}
\begin{proof}
Let $f^\ep := \sum_{i=1}^{N-1} \bar f_{i} \mathbf{1}_{P_i(t)}.$ We start with the decomposition 
\begin{equation} \label{f-decomp}
    f^\epsilon - \rho f = (f^\epsilon - \chi^\ep f) + (\chi^\ep f - \rho f).
\end{equation}
For the first term, we use the definition of $\chi^\ep$ from \eqref{def-vol} to write
\begin{equation}
    \begin{aligned}
        (f^\epsilon - \chi^\ep f)(t,x) &= \sum_{i=0}^N \mathbf{1}_{P_i(t)}(x) \left[ \frac{1}{2\ep} \int_{P_i(t)} f(t,y)~dy - f(t,x) \right]
 \\[1ex] &= \sum_{i=0}^N \mathbf{1}_{P_i(t)}(x) \left[ \frac{1}{2\ep} \int_{P_i(t)}( f(t,y) - f(t,x))~dy \right].     \end{aligned}
\end{equation}
Using $f \in W^{1,\infty}((0,T)\times I)$, we have 
\begin{equation}
      |(f^\epsilon - \chi^\ep f)(t,x)|\le 2\ep \sum_{i=0}^N \mathbf{1}_{P_i(t)}(x).
\end{equation}
Therefore, we have $(f^\epsilon - \chi^\ep f) \to 0$ strongly in $L^\infty_{t,x}$. For the second term of \eqref{f-decomp}, Proposition \ref{prop:chi-NSE} tells us that $\chi^\ep \rightharpoonup^\star \rho$ weakly-* in $L^\infty_{t,x}$. This concludes the proof. 
\end{proof}

\subsubsection{Limit passage in the weak formulation}
We have shown the following convergences for the density $\rho^\epsilon$, the critical density $\rho^{\star,\epsilon}$ and velocity $u^\epsilon$:
 \begin{align*}
    &\rho^{\epsilon} \to \rho \text{ strongly in } C([0,T]; L^{p}(I)), \hspace{0.1cm} &&\text{ (Prop. \ref{prop:rho-strong-NSE})} \\[1ex]
    &\rho^{\star,\epsilon} \to \rho^\star \text{ strongly in } C([0,T]; L^{p}(I)), \hspace{.1cm} &&\text{ (Prop. \ref{prop:rho-strong-NSE})} \\[1ex]
    &u^{\epsilon} \rightharpoonup   u \text{ weakly in } L^2(0,T; H^{1}_0(I)). \hspace{.1cm} &&\text{ (Cor. \ref{u-bounds-NSE})}
\end{align*}
In Proposition \ref{prop:NSE-eqn-eps-copy}, we showed that the continuity equation \eqref{NSE-eps-cons} and the transport equation \eqref{NSE-eps-transport} hold in the sense of distributions for any fixed $\epsilon$. More precisely, for any $\epsilon > 0$ the weak formulations \eqref{WFcontinuity}-\eqref{WFtransport} hold.

Using the above convergences as well as the convergences \eqref{nse-conv-assumption} and \eqref{u0-NSE} for the initial data $(\rho_0^\epsilon, u_0^\epsilon, \rho_0^{\star,\epsilon})$, we can pass to the limit in each term. The strong convergence of the densities is used to do so in the initial/terminal conditions that appear in the continuity/transport equations. This allows us to conclude that $(\rho,\rho^\star)$ solve the continuity and transport equations respectively with velocity $u$, in the sense of \eqref{WFcontinuity} and \eqref{WFtransport}. 

We now look at the weak formulation of the momentum equation \eqref{NSE-eps-mom}. For any $\epsilon > 0$ and $\phi \in C^1_c([0,T)\times I)$, we have
\begin{equation} \begin{aligned} \label{wf-nse}
    &\int_{0}^{T} \!\!\!\int_{I} \chi^{\epsilon}v^{\epsilon} \partial_t\phi(t,x)\,dx\,dt+\int_{I}\chi^{\epsilon}v^{\epsilon}\phi(0,x) \,dx \\[1ex]  &+\int_0^T \!\!\!\int_I \chi^{\epsilon}(v^{\epsilon})^{2}\phi(
    t,x)\,dx\,dt -\int_0^T \!\!\!\int_I w^{\epsilon}\partial_{x}\phi(s,x)\,dx\,dt +\int_0^T \!\!\!\int_I G^{\epsilon}\partial_{x}\phi(t,x)\,dx\,dt\\[1ex] &= -\int_{0}^{T}\!\!\! \int_{I} \sum_{i=1}^{N-1} \overline{f_i} \mathbf{1}_{P_i(t)}\phi(t,x)\,dx\,dt.
    \end{aligned}
\end{equation}
We have previously shown that \begin{align*}
    &\chi^{\epsilon} \rightharpoonup^\star \rho \text{ weakly-* in } L^{\infty}(0,T; L^{\infty}(I)), \hspace{.1cm} &&\text{ (Prop \ref{prop:chi-NSE}) } \\[1ex]
    &\chi^{\epsilon}v^{\epsilon} \to \rho u \text{ weakly in } L^{2}(0,T; L^{\infty}(I)), \hspace{.1cm} &&\text{ (Prop \ref{prop:chiv-conv}) } \\[1ex]
    &\chi^{\epsilon}(v^{\epsilon})^{2} \to \rho u^2 \text{ weakly in } L^{2}(0,T; L^{\infty}(I)).  \hspace{.1cm} &&\text{ (Prop \ref{prop:chiv-conv}) }
    \end{align*}
This allows us to pass to the limit in the first and third terms of \eqref{wf-nse}. For the remaining terms, we recall the weak convergence of $w^\epsilon$ from \eqref{weps-conv-NSE}, the $L^1_{t,x}$ convergence of $G^\epsilon$ from Lemma \ref{lemma-Geps-NSE} and the convergence of the forcing term from Proposition \ref{prop:feps-nse}. Additionally, we have the convergence of the initial data \eqref{nse-conv-assumption} and \eqref{u0-NSE}. We can also deduce the convergence $v^\epsilon(0,\cdot) \to u_0$ in $C(I)$ from \eqref{veps-ID-bound}. 
This is enough to pass to the limit in \eqref{wf-nse}. 

To conclude the proof, let us verify that $\rho, \rho^\star$ and $\rho u$ belong to the correct functional spaces. The strong convergences (and boundedness) of $\rho^\epsilon$ and $\rho^{\star,\epsilon}$ from Proposition \ref{prop:rho-strong-NSE} imply that $\rho, \rho^{\star}  \in C([0,T]; L^\infty(I))$.

Finally, let us prove that $\rho u \in C_{weak}([0,T); L^2(I)).$ Take $\phi \in H^1(I)$ and let $F_\ep(t) := \int_I \chi^\ep v^\ep(t,x) \phi(x)\,dx$. Note that $\|F_\ep\|_{L^\infty(0,T)} \le C$ uniformly. Using the momentum equation \eqref{NSE-eps-mom},
\begin{equation}
    \begin{aligned}
        \frac{d}{dt} F_\ep(t) =  \int_I \left[ \chi (v^\ep)^2 - w^\ep + G^\ep \right] \phi'\,dx +   \int_I f^\ep \phi ~ dx 
    \end{aligned}
\end{equation} is uniformly bounded in $L^2(0,T)$, since $\chi^\ep, v^\ep, G^\ep$ are uniformly bounded and $\|w^\ep \|_{L^2_{t,x}} \le C$. Therefore $\|F_\ep\|_{H^1(0,T)} \le C$ uniformly and so there exists $F \in H^1(0,T)$ such that up to a subsequence, $F^\ep \to F$ in $C([0,T])$. On the other hand, we have $F^\ep \rightharpoonup \int_I \rho u \phi\,dx$ in $L^2(0,T)$ since $\chi^\ep v^\ep \rightharpoonup \rho u $ in $L^2_t L^\infty_x$ (Proposition \ref{prop:chi-NSE}). Therefore, we have
\begin{equation}
F(t) = \int_I \rho u(t,x) \phi(x)\,dx \in H^1(0,T) \subset C([0,T]),
\end{equation}
which gives \begin{equation}
        \rho u \in C_{weak}([0,T); L^2(I)).
    \end{equation}
This completes the proof of Theorem \ref{thm:nse}.

\section{Numerical illustrations} 
\label{sec:numerics}
The goal of this section is to provide a series of numerical illustrations to help us visualise the behaviour of solutions to the limiting system 
 \begin{equation}
\label{sim}
    \left\{
    \begin{aligned}
        &\partial_{t} \rho + \partial_{x}(\rho u ) = 0,  \\[1ex]
       &\partial_t(\rho u) + \partial_x (\rho u^2) - \partial_{x} \left( \frac{\mu}{1- \rho} \partial_{x} u \right) 
       + \partial_{x} \left( \frac{\rho}{\rho^{\star}} \right)^{\gamma} =  \rho f, \\[1ex] &\partial_{t} \rho^{\star} + u \partial_{x} \rho^{\star} = 0, 
    \end{aligned}
    \right.
\end{equation} 
where $\mu >0, \gamma \ge 1, f$ is a given source. 
We impose the no-slip boundary conditions for $u$, i.e $u=0$ for $x=0,1$. This system can be interpreted as a compressible Navier-Stokes system with a singular viscosity and a pressure that becomes stiff in the limit $\gamma \to \infty$. 
The numerical discretization of these equations is challenging due to the singular behaviour of the viscosity as $\rho$ approaches its maximal value $\rho = 1$. 
On the one hand, the blow-up of the viscosity induces severe stability constraints, leading to a restrictive diffusive CFL condition. On the other hand, numerical errors may cause the density to exceed the physical bound $\rho = 1$, which must be carefully prevented at the discrete level.
In~\cite{HCL}, the diffusion term was treated implicitly in time and the system was reformulated in terms of an effective velocity. In the present work, we propose to
carry out numerical experiments using neural networks, more specifically ``Physics-Informed Neural Networks'' (PINNs). The experiments in this section are not intended to provide an optimal solver for the limit system, but rather to illustrate its qualitative properties and to demonstrate a solver that can be naturally extended past 1D. In the remainder of this section we include only the final results, while the details of the implementation are postponed to Appendix~\ref{sec:AppB}.

We consider three cases of initial data. We take $\gamma=1$ (except for the final experiment which compares values of $\gamma$) and $f=0$ in each of our experiments.
\subsection{Case 1: constant, equal initial densities.}
We take equal, constant initial densities $\rho_0,\rho_0^\star$ and a compressive velocity: \begin{equation}
    (\rho_0, u_0, \rho_0^\star) = (0.7, 0.5\sin(2\pi x), 0.7).
\end{equation}
The compressive velocity causes the density to grow, while the singular viscosity -- arising due to the lubrication force at the microscopic level -- imposes $\rho<1$. At the same time, the pressure term $\partial_x\left( \rho/ \rho^\star \right)^\gamma$  -- arising due to the repulsive force at the microscopic level --  becomes stronger because $\rho$ becomes much greater than $\rho^\star$. Eventually, this forces the density to decrease (see $t=0.2$ to $t=0.3$ in Figure \ref{fig:case1}). Note that $\rho^\star(t,\cdot)=\rho_0^\star<1$ for positive times since it is transported. Thus the key effect of the pressure in this example is that it restricts the growth of the density. To see this more clearly, we demonstrate later on in Figure \ref{fig:CNS} the same case when the pressure term (and $\rho^\star$) are removed. In the long-time regime (see the last column of Figure \ref{fig:case1}), the density returns to the steady state $\rho=\rho^\star=0.7$. 
\begin{figure}[H]
    \centering
    \includegraphics[width=0.9\linewidth]{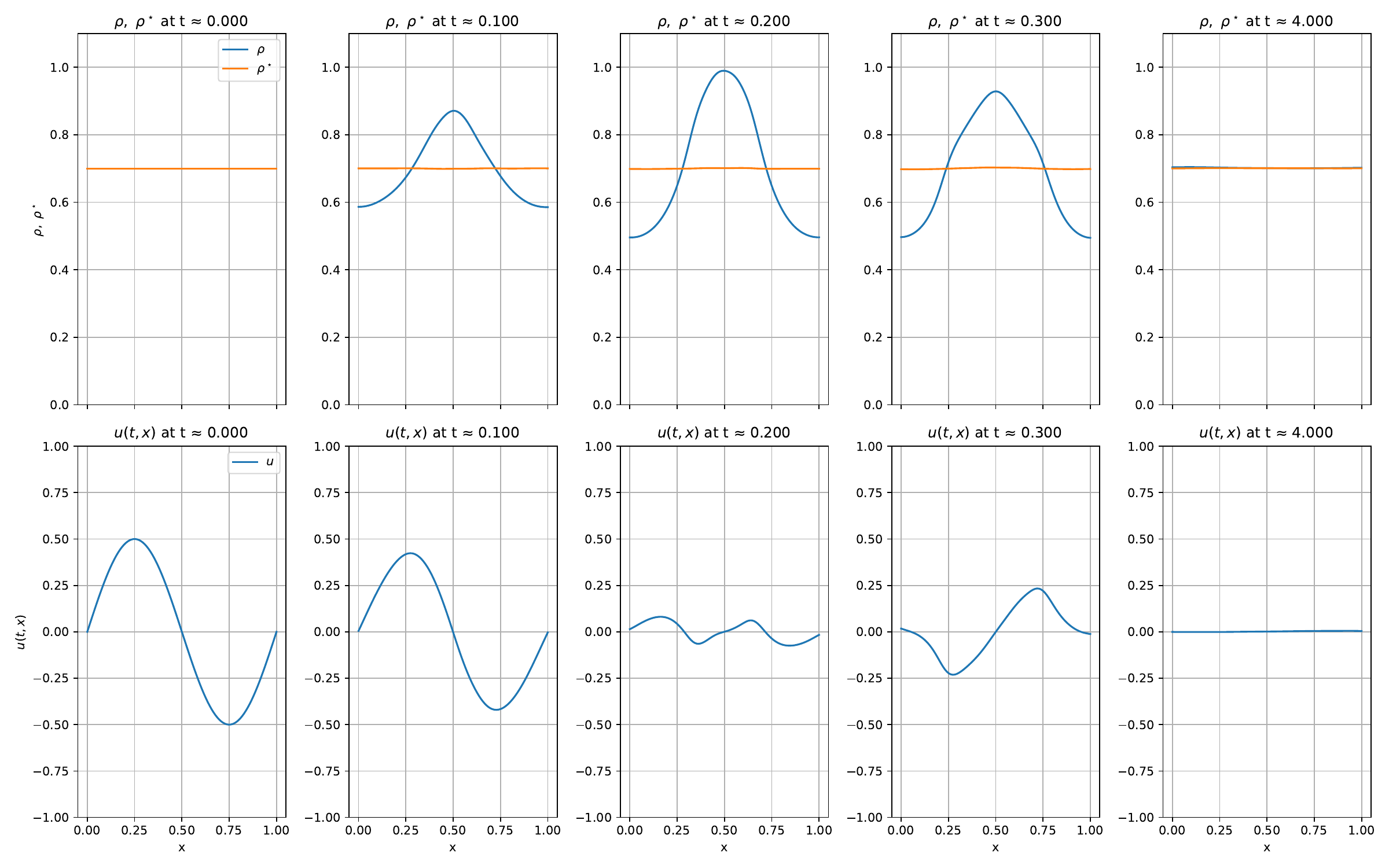}
    \caption{Case 1: the blue curve represents $\rho$ and the orange curve represents $\rho^\star$.}
    \label{fig:case1}
\end{figure}

\subsection{Case 2: congestion formation}

\subsubsection*{Case 2a: without pressure}
We first carry out an experiment for our system without the presence of $\rho^\star$ and its transport equation. In this case, the system is reduced to (setting the forcing term to $0$):
 \begin{equation}
    \left\{
    \begin{aligned}
       &\partial_t\rho + \partial_x(\rho u) = 0, \\[1ex]
&\partial_t(\rho u) + \partial_x( \rho u^2) - \partial_x \left(\frac{1}{1-\rho} \partial_x u \right) = 0.
    \end{aligned}
    \right.
\end{equation} 
This is a compressible pressureless Navier-Stokes system, which is equivalent to the dissipative Aw-Rascle system studied in \cite{HCL}. There, the authors  carried out numerical simulations for the system for initial data 
\begin{equation}
    (\rho_0, u_0) = (0.7, 0.5\sin(2\pi x)).
\end{equation}
Again, the singular viscosity prevents the density from reaching 1. 
Compared to Case 1, since the pressure term, which plays a repulsive role, has been removed, we expect high-density regions to form more easily and to persist over long times.
In order to validate our neural network solver and to additionally understand the effects of the pressure term in more detail, we display the solution obtained by our PINN for this case of initial data in Figure \ref{fig:CNS}.
\begin{figure}[H]
    \centering
    \includegraphics[width=0.9\linewidth]{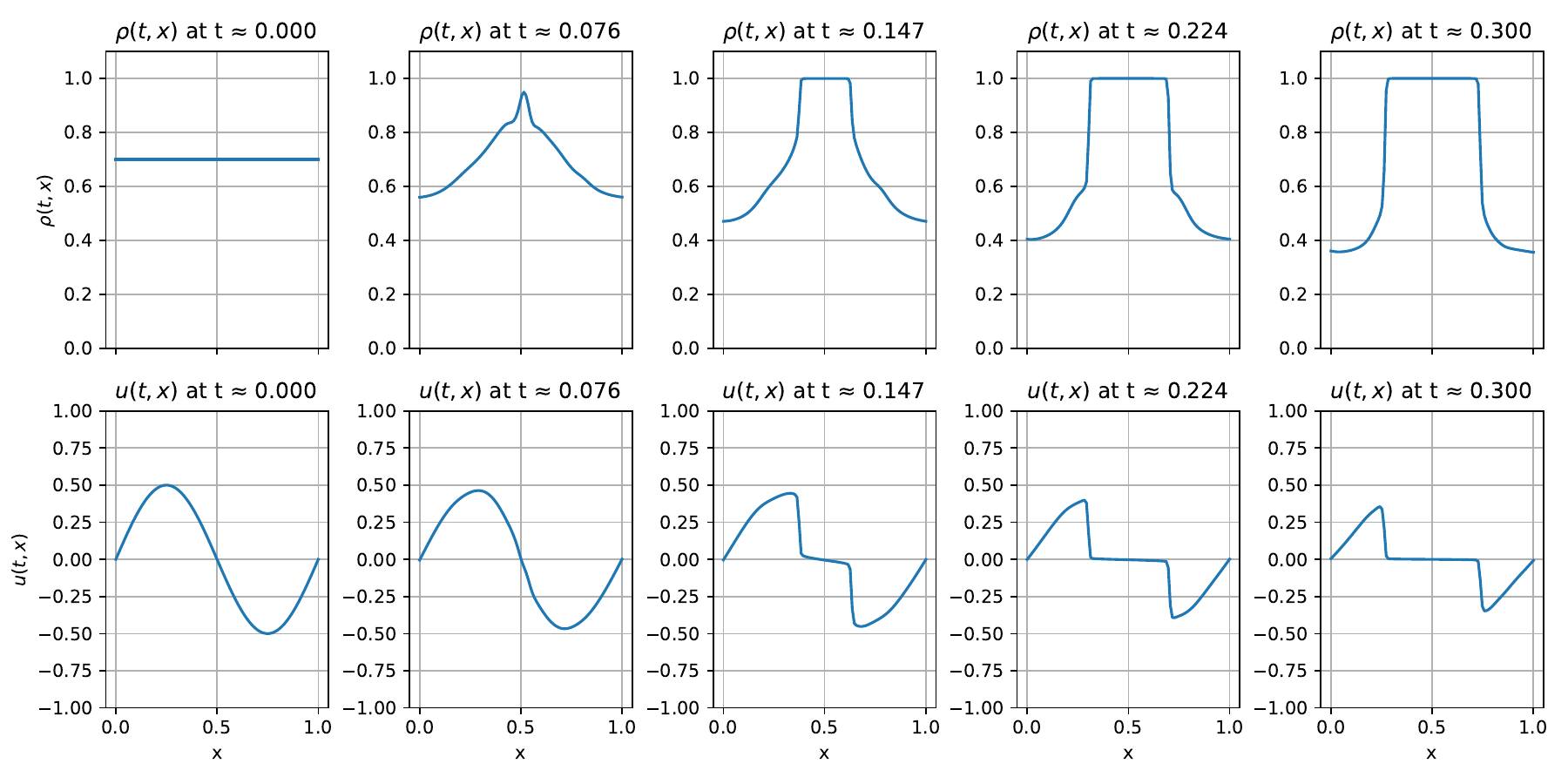}
    \caption{Case 2a: Congestion formation in the macroscopic system without the pressure.}
    \label{fig:CNS}
\end{figure}
As expected, we see in Figure~\ref{fig:CNS} that once a region becomes congested (i.e. $\rho$ close to 1), it remains congested thereafter. 
Note that the results line up with the simulations seen in \cite{HCL}, where a more traditional finite-difference numerical scheme was used.
\subsubsection*{Case 2b: with pressure}
We now increase the initial value of $\rho^\star_0$ to $1.0$, in order to highlight more clearly the effect of the pressure term $\partial_x (\rho /\rho^\star)^\gamma$ in system \eqref{sim}. We take the initial data
\begin{equation}
    (\rho_0, u_0, \rho_0^\star) = (0.7, 0.5\sin(2\pi x), 1.0).
\end{equation}
The results can be seen in Figure \ref{fig:case2}.  In this case, the pressure term is smaller than in Case 1, which results in the formation of a congestion state, as observed at $t=0.2$ in Figure~\ref{fig:case2}. This congestion is due to the lubrication force, which prevents the density from reaching 1. Unlike Case 2a, due to the pressure term, the density subsequently decreases for $t>0.2$. However, this decrease occurs more slowly than in Case 1, as the ratio $\rho / \rho^\star$ is smaller.
\begin{figure}[H]
    \centering
    \includegraphics[width=0.9\linewidth]{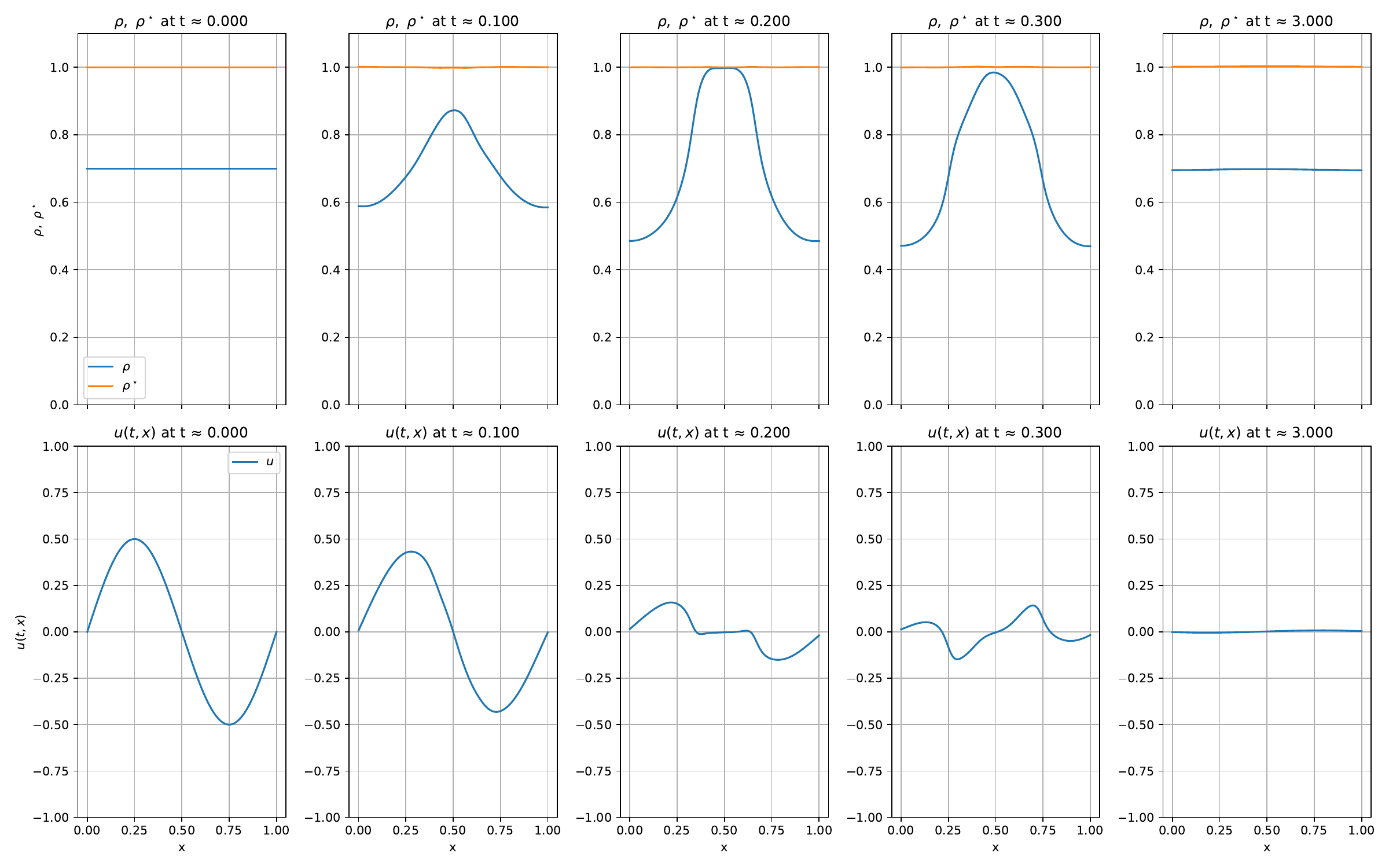}
    \caption{Case 2b: with pressure.}
    \label{fig:case2}
\end{figure}

\subsection{\texorpdfstring{Case 3: non-constant constraint $\rho^\star$}{Case 3: non-constant constraint rho*}}
We now consider initial data where the densities are not constant. More specifically, we take the initial densities to be Gaussian pulses such that $\rho_0$ is above $\rho_0^\star$ in some region of the domain, and $u_0$ is again compressive:
\begin{equation}
     (\rho_{0}, u_{0}, \rho_{0}^\star) = \left(0.6+0.2 \exp \left( \frac{(x-0.5)^2}{2(0.1)^2}\right), 0.5\sin(2\pi x), 0.6-0.2 \exp \left( \frac{(x-0.5)^2}{2(0.1)^2}\right)\right).
\end{equation}
We see that the density is divided into two peaks which travel apart due to the velocity. An interesting observation here is that in the long-time regime, we obtain states where $\rho > \rho^\star$, for which the lubrication and pressure effects compensate.
\begin{figure}[H]
    \centering
    \includegraphics[width=\linewidth]{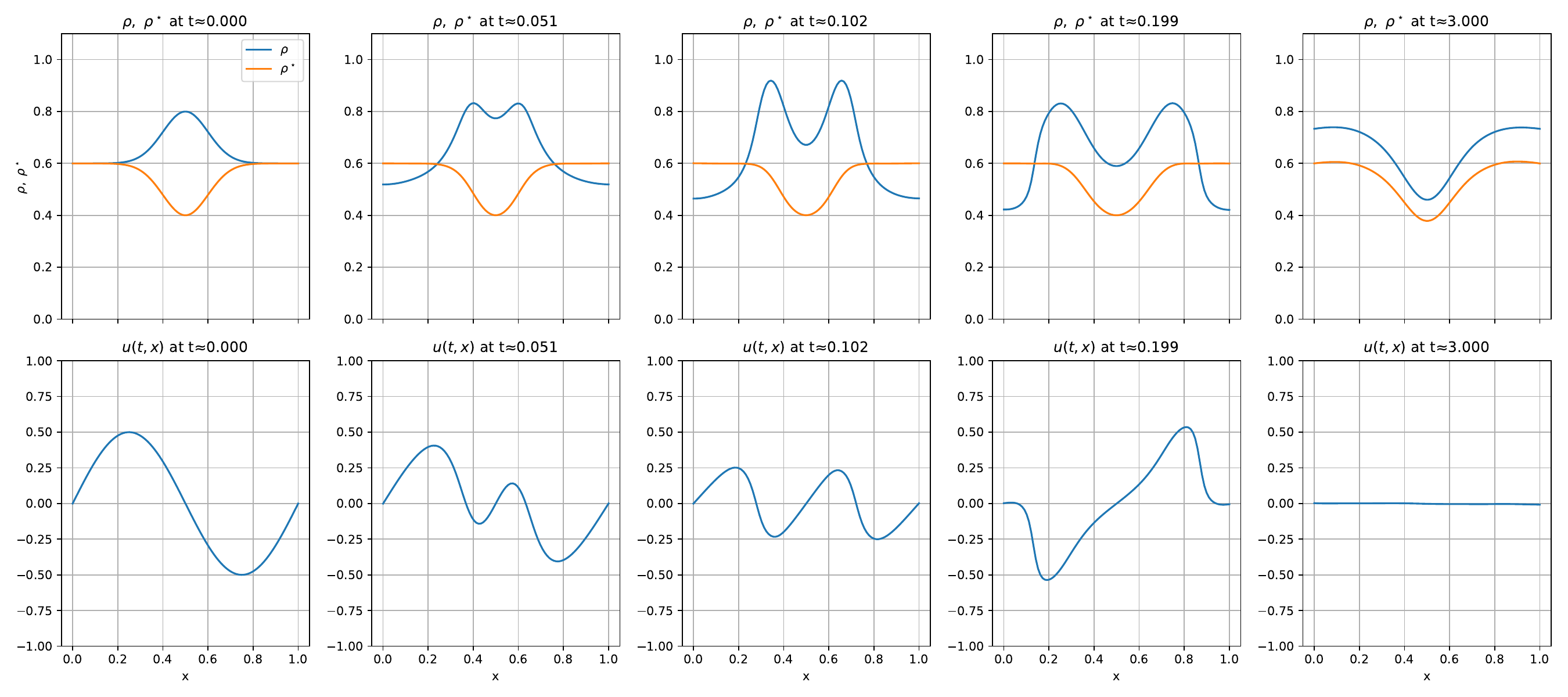}
    \caption{Case 3: the blue curve represents $\rho$ and the orange curve represents $\rho^\star$.}
    \label{fig:case3}
\end{figure}

\subsection{\texorpdfstring{Effects of variable $\gamma$}{Effects of  gamma}}

Finally, we carry out a small experiment using the initial data of Case 1 with varying values of $\gamma=2,5,10$ to better understand its effects on the dynamics. We see in Figure \ref{fig:case1-gamma} that the bigger $\gamma$ is, the stronger the force imposing the expected maximal density becomes. The maximum value of $\rho$ decreases as $\gamma$ increases. This is due to the fact that the penalisation of $\rho > \rho^\star$ becomes stronger as $\gamma$ increases, since the pressure term is exponential in $\gamma$. At the limit $\gamma \to \infty$, we expect to converge to a hard constraint $\rho \le \rho^\star$. 

\begin{figure}[H]
    \centering
    \includegraphics[width=1\linewidth]{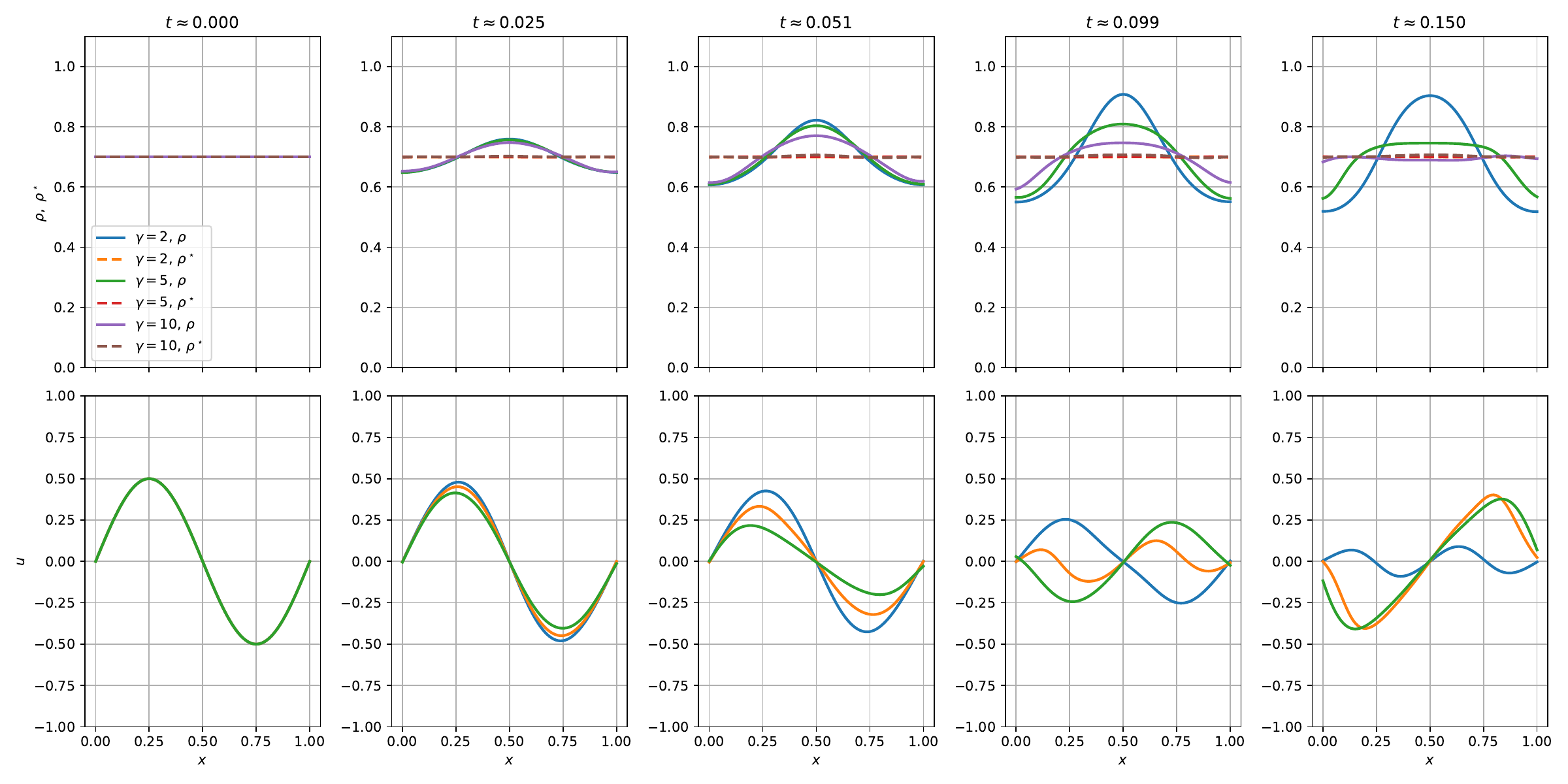}
    \caption{The effect of varying $\gamma$ for the data of Case 1. For the velocity, the blue curve corresponds to $\gamma=2$, the orange curve to $\gamma=5$ and the green curve to $\gamma=10$.}
    \label{fig:case1-gamma}
\end{figure}

\begin{appendices}	
    \section{Extension to a congested model} \label{sec:extension} 
In this section we explain how Theorem \ref{thm:nse} can be extended to the case where the limit density may reach $1$, i.e. $\overline{\rho_0} = 1$ on a subset of $[0,1]$. At time $t=0$ we now allow for a configuration $(\mathbf{q}_{0}^{\epsilon}, \mathbf{u}_0^\epsilon, \mathbf{d}^{\star, \epsilon})_{\epsilon}$ where two or more particles may be in contact creating a cluster. By a cluster we understand cluster to any group of particles that are stuck together at initial time. If a cluster exists  initially then the particles that created it move stuck together with constant velocity for all positive times. Therefore, we may essentially treat a cluster  of $k-$particles as a single particle of radius $k\epsilon$. A possible configuration is depicted a the Figure \ref{fig:u-cluster}.

 \begin{figure}[H]
	\centering
	\includegraphics[scale= 1]{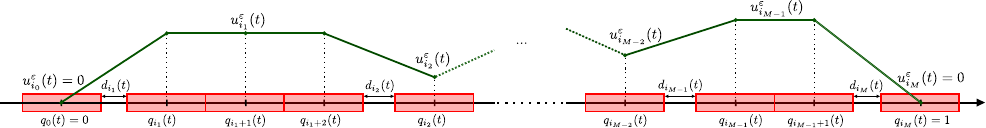}
	\caption{A particle configuration with clusters}
	\label{fig:u-cluster}
\end{figure}
 Following \cite{lefebvre2008micro}, to describe the balance of forces in this situation, we introduce some new notation. 
 
 For any $\epsilon$ we denote $M^{\epsilon} +1$ to be the number of clusters (note that a single particle not in contact with any others is also a cluster), where $M^\epsilon \le N^\epsilon$ (equality holds when no two particles are in contact). We denote by $i_{1}, ..., i_{M^{\epsilon}}$ the first particle in each cluster. We additionally let $n_{k} +1 \ge 1$ be the number of particles in cluster $i_{k}$, so that cluster $k$ consists of the particles $i_{k}, i_{k}+1, ..., i_{k}+ n_{k}$. The distance between clusters $k-1$ and $k$ is $d_{i_{k}}$, so the balance of forces for cluster $k$ is given by
    \begin{equation} \label{bof-clusters-ext}
       2n_k\epsilon \ddot{q}_{i_k} = \frac{u_{i_{k+1}} - u_{i_{k}+n_{k}}}{d_{i_{k+1}}}   - \frac{u_{i_{k}} - u_{i_{k-1}}}{d_{i_{k}}} + (G_{i_{k}} - G_{i_{k+1}})    + 2 n_k \epsilon\bar f_{i_{k}},
    \end{equation}
where \begin{equation}
    u_{i_{k}} = u_{i_{k}+1} = ... = u_{i_{k}+n_k}, 
\end{equation} and \begin{equation}
    \bar{ f }_{i_{k}}  = \frac{1}{2n_k\epsilon} \int_{q_{i_{k}} - \epsilon}^{q_{i_{k}+n_k}+ \epsilon} f(s,x)\,dx.
\end{equation}

 With this analogue of the balance of forces \eqref{balance-intro}, we may basically repeat the construction of the approximate initial data, proof of local and global existence of solutions from Section \ref{sec:global-NSE}, as well as the uniform estimates and passage to the limit in the PDE formulation from Section \ref{sec:limit} with three minor changes.
    \begin{itemize} 
        \item The distance $d_i$ between particles $q_{i-1}$ and $q_i$ is replaced by $d_{i_k}$, the distance between cluster $i_k-1$ and cluster $i_k$. However, because the clusters are neither created nor destroyed in time, the estimates for the lower and upper bounds of $d_{i_k}$ follow the same way as in Proposition \ref{prop:global-nse}.
        \item Since the limit density $\rho$ might now be equal to $1$, identification of the nonlinear viscosity term $w=\frac{\partial_x u}{1-\rho}$, needs to be revised. Indeed, because division by $(1-\rho)$ is not allowed, one first needs to identify $\partial_x u$ with $(1-\rho)w $. This is slightly more subtle as $w^\ep$ and $\rho^\ep$ are defined on ``different grids''. Similar identification has been performed in  Lemma 4.9 in \cite{lefebvre2008micro}.
    \item Finally, the weak formulation of the momentum equation from Definition \ref{def:nse-soln} needs to be re-interpreted since the function $1/ (1-\rho^{\epsilon})$ is no longer well-defined in the pointwise sense.  To this end, we define the space
\begin{equation} \label{def:Hw-space}
    H_{\omega} := \left\{  v \in H^{1}_{0}(I) : \partial_{x}v = 0 \text{ in } D(\omega)^{c} \text{ and } \int_{D(\omega)}\omega(x)|\partial_{x}v|^{2} < +\infty        \right\},
\end{equation} 
where $D(\omega) = \{ x \in I : \omega(x) < +\infty \}$. 

We can now define solutions as a triple of functions $(\rho,u,\rho^*)$ such that $\rho,\rho^*\in (0,1]$ and $u\in L^2(0,T; H_{\frac{1}{1-\rho}}(I))$ such that the momentum equation:
 \begin{equation} \label{WFmom-ext} \begin{aligned}
    &\int_0^T \!\!\!\int_I \rho  u \partial_t \phi(t,x)\,dx\,dt  + \int_I  \rho_0 u_0(x)\phi(0,x)\,dx + \int_0^T \!\!\!\int_I \rho u^2 \partial_x \phi (t,x)\,dx\,dt  \\[1ex] 
    &-  \int_0^T\!\!\!\int_{I} \frac{\mu}{1-\rho} \partial_{x}u ~\partial_{x}\phi (t,x)\,dx\,dt + \int_0^T\!\!\!\int_{I} \left(\frac{\rho}{\rho^{\star}}  \right)^{\gamma} \partial_{x}\phi (t,x)\,dx\,dt = -\int_0^T\!\!\!\int_{I}\rho f \phi (t,x)\,dx\,dt, \end{aligned}
    \end{equation}
 holds for all $t \in (0,T)$ and $\phi \in C^1_c([0,T); H_{\frac{1}{1-\rho}})$; all the rest of Definition \ref{def:nse-soln}
 remains unchanged.        
    \end{itemize}

 \begin{remark} The requirement that $u, \phi \in H_{\frac{1}{1-\rho}}$ in space ensures that the viscous term of \eqref{WFmom-ext} is bounded. Indeed, using Young's inequality,
\begin{equation*}
    \left|  \int_{I} \frac{\mu}{1-\rho} \partial_{x}u ~\partial_{x}\phi\,dx  \right| \le \frac{\mu}{2} \int_{I} \frac{|\partial_{x}u|^{2}}{1-\rho} \,dx + \frac{\mu}{2} \int_{I} \frac{|\partial_{x}\phi|^{2}}{1-\rho} \,dx,
\end{equation*} 
which is bounded since in space we have $u, \phi \in H_{\frac{1}{1-\rho}}$.
\end{remark}

\section{Details of numerical illustrations}
\label{sec:AppB}
In this section we give an overview of Physics-Informed Neural Networks and provide full details of the neural network solver used to generate the simulations in Section \ref{sec:numerics}.

\subsection{Physics-Informed Neural Networks} 
Deep learning methods represent a new approach for the numerical analysis of PDEs. The most popular paradigm in this area is that of ‘Physics-Informed Neural Networks' (PINNs), which was introduced by Raissi et al in \cite{raissi2019physics}, as a mesh-free solver for PDEs. Physics-Informed Neural Networks may be defined as fully-connected neural networks with a physics-informed loss function. 
\subsubsection*{Fully-connected neural networks}
 A fully-connected neural network is a function
whose output is computed as a composition of affine layers and non-linear activation functions. We now give a precise definition of a fully-connected neural network.
Let $\mathbf{x} \in \mathbb{R}^d$ be the input, $g^{(0)}(\bf{x}) = \bf{x}$ and $d_0:=d$. Suppose $\mathbf{v}_\theta$ is a neural network with $L$ layers and widths $(d_1, ..., d_L)$. We define $\mathbf{v}_\theta$ recursively using each layer $l$ of the network. The output of layer $l$ is given by
\begin{equation}
    \mathbf{v}_\theta^{(l)}(\mathbf{x}) = \mathbf{W}^{(l)} \cdot \mathbf{g}^{(l-1)}(\mathbf{x})+\mathbf{b}^{(l)}, ~~\mathbf{g}^{(l)}(\mathbf{x})= \sigma ( \mathbf{v}_\theta^{(l)}(\mathbf{x})), ~~ l =1,2,...,L.
\end{equation}
Then the neural network $\mathbf{v}_\theta$ is defined as the output of the final layer: 
\begin{equation}
    \mathbf{v}_\theta(\mathbf{x}) = \mathbf{W}^{(L+1)} \cdot \mathbf{g}^{(L)}(\mathbf{x}) + \mathbf{b}^{(L+1)},
\end{equation}
where $\mathbf{W}^{(l)} \in \mathbb{R}^{d_{l} \times d_{l -1} }, \mathbf{b}^{(l)} \in \mathbb{R}^{d_l}$ are the weight matrix / bias vector respectively in the $l$-th layer, and $\sigma$ is an activation function. The full set of trainable parameters is given by \begin{equation*}
    \theta = (\mathbf{W}^{(1)}, \mathbf{b}^{(1)}, ..., \mathbf{W}^{(L+1)}, \mathbf{b}^{(L+1)}).
\end{equation*}
When defining a neural network in practice, we have control over the depth of the network ($L$ above), the width of the network ($d_l$) and the activation function $\sigma$ (which could be chosen independently in each layer). 
\subsubsection*{Training PINNs}
    When defining a neural network in practice, the parameters $\theta$ need to be initialised. A default approach is to randomly sample the weights and biases from a fixed distribution (usually uniform or Normal). For PINNs, more sophisticated methods such as Xavier / He initialisation \cite{glorot2010understanding, he2015delving} are commonly used in the literature.  

After initialisation, we are left with an arbitrary $\mathbf{v}_\theta$. In order for our neural network $\mathbf{v}_\theta$ to approximate a solution $v$ of a PDE $\mathcal{L}[v]=0$ well, the parameters $\theta$ must be chosen appropriately. With PINNs, this is done by defining a loss functional
\begin{equation}
    \mathbf{J}[\mathbf{v}_\theta] := \|\mathcal{L}[\mathbf{v}_\theta] \|_{L^2(\Omega)}^2 + \|\mathcal{B}[\mathbf{v}_\theta]\|^2_{L^2(\partial \Omega)},
\end{equation}
where the operator $\mathcal{B}$ encodes the boundary/initial conditions. In this sense, we say the neural network is ``physics-informed''. In practice, these norms are approximated using a finite set of collocation points. We randomly generate points $\{y_i\}_{i=1}^{M_p}$ in the interior domain $(0,T) \times I$, as well as points $\{ z_i\}_{i=1}^{M_b}$ on the slices $\{t=0\} \times I$ and $(0,T) \times \partial I$ on which the initial data / boundary conditions are defined, respectively.  Then the empirical loss functional is given by
\begin{equation}
    \mathcal{J}[\mathbf{v}_\theta] = \frac{1}{M_p} \sum_{i=1}^{M_p} |\mathcal{L}[\mathbf{v}_\theta](y_i)|^2 + \frac{1}{M_b} \sum_{i=1}^{M_b} |\mathcal{B}[\mathbf{v}_\theta](z_i)|^2.
\end{equation}
The training phase then proceeds as follows. With the initialised network, we perform a forward pass to evaluate $\mathbf{v}_\theta$ at the collocation points. The derivatives of the output $\mathbf{v}_\theta$ are computed  precisely using the automatic differentiation functionalities of the Tensorflow/PyTorch packages for Python. This is used to evaluate the loss function $\mathcal{J}[\mathbf{v}_\theta]$. 

We then update the parameters $\theta$ in an effort to minimise the loss $\mathcal{J}[\mathbf{v}_\theta]$. This is essentially a finite-dimensional non-linear optimisation problem:
\begin{equation}
    \theta^\ast \in \arg \min_{ \theta \in \mathbb{R}^p} \mathcal{J}[\mathbf{v}_\theta].
\end{equation}
We typically employ a gradient-based optimisation algorithm, such as stochastic gradient descent or a variant (e.g. Adams) to update the parameters. Stochastic gradient descent works by randomly sampling points in $[0,T] \times I$, computing the gradient of the loss with respect to $\theta$ at those points and using the iterative update rule:
\begin{equation}
    \theta_{t+1} = \theta_t - \eta \cdot \nabla_{\theta_t} \mathcal{J}[\mathbf{v}_\theta].
\end{equation}
Here, $\eta$ represents the learning rate which controls the step-size of each update. The gradient $\nabla_\theta \mathcal{J}[\mathbf{v}_\theta]$ is computed via backpropagation, which works by applying the chain rule sequentially layer-by-layer to find the gradient of the loss with respect to each of the trainable parameters. Once the parameters $\theta$ are updated,  one epoch (iteration) of the training cycle is complete. The training stage typically consists of thousands of epochs. The end result is a neural network $\mathbf{v}_\theta$ which approximates $\mathbf{v}$.

\subsection{Application of PINNs to the limit system}
We now describe how we use PINNs to carry out numerical experiments for the limit system. We define a neural network $\mathbf{v}_\theta : (t,x) \to (\rho_\theta(t,x), u_\theta(t,x), \rho^\star_\theta(t,x))$. The inputs are the time-space coordinates $\mathbf{x}:=(t,x) \in [0,T] \times [0,1]$, and so the input dimension of the network is $2$. The output dimension is $3$, corresponding to the solution $(\rho, u, \rho^\star)$. The neural network $\mathbf{v}_\theta$ we consider is chosen to be a fully-connected neural network with $L$ hidden layers of equal width $m$. This means that $\mathbf{v}_\theta$ has the general form
\begin{equation}
    \mathbf{v}_\theta(\mathbf{x}) = \mathbf{W}^{(L+1)} \cdot \mathbf{g}^{(L)}(\mathbf{x}) + \mathbf{b}^{(L+1)}, \qquad \mathbf{x} = (t,x)
\end{equation}
where for $\ell \in \{1,2,...,L\}$, 
\begin{equation}
    \mathbf{g}^{(\ell)}(\mathbf{x}) = \sigma \left( \mathbf{W}^{(\ell-1)}\mathbf{g}^{(\ell-1)}(\mathbf{x}) + \mathbf{b}^{(\ell-1)} \right), \qquad \mathbf{g}^0(\mathbf{x}) = \mathbf{x}.
\end{equation}
Here, $(\mathbf{W}^{(\ell)}, \mathbf{b}^{(\ell)})$ corresponds to the weight matrix / bias vector in the $\ell$'th layer; $\mathbf{W}^{(1)} \in \mathbb{R}^{m\times 2}, \mathbf{b}^{(1)} \in \mathbb{R}^{m \times 1}$, $\mathbf{W}^{(\ell)} \in \mathbb{R}^{m\times m}$ and $\mathbf{b}^{(\ell)} \in \mathbb{R}^{m\times1}$ for $\ell = 2,..., L$. For the output layer, $\mathbf{W}^{(L+1)} \in \mathbb{R}^{3 \times m}, \mathbf{b}^{(L+1)} \in \mathbb{R}^{3 \times 1}$.  Furthermore, $\sigma$ is a fixed non-linear activation function applied component-wise. Generally, common choices for activation functions include $\text{ReLU}(x) := \max(0,x)$, sigmoid and $\tanh$ among many others. For PINNs, $\tanh$ is often used in the literature. The effect of different types of activation functions on the performance of PINNs is an ongoing area of research (see \cite{maczuga2023influence, khademi2025physics, jagtap2020adaptive} for recent works in this area). The set of trainable parameters of the network is given by \[ \theta := (\mathbf{W}^{(1)}, \mathbf{b}^{(1)}, ..., \mathbf{W}^{(L+1)}, \mathbf{b}^{(L)}). \]  Initialising a neural network $\mathbf{v}_\theta$ can be viewed as defining a finite-dimensional trial space 
\begin{equation}
    \mathcal{V}_P = \{ (t,x) \mapsto \mathbf{v}_\theta(t,x) : \theta \in \mathbb{R}^P \},
\end{equation}
where $P$ is the total number of trainable parameters. The strategy of PINNs is to select $\theta$ so that the resulting function $\mathbf{v}_\theta$ approximately satisfies the governing PDE, initial and boundary conditions in a least-squares sense. This is the purpose of the training phase.

\subsubsection*{Network architecture}
We choose $L=4$ hidden layers of equal width $m=256$, and the activation function $\sigma(z) = \tanh(z)$. The parameters $\theta$ of the network are initialised using the Xavier (Glorot) initialisation scheme \cite{glorot2010understanding}. This scheme samples weights from a uniform distribution, which is scaled appropriately to ensure variance of activations/gradients are roughly constant across layers. This prevents vanishing/exploding gradient issues from arising during training, which can lead to unstable parameter updates during the optimisation stage.  

The output of our neural network is somewhat non-standard. We attach a final transformation onto the network output which ensures the initial data are always satisfied. The raw network output is denoted by $(d_\rho, d_u, d_{\rho^\star})$. This is combined with an extra transformation to produce the final output:
\begin{equation*}
    \begin{aligned}
        &\rho_\theta(t,x) = \rho_0(x) \exp\left(\frac{t}{T} d_\rho(t,x) \right), \\[1ex]
        &u_\theta(t,x) = u_0(x) + \frac{t}{T}d_u(t,x), \\[1ex]
        &\rho^\star_\theta(t,x) = \rho_0^\star(x) \exp(A(t) d_{\rho^\star}(t,x)).
    \end{aligned}
\end{equation*}
This choice is primarily made to hard-code the initial data into the network. Indeed, we have $(\rho_\theta, u_\theta, \rho^\star_\theta)(0,\cdot) = (\rho_0, u_0, \rho^\star_0)$ by construction. The use of the exponential in $\rho_\theta$ and $\rho_\theta^\star$ ensures the positivity of these functions, while it is omitted for $u$ since $u$ can be negative.

\subsubsection*{Physics-informed loss functional}
 After initialisation, the next step is to sample a set of collocation points, perform a forward pass and evaluate the loss function. We uniformly sample points $\{t_p^i, x_p^i\}_{i=1}^{M_{p}}$ in the interior domain $(0,T) \times I$ as well as $\{t_{bc}^i\}_{i=1}^{M_b}$ on the slice $(0,T) \times \{0,1\}$ on which the boundary conditions are defined. We choose $M_p=7500$ and $M_b=2000$. Once these points have been sampled, we evaluate the network at each of these points and compute the loss function. The loss function we use is given by
\begin{equation}
    \begin{aligned}
        \mathcal{J}[\mathbf{v}_\theta] :=  \mathcal{J}_{\text{PDE}}[\mathbf{v}_\theta] +  \mathcal{J}_{\text{BC}}[\mathbf{v}_\theta].
    \end{aligned}
\end{equation} The PDE residual loss is given by
\begin{equation}
    \begin{aligned}
        \mathcal{J}_{\text{PDE}}[\mathbf{v}_\theta] &:= \frac{\lambda_{P_1}}{M_p} \sum_{i=1}^{M_p} | \partial_t \rho_\theta + \partial_x(\rho_\theta u_\theta)|^2 (t_p^i, x_p^i) +  \frac{\lambda_{P_2}}{M_p} \sum_{i=1}^{M_p} | \partial_t \rho^\star_\theta + u_\theta \partial_x \rho^\star_\theta |^2 (t_p^i, x_p^i)\\[1ex]
 &+   \frac{\lambda_{P_3}}{M_p} \sum_{i=1}^{M_p} | \partial_t (\rho_\theta u_\theta) + \partial_x(\rho_\theta u_\theta^2) - \mu \partial_x ( \partial_x u_\theta / (1-\rho_\theta)) - \rho_\theta f|^2 (t_p^i, x_p^i),
 \end{aligned}
\end{equation} 
where the weights $\lambda_{P_{i}}$ are fixed scalars. We choose $\lambda_{P_1} = \lambda_{P_2} = 1$ and $\lambda_{P_3} = 2$. We find that, in practice, a higher weight on the transport equation for $\rho^\star$ helps the network to learn the correct dynamics. The boundary condition loss is given by
 \begin{equation}
     \begin{aligned}
                 \mathcal{J}_{\text{BC}}[\mathbf{v}_\theta] &:=  \frac{1}{M_b} \sum_{i=1}^{M_b} |u_\theta(t_{bc}^i, 1) - u_\theta(t_{bc}^i, 0) |^2.
     \end{aligned}
 \end{equation}
In classical PINN architectures (e.g. \cite{raissi2019physics}), one usually includes a term for the initial condition loss evaluated on sampled points $\{x_{ic}^i\}_{i=1}^{M} \subset \{t=0\} \times [0,1]$, given by
 \begin{equation}
     \begin{aligned}
                 \mathcal{J}_{\text{IC}}[\mathbf{v}_\theta] &:=  \frac{1}{M} \sum_{i=1}^M |\rho_\theta(0, x_{ic}^i) - \rho_0(x_{ic}^i) |^2 + \frac{1}{M} \sum_{i=1}^M |u_\theta(0, x_{ic}^i) - u_0(x_{ic}^i) |^2 \\[1ex]
        &+\frac{1}{M} \sum_{i=1}^M |\rho^\star_\theta(0, x_{ic}^i) - \rho^\star_0(x_{ic}^i) |^2.
     \end{aligned}
 \end{equation}
Since our initial condition is hard-coded into the output, we do not incorporate this into the loss $\mathcal{J}$. Let us now note that for cases 1 and 2 of Section \ref{sec:numerics}, a constant initial $\rho^\star_0$ was used. This implies that the solution $\rho^\star$ will be equal to $\rho_0^\star$ for positive times, due to the transport equation for  $\rho^\star$. Therefore, to encourage the network to learn the correct solution, we penalise variations in $\rho^\star$ for cases 1 and 2 by adding the following regularisation term to the loss function:
\begin{equation}
    \mathcal{J}_{\text{PEN}}[\mathbf{v}_\theta] := \frac{1}{M_p} \sum_{i=1}^{M_p} | \partial_t \rho_\theta |^2 (t_p^i, x_p^i) + \frac{1}{M_p} \sum_{i=1}^{M_p} | \partial_x \rho_\theta |^2 (t_p^i, x_p^i).
\end{equation}
In order to evaluate the loss function and compute gradients with respect to the parameters of the network, the automatic differentiation functionality of PyTorch is used. The parameters are then iteratively updated using an optimisation method. We train the network for $7500$ epochs. We use the Adams optimiser \cite{KingmaB14} for the first $5500$ epochs and finish using the L-BFGS optimiser \cite{nocedal1980updating,nocedal2006numerical} for the last $2000$ epochs. 

\subsubsection*{Numerical evaluation}
Once training is completed, the learned approximation $(t,x) \mapsto (\rho_\theta(t,x), u_\theta(t,x), \rho^\star_\theta(t,x))$ is evaluated on a fixed, deterministic time-space grid that is distinct from the collocation points used during training. We sample $N_x=100$ points in time and $N_t=100$ points in space, and produce a grid of $N_t \times N_x$ points covering the domain $[0,T] \times [0,1]$. The trained network is evaluated pointwise on this grid to produce arrays $[\rho_\theta(t_i, x_j)], [u_\theta(t_i,x_j)], [\rho_\theta^\star(t_i, x_j)]$ which are then used to produce the plots. 
\end{appendices}

\printbibliography
\end{document}